\documentclass[reqno, 11pt]{amsart}
\usepackage{mathtools}
\usepackage{amsmath}
\usepackage{amssymb}
\usepackage{yhmath}
\usepackage{graphicx}
\usepackage{mathrsfs}
\usepackage{bbm}
\usepackage{xcolor}
\usepackage{tikz-cd}
\usepackage{tikz}
\usetikzlibrary{patterns}
\usepackage{hyperref}

\DeclareMathAlphabet{\mathpzc}{OT1}{pzc}{m}{it}

\usepackage{thmtools}
\usepackage{thm-restate}

\usepackage{caption}

\newtheorem{theorem}{Theorem}[section]

\newtheorem*{claim*}{Claim}

\newtheorem{lemma}[theorem]{Lemma}
\newtheorem{lem}[theorem]{Lemma}
\newtheorem{corollary}[theorem]{Corollary}

\newtheorem{cor}[theorem]{Corollary}

\newtheorem{proposition}[theorem]{Proposition}

\newtheorem{prop}[theorem]{Proposition}

\theoremstyle{definition}
\newtheorem{definition}[theorem]{Definition}
\newtheorem{Def}[theorem]{Definition}
\newtheorem{example}[theorem]{Example}

\theoremstyle{remark}

\newtheorem{rmk}[theorem]{Remark}
\newtheorem{Rmk}[theorem]{Remark}

\numberwithin{equation}{section}

\newcommand{\op}{\operatorname}

\newcommand{\be}{\begin{equation}}
\newcommand{\ee}{\end{equation}}
\newcommand{\Ga}{\Gamma}

\newcommand{\R}{\mathbb R}

\newcommand{\Z}{\mathbb Z}
\newcommand{\N}{\mathbb N}
\newcommand{\ga}{\gamma}

\newcommand{\la}{\lambda}
\newcommand{\La}{\Lambda}
\newcommand{\inte}{\op{int}}
\newcommand{\ba}{\backslash}

\newcommand{\cal}{\mathcal}
\newcommand{\br}{\mathbb R}
\newcommand{\SO}{\op{SO}}

\newcommand{\PSL}{\op{PSL}}
\newcommand{\F}{\cal F}

\newcommand{\bH}{\mathbb H}

\newcommand{\G}{\Gamma}

\newcommand{\T}{\op{T}}
\renewcommand{\frak}{\mathfrak}

\newcommand{\e}{\varepsilon}

\renewcommand{\L}{\mathcal L}
\newcommand{\fa}{\mathfrak a}

\renewcommand{\i}{\op{i}}

\renewcommand{\S}{\mathbb S}

\newcommand{\so}{\SO^\circ}

\newcommand{\C}{\cal C}

\newcommand{\fg}{\frak g}

\newcommand{\Leb}{\op{Leb}}
\newcommand{\supp}{\op{supp}}
\newcommand{\lat}{\La_{\theta}}
\newcommand{\ft}{\F_\theta} 

\newcommand{\sa}{\mathsf A}

\newcommand{\ts}{\theta}

\renewcommand{\epsilon}{\e}
\renewcommand{\t}{\theta}
\begin{document}

\title[Conformal measures]{Properly discontinuous actions, growth indicators, and Conformal measures for transverse subgroups}

\author{Dongryul M. Kim}
\address{Department of Mathematics, Yale University, New Haven, CT 06511}
\email{dongryul.kim@yale.edu}

\author{Hee Oh}
\address{Department of Mathematics, Yale University, New Haven, CT 06511 and Korea Institute for Advanced Study, Seoul}
\email{hee.oh@yale.edu}

\thanks{
 Oh is partially supported by the NSF grant No. DMS-2450703.}

\author{Yahui Wang}
\address{Department of Mathematics, Yale University, New Haven, CT 06511}
\email{amy.wang.yw735@yale.edu}

\begin{abstract} Let $G$ be a connected semisimple real algebraic group.  The class of transverse subgroups of $G$ includes all discrete subgroups of rank one Lie groups and any subgroups of Anosov or relative Anosov subgroups. 
Given a transverse subgroup $\Gamma$,  we show that the $\Ga$-action 
on the Weyl chamber flow space determined by its limit set is properly discontinuous. This allows us to consider the quotient space and define Bowen-Margulis-Sullivan measures. We then establish the ergodic dichotomy for the Weyl chamber flow, in the original spirit of Hopf-Tsuji-Sullivan.
We also introduce  the notion of growth indicators and  discuss their properties and roles in the study of conformal measures, extending the work of Quint.
We discuss several applications as well.
\end{abstract}

\maketitle

\tableofcontents
\section{Introduction}
Patterson-Sullivan theory on conformal measures of a discrete subgroup of a rank one simple real algebraic group $G$ has played a pivotal role  in the study of dynamics on rank one homogeneous spaces. One of the basic results due to Sullivan in 1979  is the relation between the support of a conformal measure and its dimension, which we recall for $G=\so(n,1)$, the identity component of the special orthogonal group $\op{SO}(n,1)$.
 The group $\so(n,1)$ is the group of orientation-preserving isometries of the real hyperbolic space $(\bH^n,d)$.
 The geometric boundary of $\bH^n$ can be identified with the sphere $\S^{n-1}$. For a discrete subgroup $\Ga<G$, denote by $\La^{\mathsf{con}}\subset \S^{n-1}$ the conical set of $\Ga$, which consists of the endpoints of all geodesic rays in $\bH^n$ which accumulate modulo $\Gamma$.
Let $\delta_\Ga$ denote the critical exponent of $\Ga$, which is the abscissa of convergence of the Poincar\'e series $s\mapsto \sum_{\ga\in \Ga} e^{-s d(o, \ga o)}$, $o\in \bH^n$.

For a given $\Ga$-conformal measure $\nu$, we denote by $\mathsf m_\nu$ the 
Bowen-Margulis-Sullivan measure  on the unit tangent bundle $\T^1(\Ga\ba \bH^n)$, which is a
locally finite measure invariant under the geodesic flow.
The following theorem is 
often referred to as the Hopf-Tsuji-Sullivan dichotomy (see \cite{Tsuji1959potential}, \cite{hopf1971}, \cite{Sullivan1979density}, \cite{Aaronson1984rational}, \cite[Theorem 1.7]{Roblin2003ergodicite}).

\begin{theorem}[Sullivan, {\cite[Corollaries 4, 20, Theorem 21]{Sullivan1979density}}, see also \cite{Aaronson1984rational}, \cite{Corlette1999limit}, \cite{Roblin2003ergodicite}]\label{sul}
Let $\Ga<\so(n,1)$, $n\ge 2$, be a non-elementary discrete subgroup.
Suppose that there exists a $\Ga$-conformal measure $\nu$ on $\S^{n-1}$ of dimension $s\ge 0$.
\begin{enumerate}
    \item We have $$s\ge \delta_\Ga.$$
\item  The following are equivalent:
\begin{enumerate}
\item 
$\sum_{\ga\in \Ga} e^{-s d(o, \ga o)} =\infty$ \hspace{1em} (resp. $\sum_{\ga\in \Ga} e^{-s d(o, \ga o)} <\infty$);
    \item  $\nu(\La^{\mathsf{con}})=1$ \hspace{4.8em} (resp. $\nu(\La^{\mathsf{con}})=0$);
    \item  the geodesic flow on $(\T^1(\Ga\ba \bH^n),\mathsf m_\nu)$ is completely conservative and ergodic.
    
    (resp. the geodesic flow  on $(\T^1(\Ga\ba \bH^n),\mathsf m_\nu)$ is completely dissipative and non-ergodic.)
\end{enumerate}
In the former case, $s=\delta_\Ga$ and $\nu$ is the unique
$\Ga$-conformal measure of dimension $\delta_\Ga$.
\end{enumerate}
\end{theorem}

The main aim of this paper is to establish an analogous result for a class of discrete subgroups of a general connected semisimple real algebraic group $G$, called $\theta$-transverse subgroups. The class of $\theta$-transverse subgroups includes all discrete subgroups of rank one Lie groups, $\theta$-Anosov subgroups and their relative versions. This class is 
  regarded as a generalization of all rank one discrete subgroups while  Anosov subgroups are regarded as higher rank analogues of convex cocompact subgroups.
  
We need to introduce some notations to state our results precisely. Let $P<G$ be a minimal parabolic subgroup with a fixed Langlands decomposition $P=MAN$ where $A$ is a maximal real split torus of $G$, $M$ is the maximal compact subgroup of $P$ commuting with $A$ and $N$ is the unipotent radical of $P$.
Let $\fg$ and $\fa$ respectively denote the Lie algebra of $G$
and $A$. Fix a positive closed Weyl chamber $\fa^+<\fa$ so that
$\log N$ consists of positive root subspaces and
set $A^+=\exp \fa^+$. We fix a maximal compact subgroup $K< G$ such that the Cartan decomposition $G=K A^+ K$ holds. We denote by $\mu : G \to \fa^+$ the Cartan projection defined by the condition $g\in K\exp \mu(g) K$ for $g \in G$. 
Let $\Pi$ denote the set of all simple roots for $(\frak g, \frak a^+)$. As usual, the Weyl group is the quotient of the normalizer of $A$ in $K$ by the centralizer of $A$ in $K$. Let 
$\i:\fa\to \fa$ denote the opposition involution, that is,
$\i(u)=-\text{Ad}_{w_0}(u)$ for all $u\in \fa$ where $w_0$ is the longest Weyl element. It induces an involution on $\Pi$ which we denote by the same notation $\i$.
Throughout the introduction, we fix a non-empty subset $$\theta\subset \Pi.$$
Let $\mathfrak{a}_\theta =\bigcap_{\alpha \in \Pi - \theta} \ker \alpha$ and let $p_{\theta} : \fa \to \fa_{\theta}$ be the unique projection, invariant under all Weyl elements fixing $\fa_{\theta}$ pointwise.
Let $P_\theta$ be the standard parabolic subgroup corresponding to $\theta$ (our convention is that $P=P_\Pi$) and consider the $\theta$-boundary: $$\F_\theta=G/P_\theta.$$
 We say that $\xi\in \F_\theta$ and $\eta\in \F_{\i(\theta)}$ are in general position if the pair
$(\xi,\eta)$ belongs to the unique open $G$-orbit in $\F_\theta\times \F_{\i(\theta)}$ under the diagonal action of $G$.

 Let $\Ga<G$ be a discrete subgroup.  The following properties of $\Ga$ are natural to consider in
 studying analogues of Theorem \ref{sul} for $\Ga$-conformal measures on the $\theta$-boundary  $\F_\theta$. Let $\La_\theta=\La_\theta(\Ga)$ denote the $\theta$-limit set of $\Ga $ in $\F_{\theta}$ (Definition \ref{tli}).

\begin{Def} 
     A discrete subgroup $\Ga$ is said to be {\it $\theta$-transverse} if
\begin{itemize}
    \item   $\Ga$ is {\it $\theta$-regular}, i.e.,
$ \liminf_{\ga\in \Ga} \alpha(\mu({\ga}))=\infty $ for all $\alpha\in \theta$; and
\item $\Ga$ is  {\it $\theta$-antipodal}, i.e.,
if any two distinct $\xi, \eta\in \La_{\theta\cup \i(\theta)}$ 
are in general position.   \end{itemize}    
A $\theta$-transverse subgroup $\Ga$ is called {\it non-elementary} if $\# \La_{\theta} \ge 3$.
\end{Def}

 Note that the  $\theta$-transverse property is hereditary:
 a subgroup of a $\theta$-transverse subgroup is also $\theta$-transverse.

 We assume that $\Ga$ is $\theta$-transverse in the rest of the introduction.
We define
the $\theta$-growth indicator $\psi_\Ga^{\theta}:\fa_\theta\to [-\infty, \infty] $  as follows: fixing any norm $\|\cdot\|$ on $\fa_\theta$, if $u \in \fa_\theta$ is non-zero,
\be \psi_\Ga^{\theta}(u)=\|u\| \inf_{u\in \cal C}
\tau^\theta_{\mathcal C}\ee 
where $ \tau^{\theta}_{\cal C}$ is the abscissa of convergence of the series $\sum_{\ga\in \Ga, \mu_\theta(\ga)\in \mathcal C} e^{-s\|\mu_\theta(\ga)\|}$ and $\cal C\subset \fa_\theta$ ranges over all open cones containing $u$. Set $\psi_{\Ga}^{\theta}(0) = 0$. This definition is independent of the choice of a norm on $\fa_\theta$.
For $\theta=\Pi$, $\psi_\Ga^{\Pi}$ coincides with Quint's growth indicator  $\psi_\Ga$ \cite{Quint2002divergence}. For a general $\theta\subset \Pi$, we have:
\be \label{qq}
    \psi_{\Ga}^{\theta}\circ p_\theta \ge  \psi_{\Ga}.
    \ee
     (Lemma \ref{qu}, see also Lemma \ref{qu2}
     for a precise relation for $G$ simple).
We show that $\psi_\Ga^\theta<\infty$, and $\psi_\Ga^\theta$
is a homogeneous, upper semi-continuous and concave function. It also follows from \eqref{qq} that
\be
\{\psi^\theta_\Ga\ge 0\}=\L_\theta\quad\text{and} \quad
\psi^\theta_\Ga>0 \;\; \text{ on $\inte\L_\theta$}
\ee where $\L_\theta=\L_\theta(\Ga)$ is the $\theta$-limit cone of $\Ga$ (Theorem \ref{three}).

Denote by $\fa_\theta^*=\op{Hom}(\fa_\theta, \br)$ the space of all linear forms on $\fa_\theta$.
For $\psi \in \fa_\theta^*$, a Borel probability measure $\nu$ on $\mathcal{F}_\theta$ is called a $(\Gamma, \psi)$-conformal measure if $$\frac{d \gamma_*\nu}{d\nu}(\xi)=e^{\psi(\beta_\xi^\theta(e,\gamma))} \quad \text{for all $\gamma \in \Gamma$ and $ \xi \in \mathcal{F}_\theta$} $$
 where ${\ga}_* \nu(D) = \nu(\ga^{-1}D)$ for any Borel subset $D\subset \F_\theta$ and $\beta_\xi^\theta$ denotes the $\fa_\theta$-valued Busemann map defined in \eqref{Bu}. We find it convenient to call the linear form $\psi$ the {\it dimension} of $\nu$.

For a collection $\{E_n: n\in \N\}$ of subsets of a given metric space $\cal X$, its topological limsup is the set of accumulation points of all sequences $\{x_n \in E_n :  n \in \N\}$, and is denoted by $\limsup_{n} E_n$.
We define the $\theta$-conical set of $\Ga$ as 
\be \La_\theta^{\mathsf{con}}=\left\{gP_\theta \in \F_\theta : \limsup_{\ga \in \Ga} \ga g M_{\theta} A^+ \ne \emptyset \right\},\ee
where $M_\theta=K\cap P_\theta$  (see Lemma \ref{cons} for an equivalent definition). If $\Ga$ is $\theta$-regular, then $\lat^{\mathsf{con}}\subset \lat$  (Proposition \ref{reg}).

 \begin{Def}
    We say $\psi\in \fa_\theta^*$ is $(\Ga, \theta)$-proper
if  $\psi\circ \mu_\theta:\Ga\to [-\e, \infty)$ 
is a proper map for some $\epsilon>0$.  \end{Def}
For example, a linear form $\psi \in \fa_{\theta}^*$ which is positive on $\L_\theta-\{0\}$ is $(\Ga, \theta)$-proper.
For a $(\Ga, \theta)$-proper form $\psi$, the critical exponent
$0<\delta_\psi=\delta_\psi(\Ga)\le \infty$
of the $\psi$-Poincar\'e series $\cal P_\psi(s)=\sum_{\ga\in \Ga}e^{-s\psi(\mu_\theta(\ga))}$ is well-defined and
we have $$\delta_\psi=\limsup_{t\to \infty}\frac{1}{t}\#\log
 \{\ga\in \Ga:\psi(\mu_\theta(\gamma)) <t\}$$ 
 (see Lemma \ref{wd}).
 
A linear form $\psi\in \fa_\theta^*$ is said to be $(\Ga, \theta)$-critical if $\psi$ is tangent to the $\theta$-growth indicator $\psi_\Ga^{\theta}$, i.e., $\psi\ge \psi_\Ga^{\theta}$  and $\psi(u)=\psi_\Ga^{\theta}(u)$ for some  $u\in \fa^+_\theta-\{0\}$.

\subsection*{Main theorems}
Our main theorems extend Theorem \ref{sul} to higher rank.

\begin{theorem}\label{g}\label{main} Let $\G<G$ be a Zariski dense $\t$-transverse subgroup.
Suppose that there exists a $(\Ga, \psi)$-conformal measure $\nu$ on $\F_\theta$ for $\psi \in \fa_{\theta}^*$.
\begin{enumerate}
    \item If $\psi$ is $(\Ga, \theta)$-proper, then
\be\label{mainin} \psi\ge \psi_\Ga^{\theta}.\ee
\item The following are equivalent:
\begin{enumerate}
    \item  $\sum_{\ga \in \Ga} e^{-\psi(\mu_{\theta}(\ga))} = \infty$ \hspace{1em} (resp. $\sum_{\ga \in \Ga} e^{-\psi(\mu_{\theta}(\ga))} <\infty$);
\item  $\nu(\La^{\mathsf{con}}_{\theta}) = 1 $ \hspace{5.05em} (resp.  $\nu(\La^{\mathsf{con}}_{\theta}) = 0 $).
\end{enumerate}
In the former case,  any $(\Ga, \theta)$-proper $\psi$ is necessarily $(\Ga, \theta)$-critical and $\nu$ is the unique $(\Ga,\psi)$-conformal measure on $\F_\theta$.
\end{enumerate}
\end{theorem}

 When $\theta=\Pi$, 
 Theorem \ref{main}(1) for a Zariski dense discrete subgroup was proved by Quint. For a general $\theta$, only a weaker bound as \eqref{quint} was known by \cite[Theorem 8.1]{Quint2002Mesures}. 
 It implies:
 \begin{theorem}\label{this} Let $\G<G$ be a Zariski dense $\t$-transverse subgroup.
     If there exists a $(\Ga, \psi)$-conformal measure on $\F_{\theta}$ for a $(\Ga, \theta)$-proper $\psi \in \fa_{\theta}^*$, then $$\delta_{\psi} \le 1.$$
 \end{theorem}

\begin{Rmk}
\begin{enumerate}
\item  Canary-Zhang-Zimmer \cite[Theorem 1.4]{Canary2023} proved  the equivalence of $(a)$ and $(b)$ in Theorem
\ref{g}(2)  for $\theta$ symmetric, that is, $\theta=\i(\ts)$, and for conformal measures supported on $\La_\theta$. We mention that transverse subgroups are sometimes called RA-subgroups (cf. \cite{Dey}).
  \item 
  For $\theta$ symmetric and conformal measures supported on $\La_{\theta}$, Theorem \ref{this} was shown in \cite[Theorem 1.4]{Canary2023}.
For some special class of $\theta$-{\it Anosov} subgroups and for
conformal measures {\it supported on} $\La_\theta$,
Theorem \ref{this}  was also proved in  \cite[Theorem C]{pozzetti2019anosov} and \cite[Theorem A]{sambarino2022report}. 
\end{enumerate} \end{Rmk}

As in the original Hopf-Tsuji-Sullivan dichotomy (Theorem \ref{sul}), Theorem \ref{main} can be extended to the dichotomy on the ergodicity of the Weyl chamber flow. Recalling the Hopf parametrization $\Ga\ba (
\F_\Pi^{(2)}\times \fa)
\simeq \Ga\ba G/M$, a natural space to consider is the quotient space
$\Ga\ba (\F_\theta^{(2)}\times \fa_\theta)$
where 
$\F_\theta^{(2)}=\{(\xi, \eta)\in \F_\theta\times \F_{\i(\theta)}: \text{$\xi$, $\eta$ are in general position}\}$ and $\Ga$ acts on 
$\F_{\theta}^{(2)} \times \fa_{\theta}$ from the left by
\be\label{hopf0} \ga . (\xi, \eta, u) = (\ga
\xi, \ga \eta, u + \beta_{\xi}^{\theta}(\ga^{-1}, e)) \ee 
for all $\ga \in \Ga$ and $(\xi, \eta, u)\in \F_{\theta}^{(2)} \times \fa_{\theta}$. 
However the $\Ga$-action on $\F_{\theta}^{(2)} \times \fa_{\theta}$ is not properly discontinuous in general; so the quotient space $\Ga\ba (\F_\theta^{(2)}\times \fa_\theta)$ is not locally compact.

On the other hand, the restriction of the $\Ga$-action on  the subspace $\La_\theta^{(2)}\times \fa_\theta$ turns out to be properly discontinuous where 
 $\La_\theta^{(2)}=\F_\theta^{(2)}\cap (\La_\theta\times \La_{\i(\theta)})$ (Theorem \ref{tproper}):

\begin{theorem}[Properly discontinuous action] \label{thm.pdintro}
Let $\Ga < G$ be a non-elementary $\theta$-transverse subgroup. Then the $\Ga$-action on  $\La_\theta^{(2)}\times \fa_\theta$ given by \eqref{hopf0} is properly discontinuous, and hence the quotient space 
$$\Omega_\theta:=\Ga \ba \La_{\theta}^{(2)} \times \fa_{\theta}$$ is a locally compact Hausdorff space on which $\fa_\theta$ acts by translations from the right.
\end{theorem}

    Indeed,  we prove a stronger property of the action: for a $(\Ga, \theta)$-proper $\varphi \in \fa_{\theta}^*$, we have a projection $\La_{\theta}^{(2)} \times \fa_{\theta} \to \La_{\theta}^{(2)} \times \R$ given by $(\xi, \eta, u) \mapsto (\xi, \eta, \varphi(u))$. The action \eqref{hopf0} descends to the action 
    \be \label{eqn.hopf05}
    \ga . (\xi, \eta, s) = (\ga \xi, \ga \eta, s + \varphi (\beta_{\xi}^{\theta}(\ga^{-1}, e)))
    \ee for all $\ga \in \Ga$ and $(\xi, \eta, s) \in \La_{\theta}^{(2)} \times \R$. 
     We show that the action \eqref{eqn.hopf05} is properly discontinuous, and prove the following (Theorem \ref{tproper2}):
    \begin{theorem}
    Let $\Ga < G$ be a non-elementary $\theta$-transverse subgroup.
For any $(\Ga, \theta)$-proper $\varphi \in \fa_{\theta}^*$,
$\Omega_{\varphi} := \Ga \ba \La_{\theta}^{(2)} \times \R$ is a locally compact Hausdorff space.
Moreover,
$\Omega_{\varphi}$ is compact if and only if $\Ga$ is $\theta$-Anosov.
    \end{theorem}
    
Furthermore, we have a trivial $\ker \varphi$-bundle $\Omega_{\theta} \to \Omega_{\varphi}$ so that $ \Omega_\theta$ is homeomorphic to
    $\Omega_\varphi \times \ker \varphi$ \eqref{bdef1}.

For $\psi \in \fa_{\theta}^*$, we denote by $\mathcal{M}_{\psi}^{\theta}$
the space of all $(\Ga,\psi)$-conformal measures {\it  supported on} $\La_\theta$.
For a pair $(\nu, \nu_{\i}) \in \mathcal{M}_{\psi}^{\theta} \times \mathcal{M}_{\psi \circ \i}^{\i(\theta)}$, we denote by $\mathsf m_{\nu, \nu_{\i}}$ the associated Bowen-Margulis-Sullivan measure on  $\Omega_\theta$ (see \eqref{bms} for its definition).

We expand Theorem \ref{g} to the dichotomy on conservativity and ergodicity of the $\fa_\theta$-action on the space $(\Omega_\theta,\mathsf m_{\nu, \nu_{\i}}) $.  See Theorem \ref{ceq} for a more elaborate statement.
\begin{theorem}\label{c}
   Let $\Ga<G$ be a non-elementary
$\ts$-transverse subgroup. Let $\psi \in\fa_\theta^*$ be $(\Ga, \theta)$-proper such that $\mathcal{M}_{\psi}^{\theta} \neq \emptyset$. 
In each of the following complementary cases, the claims
$(1)-(4)$ are equivalent to each other.

The first case:
\begin{enumerate}
\item $\sum_{\ga\in \Ga}e^{-\psi(\mu_\theta(\ga))}=\infty$;
    \item For any $\nu \in \mathcal{M}_{\psi}^{\theta}$, $\nu(\La_{\theta}^{\mathsf{con}}) = 1$;

\item For any $(\nu, \nu_{\i}) \in \mathcal{M}_{\psi}^{\theta} \times \mathcal{M}_{\psi \circ \i}^{\i(\theta)}$,
the $\Ga$-action on $(\La_\theta^{(2)}, \nu\times \nu_{\i})$ is completely conservative and ergodic;
   
\item  For any $(\nu, \nu_{\i}) \in \mathcal{M}_{\psi}^{\theta} \times \mathcal{M}_{\psi \circ \i}^{\i(\theta)}$, the $\fa_\theta$-action on $(\Omega_\theta, {\mathsf m}_{\nu,\nu_{\i}})$  is completely conservative and ergodic.

\end{enumerate}

The second case:
\begin{enumerate}
\item $\sum_{\ga\in \Ga}e^{-\psi(\mu_\theta(\ga))}<\infty$;
    \item For any $\nu \in \mathcal{M}_{\psi}^{\theta}$, $\nu(\La_{\theta}^{\mathsf{con}}) = 0$;

\item For any $(\nu, \nu_{\i}) \in \mathcal{M}_{\psi}^{\theta} \times \mathcal{M}_{\psi \circ \i}^{\i(\theta)}$, 
the $\Ga$-action on $(\La_\theta^{(2)}, \nu\times \nu_{\i})$ is completely dissipative and non-ergodic;
    
\item  For any $(\nu, \nu_{\i}) \in \mathcal{M}_{\psi}^{\theta} \times \mathcal{M}_{\psi \circ \i}^{\i(\theta)}$, the $\fa_\theta$-action on $(\Omega_\theta, {\mathsf m}_{\nu,\nu_{\i}})$  is completely dissipative and non-ergodic.
 \end{enumerate}
 \end{theorem}

 When $\theta$ is symmetric,
    the equivalences (1)-(3) in both cases were proved in \cite[Theorem 1.4]{Canary2023} by a different approach.

\subsection*{Disjoint dimensions phenomenon}
Let $$\cal D^\theta_{\Ga}=\left\{\psi\in \fa_\theta^*: \text{$(\Ga, \theta)$-proper}, \delta_\psi(\Ga)=1 \mbox{ and }
\cal P_\psi(1)=\infty \right\}.$$
This is in fact same as 
$$\left\{\psi\in \fa_\theta^*: \text{$(\Ga, \theta)$-proper},  \text{$\exists$ a $(\Ga,\psi)$-conformal measure},
\cal P_\psi(1)=\infty \right\}$$
when $\Ga$ is a $\theta$-transverse subgroup (see Lemma \ref{dsame}).

Inspired by the entropy drop phenomenon proved by Canary-Zhang-Zimmer \cite[Theorem 4.1]{Canary2023} for $\theta=\i(\theta)$, we deduce from Theorem \ref{g} the following disjointness of dimensions (Theorem \ref{disss}), which turns out to be equivalent to the entropy drop phenomenon (Corollary \ref{entropy}):
\begin{cor} [Disjoint dimensions]\label{dis0}
    Let $\Ga<G$ be a non-elementary  $\t$-transverse subgroup. For any   subgroup $\Ga_0<\Ga$ with $\La_\theta(\Ga_0)\ne \La_\theta(\Ga)$, we have
     $$\mathcal D^\theta_{\Ga}\cap \mathcal D^\theta_{\Ga_0}=\emptyset.$$
\end{cor}

In the rank one case, this corollary says that if $\La(\Ga_0)\ne \La(\Ga)$ and $\Ga_0<\Ga$ are of divergence type, that is, their Poincar\'e series diverge at the critical exponents,
then $\delta_{\Ga_0}<\delta_\Ga$. We refer to \cite{Canary2023} for a more detailed background on this phenomenon.

\subsection*{$\theta$-Anosov subgroups} 
A finitely generated subgroup $\Ga<G$ is a $\theta$-Anosov subgroup if there exists $C>0$ such that for all $\ga\in \Ga$,
\be\label{ma} \min_{\alpha\in \theta} \alpha(\mu(\ga))\ge C |\ga| - C^{-1} \ee 
    where $|\ga|$ denotes the word length of $\ga$ with respect to a fixed finite generating set of $\Ga$ (\cite{Labourie2006anosov}, \cite{Gueritaud2017anosov}, \cite{Kapovich2018discrete}, \cite{Kapovich2017anosov}, \cite{Kapovich2018morse}).  All $\theta$-Anosov subgroups
are $\ts$-transverse and $\lat=\lat^{\mathsf{con}}$ (\cite{Guichard2012anosov}, \cite{Kapovich2017anosov}).  
We deduce the following  from Theorem \ref{g}:

\begin{theorem} \label{g2}\label{main2}
\label{thm.Anosovmain} \label{cor.critdim}\label{thm.Ahlfors}  Let $\G<G$ be a Zariski dense $\theta$-Anosov subgroup.
Suppose that there exists a $(\Ga, \psi)$-conformal measure $\nu$ on $\F_\theta$ for  $\psi \in \fa_{\theta}^*$. We have:
\begin{enumerate}
    \item 
The linear form $\psi$ is $(\Ga, \theta)$-proper and
$\psi\ge \psi_\Ga^{\theta}.$
\item The following are equivalent to each other:
\begin{enumerate}
    \item  $\sum_{\ga \in \Ga} e^{-\psi(\mu_{\theta}(\ga))} = \infty$ \hspace{1em} (resp. $\sum_{\ga \in \Ga} e^{-\psi(\mu_{\theta}(\ga))} <\infty$);
\item  $\nu(\La_{\theta}) = 1 $ \hspace{5.8em} (resp.  $\nu(\La_{\theta}) = 0 $);
\item $\psi$ is $(\Ga, \theta)$-critical \hspace{2.3em} (resp. $\psi$ is not $(\Ga, \theta)$-critical).
\end{enumerate}
\item For each $(\Ga, \theta)$-critical $\psi\in \fa_\theta^*$, there exists a unique $(\Ga, \psi)$-conformal measure, say, $\nu_{\psi}$, on $\F_\theta$, which is necessarily supported on $\La_\theta$. Moreover
the $\fa_\theta$-action on $(\Omega_\theta, \mathsf m_{\nu_{\psi}, \nu_{\psi\circ \i}})$
is completely conservative and ergodic.
\end{enumerate}
\end{theorem}

The equivalence $(a) \Leftrightarrow (b)$ in (2) 
answers  a question asked by Sambarino \cite[Remark 5.10]{sambarino2022report}.

\subsection*{Analogue of Ahlfors measure conjecture for $\theta$-Anosov groups}
We denote by $\Leb_\theta$ the Lebesgue measure on $\F_{\theta}$, which is the unique $K$-invariant probability measure on $\F_{\theta}$. The following corollary is motivated by the Ahlfors measure conjecture \cite{Ahlfors1964finitely}. 
\begin{cor} \label{cor.Ahlforstrans}\label{critical} 
    If $\Ga < G$ is Zariski dense $\t$-Anosov, then $$\text{either }\La_{\theta} = \F_{\theta} \quad \mbox{or} \quad \Leb_{\theta}(\La_{\theta}) = 0.$$
Moreover, in the former case,
   $\theta$ is the simple root of a rank one factor, say $G_0$, of $G$ and
   $\Ga$ projects to a cocompact lattice  of $G_0$. 
\end{cor}

See Theorem \ref{Ahlfors} for a more general version
stated for a $\theta$-transverse subgroup.

\subsection*{Critical forms and conformal measures} We set $$ \cal T_\Ga^\theta :=\{\psi\in \fa_\theta^*:\text{$\psi$ is $(\Ga, \theta)$-critical}\}.$$ 
Note that  $\cal D_{\Ga}^{\theta} \subset \cal T_{\Ga}^{\theta}$ (Corollary \ref{oneone}).
For $\theta$-Anosov subgroups, we further have $\cal T_\Ga^\theta =\cal D_\Ga^\theta$, which is again same as the set of all
$\psi\in \fa_\theta^*$ for which there exists a $(\Ga, \psi)$-conformal measure supported on $\La_\theta$ (Lemma \ref{cri}). 
A $\Ga$-conformal measure is said to be of {\it critical} dimension if the associated linear form belongs to $ \cal T_{\Ga}^{\theta}$.
Using Sambarino's parametrization of
the space of all conformal measures on $\La_\theta$ as
$\{\delta_\psi=1\}$ \cite[Theorem A]{sambarino2022report}, we deduce:
\begin{cor}    \label{co.equivtheta}\label{coeq}
For any Zariski dense $\theta$-Anosov subgroup $\Ga<G$,
we have a one-to-one correspondence among
\begin{enumerate}
    \item the set $\cal T_\Ga^\theta$ of all $(\Ga,\theta)$-critical forms on $\fa_{\theta}$;
\item the set of all unit vectors in $\inte \L_\theta$;
\item the set of all $\Ga$-conformal measures supported on $\La_\theta$;
\item the set of all $\Ga$-conformal measures on $\F_\theta$ of critical dimensions.
\end{enumerate}
More precisely, for any $\psi\in \cal T_\Ga^\theta$, there exists a unique unit vector $u_\psi\in \fa_\theta^+$
such that $\psi(u_\psi)=\psi^\theta_\Ga(u_\psi)$; moreover $u_\psi\in \inte \L_\theta$. There also exists a unique $(\Ga, \psi)$-conformal measure $\nu_\psi$ on $\F_\theta$, which is necessarily supported on $\La_\theta$. Moreover  every $\Ga$-conformal measure supported on $\La_{\theta}$ arises in this way.

\end{cor}

\begin{cor} [Disjoint critical dimensions]\label{dis00}
Let $\Ga < G$ be a Zariski dense $\theta$-Anosov subgroup. For any subgroup  $\Ga_0<\Ga$ such that $\La_\theta(\Ga_0)\ne \La_\theta(\Ga)$,
    we have
     $$\mathcal T_{\Ga}^\theta\cap \mathcal T^\theta_{\Ga_0}=\emptyset
     \quad \text{and} \quad 
     \psi^\theta_{\Ga_0}< \psi^\theta_\Ga \mbox{ on } \inte \L_{\theta}(\Ga).
     $$
\end{cor}
Indeed, the above two conclusions are equivalent to each other by the vertical tangency  and concavity of $\psi_{\Ga}^{\theta}$ (see Corollary \ref{cor.equiv in 1.15} for the proof).

\begin{Rmk}\label{intror}
 Related dichotomy properties for
conformal measures 
were studied in  \cite{Dey}, \cite{burger2021hopf}, \cite{lee2022dichotomy}, \cite{edwards2021unique}, \cite{sambarino2022report}, \cite{Canary2023}, etc.
In particular, when $\Ga$ is $\Pi$-Anosov,  Theorem \ref{thm.Anosovmain},
     Corollaries \ref{critical} and \ref{coeq} were proved by Lee-Oh  \cite[Theorems 1.3, 1.4]{lee2022dichotomy}.
    The papers \cite{Dey}, \cite{sambarino2022report},  and \cite{Canary2023} study conformal measures {\it  supported on the limit set $\La_\theta$} and  the papers \cite{burger2021hopf}
and \cite{edwards2021unique}
study the role of {\it  directional} conical sets in the ergodic behavior of conformal measures.
Our focus on this paper is to address 
general conformal measures without restriction on their supports following \cite{lee2022dichotomy} and to study the relationship between the dimensions of conformal measures and $\theta$-growth indicators so as to establish an analogue of Sullivan's theorem (Theorem \ref{sul}) and the analogue of the Ahlfors measure conjecture. We also emphasize that the $\theta$-growth indicator is first introduced  in our paper.  
Notably, Theorem \ref{thm.pdintro} provides a new locally compact Hausdorff space $\Omega_{\theta} := \Ga \ba \La_{\theta}^{(2)} \times \fa_{\theta}$ which is a non-wandering set for the Weyl chamber flow $A_{\theta}$. This allows us to define Bowen-Margulis-Sullivan measures  as in the rank one setting. Hence the dynamical properties of the Weyl chamber flow can be studied also in higher rank, fully recovering the original work of Hopf-Tsuji-Sullivan.

\end{Rmk}
Finally, we mention that there is a plethora of examples of $\ts$-transverse subgroups which are not $\theta$-Anosov. First of all, any subgroup of
$\theta$-Anosov subgroups are $\theta$-transverse. For instance, a co-abelian subgroup of a $\theta$-Anosov subgroup of infinite index is $\theta$-transverse but not $\theta$-Anosov. The images of
cusped Hitchin representations
of geometrically finite Fuchsian groups  by
\cite{Canary2022cusped} are also $\theta$-transverse but not $\theta$-Anosov.
Another important examples are self-joinings of geometrically finite subgroups of rank one Lie groups, that is,
$\G=(\prod_{i=1}^k\rho_i)(\Delta)=\{(\rho_i(g))_i:g\in \Delta\}$
where $\Delta$ is a geometrically finite subgroup of a rank one simple real algebraic group $G_0$ and
$\rho_i:\Delta\to G_i$ is a type-preserving isomorphism onto its image $\rho_i(\Delta)$ which is a geometrically finite subgroup
of a rank one simple real algebraic group $G_i$ for each $1\le i \le k$. It follows from \cite[Theorem 3.3]{Tukia1985isomorphisms} and \cite[Theorem A.4]{Das2016tukia} (see also \cite[Theorem 0.1]{Yaman2004topological}) that there exists a $\rho_i$-equivariant homeomorphism between the limit set of $\Delta$ and the limit set of $\rho_i(\Delta)$ for each $1\le i\le k$. This implies
that $\G$ is $\Pi$-transverse.

\subsection*{Organization}
\begin{itemize}
    \item 
In section \ref{sec.convergence}, we introduce the notion of convergence of elements of $G$ to those of $\F_\theta$ and present some basic lemmas which will be used in the proof of our main theorems. 
\item  In section \ref{sec.growthindicator}, we define the $\theta$-growth indicator $\psi_\Ga^{\theta}$ for a $\theta$-discrete subgroup $\Ga<G$. Properties of the $\theta$-growth indicator  and its relationship with Quint's growth indicator \cite{Quint2002divergence} are also discussed. 
\item  In section \ref{sec.linforms},  we introduce $(\Ga, \theta)$-proper linear forms and $(\Ga, \theta)$-critical linear forms and discuss properties of their critical exponents.
\item  In section \ref{sec.conical}, we define the $\theta$-limit set and the $\theta$-conical set of $\Gamma$. For $\theta$-regular subgroups, we show that the $\theta$-conical set is a subset of the $\theta$-limit set and construct conformal measures supported on the $\theta$-limit set for each $\psi\in \cal D_\Ga^\theta$. 
\item In section \ref{sec.shadow},
we prove that for $\ts$-transverse subgroups, $\theta$-shadows with bounded width have bounded multiplicity,
which is one of the key technical ingredients of our main results. 
\item In section \ref{sec.dimensions}, we show that if $\Ga$ is a $\ts$-transverse subgroup, the dimension of a $\Ga$-conformal measure is at least $\psi_{\Ga}^{\theta}$ (Theorem \ref{tgg}).
\item In section \ref{sec.supp}, we prove the zero-one law for
the $\nu$-size of the conical set depending on whether or not
the associated Poincar\'e series converges at its dimension (Theorem \ref{thm.divthensupp}).
\item In section \ref{pdt}, we prove that a $\theta$-transverse subgroup $\Ga$ acts properly discontinuously on $\La_\theta^{(2)}\times \fa_\theta$ and define Bowen-Margulis-Sullivan measures on the space $\Omega_\theta =\Gamma\ba \La_\theta^{(2)}\times \fa_\theta$. For any $(\Ga, \theta)$-proper form $\varphi$, we also show that
 the $\varphi$-twisted $\Ga$-action on $\La_\theta^{(2)}\times \br$ is properly discontinuous
 and gives rise to a trivial vector bundle
 $\Omega_\theta \to \Omega_{\varphi}=\Ga\ba \La_\theta^{(2)} \times \br$.

 \item In section \ref{sec.BMS measures}, we give the definition of Bowen-Margulis-Sullivan measures.

\item In section \ref{ec}, we expand the equivalence between dichotomies to conservativity and ergodicity of the $\fa_\theta$-action on $\Omega_\theta$, proving Theorem \ref{c}.  We also explain how to deduce
Theorem \ref{main} from Theorems \ref{tgg} and \ref{ggg}.

\item In section \ref{app}, we discuss several consequences of Theorem \ref{ggg}, including disjoint dimension phenomenon.
\item Finally, in section \ref{sec.Anosov} we discuss how our theorems are applied for $\theta$-Anosov groups. We also prove Corollary \ref{cor.Ahlforstrans}.
\end{itemize}

\subsection*{Acknowledgement} We would like to thank Jean-Fran\c{c}ois Quint for useful conversations about Lemma \ref{qu2}.

\section{Convergence in $G\cup \F_\theta$.} \label{sec.convergence}
In the whole paper, let $G$ be a connected semisimple real algebraic group. 
 Let $P<G$ be a minimal parabolic subgroup with a fixed Langlands decomposition $P=MAN$ where $A$ is a maximal real split torus of $G$, $M$ is the maximal compact subgroup of $P$ commuting with $A$ and $N$ is the unipotent radical of $P$.
Let $\fg$ and $\fa$ respectively denote the Lie algebras of $G$
and $A$. Fix a positive closed Weyl chamber $\fa^+<\fa$ so that
$\log N$ consists of positive root subspaces and
set $A^+=\exp \fa^+$. We fix a maximal compact subgroup $K< G$ such that the Cartan decomposition $G=K A^+ K$ holds. We denote by 
$$\mu : G \to \fa^+$$ the Cartan projection defined by the condition $g\in K\exp \mu(g) K$ for $g \in G$. 
Let $X = G/K$ be the associated Riemannian symmetric space, and set $o = [K] \in X$.  Fix a $K$-invariant norm $\| \cdot \|$ on $\fg$ induced from the Killing form on $\fg$ and let $d$ denote the Riemannian metric on $X$ induced by $\| \cdot \|$.

\begin{lemma} \cite[Lemma 4.6]{Benoist1997proprietes} \label{lem.cptcartan}
For any compact subset $Q \subset G$, there exists $C=C(Q)>0$ such that for all $g \in G$, $$\sup_{q_1, q_2\in Q} \| \mu(q_1gq_2) -\mu(g)\| \le C .$$  
\end{lemma}

Let $\Phi=\Phi(\fg, \fa)$ denote the set of all roots,  $\Phi^{+}\subset \Phi$
the set of all positive roots, and $\Pi \subset \Phi^+$  the set of all simple roots. We denote by $N_K(A)$ and $C_K(A)$ the normalizer and centralizer of $A$ in $K$ respectively.
Consider the Weyl group $\cal W=N_K(A)/C_K(A)$.
Fix an element $$w_0\in N_K(A) $$ representing the longest Weyl element so that $\op{Ad}_{w_0}\mathfrak a^+= -\mathfrak a^+$ and $w_0^{-1}=w_0$. Hence the map $$\i= -\op{Ad}_{w_0}:\fa\to \fa$$ defines an involution of $\fa$ preserving $\fa^+$; this is called the opposition involution. It induces a map $\Phi\to \Phi$ preserving $\Pi$, for which we use the same notation $\i$, such that $\i (\alpha ) \circ  \op{Ad}_{w_0}  =-\alpha $ for all $\alpha\in \Phi$. We have 
\be \label{mu}
\mu(g^{-1})=\i (\mu(g))\quad\text{ for all $g\in G$ }.
\ee

In the rest of the paper, we  fix a  non-empty subset $\theta\subset \Pi$. 
Let $ P_\theta$ denote a standard parabolic subgroup of $G$ corresponding to $\theta$; that is,
$P_{\theta}$ is generated by $MA$ and all root subgroups $U_\alpha$,
$\alpha\in \Phi^{+} \cup [\Pi-\theta]$ where $[\Pi-\theta]$ denotes the set of all roots in $\Phi$ which are $\mathbb Z$-linear combinations of $\Pi-\theta$. Hence $P_\Pi=P$.
The subgroup $P_\theta$ is equal to its own normalizer; for $g\in G$,
$gP_\theta g^{-1}=P_\theta$ if and only if $g\in P_\theta$.
Let 

$$
\begin{aligned}
    \mathfrak{a}_\theta &=\bigcap_{\alpha \in \Pi-\theta} \ker \alpha, & \fa_\theta^+ & =\fa_\theta\cap \fa^+, \\
    A_{\theta} & = \exp \fa_{\theta}, \mbox{ and} &    A_{\theta}^+ & = \exp \fa_{\theta}^+.
\end{aligned}
$$

Let $$ p_\theta:\mathfrak{a}\to\mathfrak{a}_\theta$$ denote  the projection invariant under $w\in \cal W$ fixing $\fa_\theta$ pointwise. 

Let $L_\theta$ denote the centralizer of $A_{\theta}$;
it is a Levi subgroup of $P_\theta$ and $P_\theta=L_\theta N_\theta$ where $N_\theta=R_u(P_\theta)$ is the unipotent radical of $P_\theta$. We set $M_{\theta} = K \cap P_{\theta}\subset L_\theta$.
We may then write $L_{\theta} = A_{\theta}S_{\theta}$ where $S_{\theta}$ is an almost direct product of
 a connected semisimple real algebraic subgroup and a compact subgroup.
 Then $B_\theta=S_\theta \cap A$ is a maximal $\br$-split torus of $S_\theta$ and $\Pi-\theta$ is the set of simple roots
 for $(\op{Lie} S_\theta, \op{Lie} B_\theta)$.
 Letting $$B_\theta^+=\{b\in B_\theta: \alpha (\log b)\ge 0
 \text{ for all $\alpha\in \Pi-\theta$}\},$$
 we have the Cartan decomposition of $S_\theta$: 
$$S_\theta = M_{\theta} B_\theta^+ M_{\theta}.$$

Any $u\in \fa$ can be written as $u= u_1+u_2$ for unique $u_1\in \fa_\theta$
and $u_2\in \op{Lie} B_\theta$, and we have $p_\theta(u)=u_1$. 
In particular, we have  $$A=A_\theta B_\theta\quad\text{and}\quad A^+\subset A_\theta^+ B_\theta^+.$$

We denote by $\fa_\theta^*=\op{Hom}(\fa_\theta, \br)$ the dual space of $\fa_\theta$. It can be identified with the subspace of $\fa^*$ which is $p_\theta$-invariant: $\fa_\theta^*=\{\psi\in \fa^*: \psi\circ p_\theta=\psi\}$; so for $\theta_1\subset \theta_2$, we have $\fa_{\theta_1}^*\subset \fa_{\theta_2}^*$.

\subsection*{The $\theta$-boundary $\F_\theta$ and
convergence to $\F_\theta$} 

We set $$\F_\theta=G/P_{\theta} \quad\text{and}\quad \F=G/P.$$
Let $$ \pi_\theta:\F\to \F_\theta$$ denote
 the canonical projection map given by $gP\mapsto gP_\theta$, $g\in G$. 
 We set \be\label{xit} \xi_\theta=[P_\theta] \in \F_{\theta}.\ee 
 By the Iwasawa decomposition $G=KP=KAN$, the subgroup $K$ acts transitively on $\F_\theta$, and hence
 $$\F_\theta\simeq K/ M_\theta.$$

We consider the following notion of convergence of a sequence in $G$ to an element of $\F_\theta$.
\begin{definition} \label{fc} For a sequence $g_i\in G$  and $\xi\in \ft$, we write $\lim_{i\to \infty} g_i=\lim g_i o =\xi$ and
 say $g_i $ (or $g_io \in X$) converges to $\xi$ if \begin{itemize}
     \item $\min_{\alpha\in \theta} \alpha(\mu(g_i)) \to \infty$; and
\item $\lim_{i\to\infty} \kappa_{g_i}\xi_\theta= \xi$ in $\F_\theta$ for some $\kappa_{g_i}\in K$ such that $g_i\in \kappa_{g_i} A^+ K$.
 \end{itemize}         
\end{definition}

\subsection*{Points in general position} Let $P_\theta^+$ be the
standard parabolic subgroup of $G$ opposite to $P_\theta$ such that $P_\theta\cap P_\theta^+=L_\theta$. Set $P^+:=P_\Pi^+$. We have $P_\theta^+ =w_0 P_{\i(\theta)}w_0^{-1}$ and hence 
$$\F_{\i(\theta)}=G/P_\theta^+.$$
In particular, if $\theta$ is symmetric in the sense that $\theta=\i(\theta)$, then $\F_\theta=G/P_\theta^+$. Let $N_\theta^+$ denote the unipotent radical of $P_\theta^+$.
The set $N_\theta^+P_\theta$ is a Zariski open and dense
subset of $G$. In particular, $N_\theta^+\xi_\theta\cap hN_\theta^+\xi_\theta\ne \emptyset$ for any $h\in G$.
The $G$-orbit
of $(P_\theta, P_\theta^+)$ is the unique open $G$-orbit
in $G/P_\theta\times G/P_\theta^+$ under the diagonal $G$-action. 
 Since $P = MAN$ and $P^+ = MAN^+$, $a\in A^+$ centralizes $MA$, and its conjugation action on $N$ (resp. $N^+$) contracts
 (resp. expands),
 the following is immediate:

\begin{lemma} \label{lem.boundedofparabolic}
    Let $Q \subset P$ and $Q^+ \subset P^+$ be bounded subsets. For any sequence $a_i \in A^+$, both sequences $a_i^{-1} Q a_i$ and $a_i Q^+ a_i^{-1}$ are uniformly bounded.
\end{lemma}

\begin{Def}
 Two elements
$\xi\in \F_\theta$ and $\eta\in \F_{\i(\theta)}$ are said to be in general position if $(\xi, \eta)\in G.(P_\theta, w_0 P_{\i(\theta)})=G. (P_\theta, P_\theta^+)$, i.e.,
$\xi=gP_\theta$ and $\eta=gw_0 P_{\i(\theta)}$ for some $g\in G$.
\end{Def}

 We set
\be\label{f2} \F_{\theta}^{(2)}=\{(\xi, \eta) \in \F_{\theta} \times \F_{\i(\theta)}: \text{$\xi$, $\eta$ are in general position}\}, \ee 
which is
the unique open $G$-orbit in $ \F_{\theta} \times \F_{\i(\theta)}$. It follows from the identity $P_\theta^+=N_{\theta}^+(P_\theta\cap P_\theta^+)$ that 
\be \label{op} (gP_\theta, P_\theta^+)\in 
\F_{\theta}^{(2)} \;\; \text{ if and only if }\;\; g\in N_\theta^+P_\theta .\ee

\subsection*{Basic lemmas} We generalize \cite[Lemmas 2.9-11]{lee2020invariant} from $\theta=\Pi$ to a general $\theta$ 
as follows. For subsets $S_i\subset G$,
 we often write $g=g_1g_2g_3 \in S_1S_2S_3$ to mean that $g_i\in S_i$ for each $i$, in addition to $g=g_1g_2g_3$.
\begin{lemma} \label{lem.29inv}
    Consider a sequence $g_i = k_i a_i h_i^{-1}$ where $k_i \in K$, $a_i \in A^+$, and $h_i \in G$. Suppose that $k_i \to k_0\in K$, $h_i \to h_0\in G$, and $\min_{\alpha\in \theta} \alpha(\log a_i) \to \infty$, as $i\to \infty$. Then for any $\xi \in h_0N_\theta^+\xi_\theta$
    (i.e., $\xi$ is in general position with $h_0P_\theta^+$),  we have $$\lim_{i \to \infty} g_i \xi = k_0 \xi_\theta.$$
\end{lemma}

\begin{proof}
    Since $h_i^{-1}\xi$ converges to the element $ h_0^{-1} \xi\in N_\theta^+\xi_\theta$ by the hypothesis
    and $N_\theta^+\xi_\theta\subset \F_\theta$ is open, we have $h_i^{-1}\xi\in N_\theta^+\xi_\theta$ for all large $i$. 
    Hence we can write $h_i^{-1}\xi= n_i \xi_\theta$ with $n_i\in N_\theta^+$ uniformly bounded. Since $\min_{\alpha\in \theta} \alpha(\log a_i) \to \infty$
    and $n_i\in N_\theta^+$ is uniformly bounded, we have $a_i n_i a_i^{-1}\to e$ as $i \to \infty$. Therefore
     the sequence $a_i h_i^{-1} \xi = a_i n_i a_i^{-1}  \xi_\theta$ converges to $ \xi_\theta$.
Hence we have $$\lim_{i \to \infty} g_i \xi = \lim_{i \to \infty} k_i(a_i h_i^{-1} \xi) = k_0\xi_\theta.$$
\end{proof}

  \begin{cor} \label{cor.genweyl}
If $w\in N_K(A)$ is such that
$m w\in N_\theta^+P_\theta$ for some $m\in M_\theta$, then $w \in M_{\theta}$.    In particular, if $wP_\theta$ and $P_\theta^+$ are in general position, then $w\in M_\theta$.
\end{cor}

\begin{proof}
    Choose any sequence $a_i \in A_\theta^+$ such that
    $\min_{\alpha\in \theta}\alpha(\log a_i)\to \infty$.
    Since $mw\xi_\theta\in N_\theta^+\xi_\theta$,
    we deduce from Lemma \ref{lem.29inv} that $a_im w\xi_\theta $ converges to $ \xi_\theta$ as $i \to \infty$. On the other hand,
    since $w\in N_K(A)$, $A\subset P_\theta$ and $m\in M_\theta$,
   we have $a_i mw\xi_\theta=
   m w(w^{-1} a_iw) \xi_\theta= mw\xi_\theta $ for all $i$.  Hence 
    $mw\xi_{\theta} = \xi_{\theta}$. Since $m\in M_\theta$, this implies $w \xi_{\theta} = \xi_{\theta}$ and hence $w\in P_\theta\cap K=M_\theta$.
\end{proof}

It turns out that the convergence of $g_i\to \xi$ is equivalent to $g_i p\to \xi$
for any $p\in X$. More generally, we have 
\begin{lemma} \label{lem.210inv}\label{bdd}
    If a sequence $g_i \in G$ converges to $\xi \in \F_{\theta}$ and $p_i \in X$ is a
    bounded sequence, then $$\lim_{i \to \infty} g_i p_i = \xi.$$
\end{lemma}

\begin{proof}  Let $g_i' \in G$ be such that $g_i' o = p_i$; then $g_i'$ is bounded.
Since $\lim g_i=\xi$, we may
    write $g_i = k_i a_i \ell_i^{-1}$ with $k_i, \ell_i\in K$ and $a_i\in A^+$ where $\min_{\alpha\in \theta} \alpha(\log a_i) \to \infty$, and $k_i \xi_\theta \to \xi$ as $i \to \infty$. 
    Write $g_i g_i' = k_i'a_i'(\ell_i')^{-1} \in KA^+K$. Since $g_i'$ is bounded,
    $\lim_{i\to \infty}
    \min_{\alpha\in \theta} \alpha(\log a_i') = \infty$, by Lemma \ref{lem.cptcartan}. 
    Let $q \in K$ be a limit of the sequence $q_i:=k_i^{-1} k_i'$. By passing to a subsequence, we may assume that $q_i \to q$. Since $d(o, p_i) = d(g_i o, g_i p_i) = d( o, a_i^{-1}q_ia_i' o)$, the sequence $h_i^{-1} := a_i^{-1} q_i a_i'$ is bounded. Passing to a subsequence, we may assume that $h_i$ converges to some $h_0 \in G$. Choose any $\eta \in N_\theta^+\xi_\theta\cap h_0N_\theta^+\xi_\theta$. By Lemma \ref{lem.29inv}, we have 
    $$\lim_{i \to \infty} a_i h_i^{-1} \eta = \xi_\theta\;\;\text{  and }\;\; \lim_{i \to \infty} q_i a_i' \eta = q \xi_\theta.$$
    Since $a_i h_i^{-1} = q_i a_i' $, it follows that $ q\xi_\theta = \xi_\theta$; so $q\in K\cap P_\theta$.
    Hence $\xi=\lim k_i\xi_\theta =\lim k_i' \xi_\theta $. It follows that $\lim g_ip_i=\xi$.
\end{proof}

\begin{lemma}  \label{lem.211inv}
    If a sequence $g_i\in G$ converges to $g$ and a sequence $a_i\in A^+$ satisfies $
    \min_{\alpha\in \theta} \alpha(\log a_i) \to \infty$ as $i\to \infty$, then for any $p \in X$, we have $$\lim_{i \to \infty} g_i a_i p = g \xi_\theta .$$
\end{lemma}

\begin{proof}
    By Lemma \ref{lem.210inv}, it suffices to consider the case when $p = o$. Write $g_i a_i = k_i b_i \ell_i^{-1} $
    with $k_i, \ell_i\in K$ and $b_i\in A^+$. Since the sequence $g_i$ is bounded,  $\lim_{i\to \infty}
    \min_{\alpha\in \theta} \alpha(\log b_i) =\infty$ by Lemma \ref{lem.cptcartan}.
Let $k_0$ be a limit of the sequence $k_i$; without loss of generality, we may assume that $k_i$ converges to $k_0$ as $i\to \infty$. Then $\lim_{i \to \infty} g_i a_i o=k_0\xi_\theta$.
We may also assume that $\ell_i $ converges to some $ \ell_0\in K$. Choose $\xi \in \ell_0N_\theta^+\xi_\theta\cap N_\theta^+\xi_\theta $. Then by Lemma \ref{lem.29inv}, as $i\to \infty$,
$g_i a_i \xi \to k_0 \xi_\theta$ and $a_i \xi \to  \xi_\theta$. Since $g_i $ converges to $ g$, this implies that $k_0 \xi_{\theta} = g \xi_{\theta}$. This finishes the proof.
\end{proof}

 \section{Growth indicators}\label{grow} \label{sec.growthindicator}

Let $\Ga<G$ be a discrete subgroup.
We set
\be\label{mut}\mu_\theta:=p_\theta\circ \mu:G\to \fa_\theta^+.\ee 
\begin{Def} \label{def.propergroup} We say that $\Ga$ is $\theta$-discrete 
if  the restriction $\mu_\theta|_{\Ga}:\Ga \to \fa_\theta^+$ is proper.
\end{Def}
The $\theta$-discreteness of $\Ga$ implies that $\mu_\theta(\Ga)$ is a closed discrete subset of $\fa_\theta^+$. Indeed, $\Ga$ is $\theta$-discrete if and only if the counting measure on $\mu_\theta(\Gamma)$ weighted with multiplicity is a Radon measure on $\fa_\theta^+$.

\begin{Def}[$\theta$-growth indicator]  For a $\theta$-discrete subgroup $\Ga<G$,
we define the $\theta$-growth indicator $\psi_\Ga^{\theta}:\fa_\theta\to [-\infty, \infty] $ as follows: if $u \in \fa_\theta$ is non-zero,
\be\label{gi2} \psi_\Ga^{\theta}(u)=\|u\| \inf_{u\in \cal C}
\tau^\theta_{\mathcal C}\ee 
where $\cal C\subset \fa_\theta$ ranges over all open cones containing $u$, and $\psi_{\Ga}^{\theta}(0) = 0$.
Here $-\infty\le \tau^{\theta}_{\cal C}\le \infty$ denotes the abscissa of convergence of the series ${\mathcal P}^{\theta}_{\cal C}(s)=\sum_{\ga\in \Ga, \mu_\theta(\ga)\in \mathcal C} e^{-s\|\mu_\theta(\ga)\|}$, that is,
$$\tau^{\theta}_{\cal C}=\sup\{s\in \br : {\mathcal P}_{\cal C}^{\theta}(s)=\infty\}= \inf\{s\in \br:\mathcal P_{\cal C}^{\theta}(s)<\infty\} .$$
\end{Def}
This definition is independent of the choice of a norm on $\fa_\theta$.  For $\theta=\Pi$,
we set $$\psi_\Ga:=\psi_\Ga^{\Pi}.$$
The main goal of this section is to establish the following properties of $\psi_\Ga^\theta$ for a general $\theta\subset \Pi$: 
for $\theta=\Pi$, this theorem is due to Quint \cite[Theorem 1.1.1]{Quint2002divergence}. 
\begin{theorem}\label{three} Let $\Ga<G$ be 
a Zariski dense $\theta$-discrete subgroup.
\begin{enumerate}
    \item  $\psi_\Ga^{\theta} <\infty$.
   \item $\psi_\Ga^\theta$ is a homogeneous, upper semi-continuous and concave function. 
     \item $\L_\theta=\{\psi_\Ga^{\theta}\ge 0\}$, $\psi_\Ga^{\theta}=-\infty$ outside $\L_\theta$ and
    $\psi_\Ga^{\theta}>0$ on $\inte \L_\theta$.
\end{enumerate}
\end{theorem}
Here, $\L_\theta\subset \fa_\theta^+$ denotes the $\theta$-limit cone of $\Ga$, which is the asymptotic cone of $\mu_\theta(\Ga)$:
 \be\label{limitc} \L_\theta=\{\lim t_i\mu_\theta(\ga_i): \ga_i\in \Ga, t_i\to 0\}.\ee 
We set $\L=\L_{\Pi}$, which is the usual limit cone.
 By \cite[Sections 1.2, 4.6]{Benoist1997proprietes}, if $\Ga$ is Zariski dense, then $\L$ is a convex cone with non-empty interior and $\mu(\Ga)$ is within a bounded distance from $\L$.
 We have \be\label{pg}
 \L=\{\psi_\Ga \ge 0\},\quad  \text{and}  \quad \psi_\Ga >0
 \text{ on $\inte \L$} \ee
 and $\psi_\Ga=-\infty$ outside $\L$ \cite[Theorem 1.1.1]{Quint2002divergence}. Noting that $\L_\theta=p_\theta(\L)$, we get:
\begin{lemma}\label{tcone}  Let $\Ga$ be a Zariski dense $\theta$-discrete subgroup. The $\theta$-limit cone $\L_\theta$ is a convex cone in $\fa_\theta^+$ with non-empty interior and $\mu_\theta(\Ga)$ is within a bounded distance from $\L_\theta$. 
\end{lemma}
\subsection*{$\psi_\Ga^\theta < \infty$
and $\theta$-critical exponent}
In this subsection, we show Theorem \ref{three}(1), that is, for a Zariski dense $\theta$-discrete $\Ga < G$,  $\psi_\Ga^\theta$ does not take $+\infty$-value. This will be achieved by proving $\delta_\Ga^\theta<\infty$ (Proposition \ref{prop.propertaufinite})
where $$ -\infty\le \delta^{\theta}_\Ga\le \infty $$ denotes the abscissa of convergence of the series $s\mapsto \sum_{\ga\in \Ga} e^{-s\|\mu_\theta(\ga)\|}$. For $\theta=\Pi$, we have $0<\delta_\Ga=\delta_\Ga^{\Pi}<\infty$
\cite[Theorem 4.2.2]{Quint2002divergence}. Since $\|\mu_\theta(g)\|\le \|\mu(g)\|$ for all $g\in G$ and hence
$\sum_{\ga\in \Ga}e^{-s\|\mu(\ga)\|} \le \sum_{\ga\in \Ga}
e^{-s\|\mu_\theta(\ga)\|}$ for all $s\ge 0$, we have
 \be\label{po} 0<\delta_\Ga\le \delta^{\theta}_\Ga .\ee

\begin{lem}
 If $\Ga$ is Zariski dense and  $\theta$-discrete, 
 then $$\delta_\Ga^{\theta}=\limsup_{t\to \infty}\frac{1}{t}\#\log
 \{\ga\in \Ga:\|\mu_\theta(\ga)\|<t\}\in (0, \infty].$$
\end{lem}
\begin{proof}
For $x \in \fa_{\theta}$, we denote by $D_x$ the Dirac mass at $x$. Since $\sum_{\ga\in \Ga}D_{\mu_\theta(\gamma)}$
 is a Radon measure on $\fa_\theta^+$ and  $\delta_\Ga^{\theta}>0$ by \eqref{po}, it follows from 
\cite[Lemma 3.1.1]{Quint2002divergence}. 
\end{proof} 

For a general discrete subgroup $\Ga<G$, $\delta_\Ga^{\theta}$ may be infinite (e.g., $\Ga=\Ga_1\times \Ga_2$ where $\Ga_i$ is an infinite discrete subgroup of $G_i$ for both $i=1,2$ and $\theta$ consists of simple roots of $G_1$). Since
$\tau_{\mathcal C}^{\theta}\le \delta_\Ga^\theta$ for all cones $\mathcal C$
in $\fa_\theta$, we have
 $$\sup_{u\in \fa_\theta, \|u\|=1} \psi_\Ga^{\theta}(u)\le \delta_\Ga^\theta.$$ Hence Theorem \ref{three}(1) follows once we show  the that $\delta_\Ga^{\theta}<\infty$ for any $\theta$-discrete subgroup $\Ga<G$ as in Proposition \ref{prop.propertaufinite}.

\begin{figure}[h]
\begin{tikzpicture}[scale=0.8, every node/.style={scale=0.8}]
    \filldraw[draw = white, fill=blue!10, opacity=0.3] (0, 0) -- (3, 0) arc(0:60:3) -- (0, 0);
    \draw (-3, 0) -- (3, 0);
    \draw (-1.5, -2.5980762114) -- (1.5, 2.5980762114);
    \draw (-1.5, 2.5980762114) -- (1.5, -2.5980762114);

    \draw (2, 2.1) node {$\fa^+$};

    \draw (3, 0) node[right] {$\ker \alpha_2$};
    \draw (1.5, 2.5980762114) node[above] {$\ker \alpha_1$};

    \filldraw[blue] (1, 0) circle(1pt);
    \draw (1, 0) node[below] {$u$};

    \draw[blue, thick] (1, 0) -- (1, 1.7320508076);
    \draw[blue] (1, 0.9) node[right] {$p_{\alpha_1}^{-1}(u) \cap \fa^+$};
    
\end{tikzpicture}
\caption{$G = \PSL_3(\R)$ and $\theta = \{\alpha_1\}$.} \label{fig}
\end{figure}

\begin{lem}\label{thetapp}
  If $p_\theta|_{\fa^+}$ is a proper map
  (e.g., $G$ is simple), then $$\delta^{\theta}_\Ga <\infty$$ for any discrete subgroup $\Ga<G$. In particular, if $G$ is simple, any discrete subgroup $\Ga<G$ is $\theta$-discrete.
\end{lem}
\begin{proof} First, observe that if $G$ is simple, then
 the angle between any two walls of $\fa^+$ is strictly smaller than $\pi/2$ and hence $p_\theta|_{\fa^+}$ is a proper map (see Figure \ref{fig}).
 Now,  if $p_\theta|_{\fa^+}$ is a proper map, then for some constant $C>1$, we have
$$C^{-1}\|u\|\le \|p_\theta(u)\| \le C \|u\|$$ for all $u\in \fa^+$. 
Hence $\delta_\Ga<\infty$ implies that $$\delta^{\theta}_\Ga <\infty.$$ 
\end{proof}

 It follows from the definition of $\delta_\Ga^\theta$ that the finiteness of $\delta_\Ga^\theta$
implies the $\theta$-discreteness of $\Ga$. Indeed, the converse holds as well from which Theorem \ref{three}(1) follows.
\begin{prop} \label{prop.propertaufinite}
 We have $$\Ga \mbox{ is } \theta \mbox{-discrete} \quad \mbox{if and only if} \quad \delta_\Ga^{\theta}<\infty.$$
\end{prop}
\begin{proof} 
It suffices to show that the $\theta$-discreteness of $\Ga$ implies $\delta_{\Ga}^{\theta} < \infty$.
Write $G=G_1G_2$ as an almost direct product
of semisimple real algebraic groups where $G_1$ is the smallest group such that
$\theta$ is contained in the set of simple roots for 
$(\frak g_1, \fa_1^+=\fa^+\cap \frak g_1)$.
Then $\mu_\theta(\Gamma)\subset \fa_\theta^+ \subset \fa_1^+$.
Since the kernel of $p_\theta|_{\mu(\Ga)}$ contains
 $\mu(\Ga\cap (\{e\}\times G_2))$, the properness hypothesis implies that $\Ga\cap (\{e\}\times G_2)$ is finite. By passing to a subgroup of finite index, we may assume that $\Ga\cap (\{e\}\times G_2)$ is trivial.
 The properness of
 $\mu_\theta|_{\Ga}$ also implies that
the projection of $\Ga$ to $G_1$ is a discrete subgroup, which we denote by $\Ga_1$. Since there exists a unique element, say, $\sigma(\ga_1)\in G_2$ such that $(\ga_1, \sigma(\ga_1))\in \Ga$ for each $\ga_1\in \Ga_1$, we get a faithful representation $\sigma:\Ga_1\to G_2$, and  $\Ga$ is of the form
 $\{(\ga_1, \sigma(\ga_1)):\ga\in  \Ga_1 \}$.
Since $\mu_\theta(\ga)=\mu_\theta(\ga_1)$ for $\ga=(\ga_1, \sigma(\ga_1))\in \Ga$, we have
$$\delta_\Ga^\theta=\delta_{\Ga_1}^\theta.$$
Hence we may assume without loss of generality that
$\theta$ contains at least one root of each simple factor of $G$. Since the restriction $p_\theta: \fa^+\cap \operatorname{Lie} G_0 \to \fa_\theta\cap \operatorname{Lie} G_0$ is a proper map for each simple factor $G_0$ of $G$ as mentioned before, it follows that
$p_\theta$ is a proper map. Hence the claim $\delta_\Ga^\theta<\infty$ follows by Lemma \ref{thetapp}.
\end{proof}

\begin{rmk} We remark that the
    $\theta$-discreteness of $\Ga$ does not necessarily imply that the map $p_\theta|_\L$ is a proper map. For example, let $\Ga_0$ be a Zariski dense and convex cocompact subgroup of $\SO^\circ(k,1)$, $k\ge 2$, and let $\sigma:\Ga_0\to \SO^\circ(k,1)$ be a discrete faithful representation such that $\sigma(\Ga_0)$ is Zariski dense but not convex cocompact. Consider $\Ga=\{(g, \sigma(g)):g\in \Ga_0\}$ and
    $G= \SO^\circ(k,1)\times \SO^\circ(k,1)$. We may identify $\fa=\{(x_1, x_2)\in \br^2\}$ and
    $\fa^+=\br_{\ge 0}\times \br_{\ge 0}$.
    Then the limit cone of $\Ga$ is a convex cone of $\fa^+$ containing the $x_1$-axis; otherwise, $\sigma$ must be convex cocompact. Hence for $\theta=\{\alpha_2\}$ where $\alpha_2(x_1, x_2)=x_2$, the fiber $p_\theta^{-1}(0)$ is the whole $x_1$-axis, and hence $p_\theta|_\L$ is not proper. On the other hand,
    the discreteness of $\sigma(\Ga_0)$ is same as $\theta$-discreteness of $\Ga$. 
 \end{rmk}

\subsection*{Concavity of $\psi_\Ga^\theta$} The growth indicator $\psi_\Ga^\theta$ is clearly a homogeneous and upper semi-continuous function \cite[Lemma 3.1.7]{Quint2002divergence}.  It is also a concave function, but its proof requires the following lemma, which is proved in \cite[Proposition 2.3.1]{Quint2002divergence} for $\theta=\Pi$.
\begin{lemma} \label{lem.concavemeas}
    Suppose that $\Ga$ is Zariski dense and $\theta$-discrete. Then there exists a map $\pi : \Ga \times \Ga \to \Ga$ satisfying the following: \begin{enumerate}
        \item there exists $\kappa \ge 0$ such that for every $\ga_1, \ga_2 \in \Ga$, $$\|\mu_{\theta}(\pi(\ga_1, \ga_2)) - \mu_{\theta}(\ga_1) - \mu_{\theta}(\ga_2) \| < \kappa; \ \mbox{and}$$
        \item for every $R \ge 0$, there exists a finite subset $H$ of $\Ga$ such that for $\ga_1, \ga_1', \ga_2, \ga_2' \in \Ga$ with $\|\mu_{\theta}(\ga_i) - \mu_{\theta}(\ga_i') \| \le R$ for $i = 1, 2$, $$\pi(\ga_1, \ga_2) = \pi(\ga_1', \ga_2') \Rightarrow \ga_1' \in \ga_1 H \mbox{ and } \ga_2' \in H \ga_2.$$
    \end{enumerate}
\end{lemma}

\begin{proof}
    Since $p_{\theta}$ is norm-decreasing, (1) follows from \cite[Proposition 2.3.1(1)]{Quint2002divergence}. 
    By the proof of \cite[Proposition 2.3.1(2)]{Quint2002divergence}, the claim (2) holds if we set $H$ to be the subset consisting of all elements $\ga \in \Ga$ such that $\mu_{\theta}(\ga) < R'$ for some $R' > 0$ depending only on $R$. Since $\Ga$ is $\theta$-discrete, this subset $H$ is finite, as desired.
\end{proof}

\begin{prop}\label{st} 
If $\Ga$ is Zariski dense and $\theta$-discrete, then $\psi_\Ga^\theta$ is concave, and hence there exists a unique unit vector $u_{\Ga}^{\theta} \in \fa_{\theta}^+$ such that 
$$\psi_{\Ga}^{\theta}(u_{\Ga}^{\theta}) = \max_{\|u \| = 1, u\in \fa_\theta^+} \psi_{\Ga}^{\theta}(u) = \delta_{\Ga}^{\theta}.$$
\end{prop}
\begin{proof}
By Lemma \ref{lem.concavemeas}, the counting measure $\sum_{\ga \in \Ga} D_{\mu_{\theta}(\ga)}$ is of concave growth (see \cite[Section 3.2]{Quint2002divergence} for details).  It follows from \cite[Theorem 3.2.1]{Quint2002divergence} that $\psi_{\Ga}^{\theta}$ is concave. By \cite[Corollary 3.1.4, Corollary 3.3.5]{Quint2002divergence},
the second claim follows. 
\end{proof}

\begin{Def} \label{tangg} A linear form $\psi\in \fa_\theta^*$ is said to be tangent to $\psi_\Ga^\theta$ (at $u\in \fa_\theta^+-\{0\}$) if  $\psi \ge \psi_{\Ga}^{\theta}$ on $\fa_\theta^+$ and $\psi(u) = \psi_{\Ga}^{\theta}(u)$.
\end{Def}
By the supporting hyperplane theorem, we have the following corollary:

\begin{cor} \label{cor.tangentinterior}
    Let $\Ga < G$ be Zariski dense and $\theta$-discrete. For any $u\in \inte\L_\theta$, there exists a linear form $\psi \in \fa_{\theta}^*$ tangent to $\psi_{\Ga}^{\theta}$ at $u$.
\end{cor}

\subsection*{Positivity of $\psi_\Ga^\theta$}
By Lemma \ref{tcone}, we have $\psi_\Ga^{\theta}=-\infty$ outside $\L_\theta$.
If $\Theta\supset \theta$, then any $\theta$-discrete
$\Ga$ is $\Theta$-discrete as well. The following lemma
shows how $\psi_\Ga^{\theta}$ is related to $\psi_\Ga^{\Theta}$ from which Theorem \ref{three}(3) follows:

\begin{lemma} \label{qu}
       For $\Theta\supset \theta$, let $p_\theta=p_\theta|_{\fa_{\Theta}}:\fa_{\Theta}\to \fa_\theta$ by abuse of notation.
 For any Zariski dense $\theta$-discrete $\Ga<G$, we have  \be \label{eqn.growthrel1}
    \psi_{\Ga}^{\theta}\circ p_\theta \ge \psi^{\Theta}_{\Ga}\quad\text{on } \fa_\Theta.
    \ee

    In particular, \be \label{eqn.jul142025}
    \psi^\theta_\Ga \ge 0 \text{ on $ \L_\theta$}\quad\text{ and }\quad   \psi^\theta_\Ga > 0 \text{ on $\inte \L_\theta$}.\ee
    \end{lemma}

    \begin{proof}
    Note that \eqref{eqn.jul142025} follows from \eqref{pg} and \eqref{eqn.growthrel1}.
    By homogeneity,
    it suffices to prove \eqref{eqn.growthrel1} for all $v \in p_{\theta}^{-1}(u) \cap \fa_{\Theta}$, where  $u\in \L_\theta$ is an arbitrary unit vector.
    Let  $v \in p_{\theta}^{-1}(u) \cap \fa_{\Theta}$. 
    Let $\C \subset \fa_{\theta}$ be an open cone containing $u$.   For each $\e > 0$,
 set 
    \be \label{eqn.smallcone}
    \C (v, \e) := \left\{ w \in \fa_\Theta : p_{\theta}(w) \neq 0 \mbox{ and } \left\|  \tfrac{w}{\|p_{\theta}(w)\|} - v \right\|  < \e \right\}.
    \ee
    Since $\|p_\theta(v)\|=\|u\|=1$, $\C(v, \e)$
    is an open cone containing $v$.
   In the following, let $\e>0$ be small enough so that 
     $\C(v, \e) \subset p_{\theta}^{-1}(\C)$.

Then for all $s \in \R$, we have $$\begin{aligned}
        \sum_{\ga \in \Ga, \mu_\Theta(\ga) \in \C(v, \e)} e^{-s \| \mu_\Theta(\ga) \|} & \le \sum_{\ga \in \Ga, \mu_\Theta(\ga) \in \C(v, \e)} e^{- \left( s{\|v\| } - |\e s|  \right) \| \mu_{\theta}(\ga) \|} \\
        & \le \sum_{\ga \in \Ga, \mu_{\theta}(\ga) \in \C} e^{- \left( s{\|v\| } - |\e s| \right) \| \mu_{\theta}(\ga) \|}.
    \end{aligned}$$
    Hence we have $$\tau^\Theta_{\C(v, \e)} \le \left( {\|v\| } - \e \right)^{-1} \tau^{\theta}_{\C}.$$
  Therefore we have $$\psi_{\Ga}^\Theta(v) \le \|v\| \tau^\Theta_{\C(v, \e)} \le \|v\|\left( {\|v\| } - \e \right)^{-1} \tau^{\theta}_{\C}.$$

  Taking $\e \to 0$ yields that $$\psi_{\Ga}^\Theta(v) \le \tau_{\C}^{\theta}.$$ Since $\C\subset\fa_\Theta$ is an arbitrary open cone in $\fa_{\theta}$ containing $u$, it follows that $$\psi^\Theta_{\Ga}(v) \le \psi_{\Ga}^{\theta}(u),$$ and hence \eqref{eqn.growthrel1} is proved. 
  Last claim follows the from \eqref{pg} and \eqref{eqn.growthrel1} applied to $\Theta=\Pi$.
\end{proof}

\subsection*{Comparison between $\psi_\Ga^\theta$ and $\psi^\Theta_\Ga$}
 Note that for a discrete subgroup $\Ga < G$,
the properness of $p_{\theta}|_{\L_\theta}$ implies the $\theta$-discreteness of $\Ga$ as $\mu(\Ga)$
is within a bounded distance from $\L$. The following lemma is to appear in \cite{Quint_inprep} in a more general context. 
\begin{lemma} \label{qu2} 
Let $\Ga < G$ be a Zariski dense discrete subgroup.
   If  $p_\theta|_{\L}$ is a proper map (e.g., $G$ is simple), then for any $\Theta\supset \theta$ and for any $u \in \fa_{\theta}$,
   \be \label{eqn.growthrel2}
    \psi_{\Ga}^{\theta}(u) = \max_{v \in p_{\theta}^{-1}(u)} \psi_{\Ga}^\Theta(v)
    \ee where $p_\theta=p_\theta|_{\fa_\Theta}$ by abuse of notation.
    \end{lemma}
   \begin{proof}  Suppose that $p_\theta|_{\L}:\L\to \fa_\theta$ is a proper map. By Lemma \ref{qu}, 
   it suffices to consider a unit vector $u\in \L_\theta$ with $\psi_\Ga^{\theta}(u)>0$.
   Since $p_{\theta}^{-1}(u)\cap \L_\Theta$ is a compact subset and  $\psi^\Theta_\Ga$ is upper semi-continuous, we have
   $$\sup_{v \in p_{\theta}^{-1}(u)} \psi_{\Ga}^{\Theta}(v)=
   \max_{v \in p_{\theta}^{-1}(u)\cap \L_\Theta} \psi^\Theta_{\Ga}(v).$$
 
 For all sufficiently small $\e>0$ and each $v \in p_{\theta}^{-1}(u)$,
    there exists $0 < \e_v < \e$ such that 
    \be \label{eqn.choiceofsmallcone}
    \|v\| \tau^\Theta_{\C(v, \e_v)} < \psi^\Theta_{\Ga}(v) + \e
    \ee
    where $\C(v, \e_v)$ is as defined in \eqref{eqn.smallcone}.  
   Since $p_{\theta}^{-1}(u) \cap \L_\Theta$ is compact, there exist  $v_1, \cdots, v_n \in p_{\theta}^{-1}(u)$ such that $$p_{\theta}^{-1}(u) \cap \L_\Theta \subset \bigcup_{i = 1}^n \C(v_i, \e_{v_i}).$$
    Take an open cone $\C \subset \fa_{\theta}$ containing $u$ such that 
    $$p_{\theta}^{-1}(u) \cap \L_\Theta \subset p_{\theta}^{-1}(\C) \cap \L_\Theta \subset \bigcup_{i = 1}^n \C(v_i, \e_{v_i}).$$
    This is indeed possible; if not, there is a sequence of unit vectors $u_j \in \fa_{\theta}$ converging to $u$ as $j \to \infty$ such that for each $j$, there exists $w_j \in p_{\theta}^{-1}(u_j) \cap \L_\Theta$ that does not belong to $\bigcup_{i = 1}^n \C(v_i, \e_{v_i})$. Since $p_{\theta}|_{\L_\Theta}$ is proper and the unit sphere in $\fa_{\theta}$ is compact, we may assume that the sequence $w_j$ converges to some $w \in \L_\Theta$ after passing to a subsequence. Since $p_{\theta}(w_j) = u_j \to u$ as $j \to \infty$, we have $p_{\theta}(w) = u$, and hence $w \in p_{\theta}^{-1}(u) \cap \L_\Theta$. It implies that $w_j \in \bigcup_{i = 1}^n \C(v_i, \e_{v_i})$ for all large $j$, contradiction.

    Since $\mu_\Theta(\Ga)$ is within a bounded distance from $\L_\Theta$ (Lemma \ref{tcone}), there are only finitely many elements of $\mu_{\Theta}(\Ga) \cap p_{\theta}^{-1}(\cal C)$ outside of $\bigcup_{i=1}^n \cal C(v_i, \e_{v_i})$. 
    Hence for each $s \ge 0$, we have
    $$\begin{aligned}
        \sum_{\ga \in \Ga, \mu_{\theta}(\ga) \in \C} e^{-s \| \mu_{\theta}(\ga) \|}  &\ll  \sum_{i = 1}^n \sum_{\ga \in \Ga, \mu_\Theta(\ga) \in \C(v_i, \e_{v_i})} e^{-s \| \mu_{\theta}(\ga) \|}\\
        & \le \sum_{i = 1}^n \sum_{\ga \in \Ga, \mu_\Theta(\ga) \in \C(v_i, \e_{v_i})} 
        e^{-s \frac{ \| \mu_{\Theta}(\ga) \|}{\| v_i\| + \varepsilon_{v_i}}}.
    \end{aligned}$$
Here and afterwards, the notation $f(s)\ll g(s)$ means that for some uniform constant $C\ge 1$, $f(s) \le C g(s)$ for all $s$ at hand. Since $\tau_{\cal C}^\theta \ge \psi_\Ga^{\theta}(u)>0$ is positive,
  it follows that $$\tau_{\C}^{\theta} \le \max_i  ( \| v_i \| + \varepsilon_{v_i}) \tau^\Theta_{\C(v_i, \e_{v_i})}.$$
    Therefore, together with $0 < \e_{v_i} < \e$ and \eqref{eqn.choiceofsmallcone}, we get $$\begin{aligned}
        \psi_{\Ga}^{\theta}(u)  \le \tau_{\C}^{\theta} & \le  \frac{\| v_i \| + \varepsilon_{v_i}}{\| v_i\|} \left( \max_i \psi^\Theta_{\Ga}(v_i) + \e \right)  \le  \frac{\| v_i \| + \varepsilon}{\| v_i\|} \left( \max_{v \in p_{\theta}^{-1}(u)} \psi^\Theta_{\Ga}(v) + \e \right).
    \end{aligned}$$ Since $0 < \e < 1$ was arbitrary, this proves the claim by Lemma \ref{qu}.
    \end{proof}

\begin{example} \label{hit}
We discuss some  explicit upper bounds for $\psi_\Ga^\theta$
when $G=\PSL_d(\br)$. Identify $\fa^+=\{(t_1, \cdots , t_d): t_1\ge \cdots \ge t_d, t_1+\cdots +t_d=0\}.$
Let 
$\alpha_i(t_1, ... , t_d)=t_i-t_{i+1}$ for $i=1,2, ..., d-1$. 
 Let $$w_i=  \left(\tfrac{d-i}{d}, \cdots ,\tfrac{d-i}{d},-\tfrac{i}{d}, \cdots ,-\tfrac{i}{d}\right),$$ where the first $i$ coordinates are ${d - i \over d}$'s and the last $d-i$ coordinates are $-\frac{i}{d}$'s, so that $\fa_{\alpha_i} = \br w_i $ and $\alpha_i(w_i)=1$.
 We compute that
   $$p_{\alpha_i}(t_1, \cdots, t_d) = {d(t_1 + \cdots + t_i) \over i(d-i)}  w_i $$ and hence $$p_{\alpha_i}^{-1}(w_i) \cap \fa^+ = \{ (t_1, \cdots, t_d) \in \fa^+: d(t_1 + \cdots + t_i) = i(d-i) \}.$$

For any non-lattice discrete subgroup $\Ga<\PSL_d(\br)$, we have
\be\label{pssss} \psi_\Ga(t_1, \cdots, t_d) \le \sum_{i<j} (t_i-t_j) -\frac{1}{2} \sum_{i=1}^{\lfloor d/2\rfloor} (t_i-t_{d+1-i}) \ee 
by (\cite{quint_ka}, \cite{oh}, \cite[Theorem 7.1]{lee2022dichotomy}).  By Lemma \ref{qu2},  for any discrete non-lattice subgroups, we get 
  \be\label{upper} \psi_\Ga^{\alpha_i} (w_i) \le \max \sum_{i<j} (t_i-t_j) -\frac{1}{2} \sum_{i=1}^{\lfloor d/2\rfloor} (t_i-t_{d+1-i}) \ee 
  where the maximum is taken over all
  $(t_1, \cdots, t_d) \in \fa^+$ such that $d(t_1 + \cdots + t_i) = i(d-i)$. 

For instance, for $d=3$, the right hand side is always $3$ and hence for each $i=1,2$,
$\psi_\Ga^{\alpha_i}  \le 3 \alpha_i $ on $\br w_i$. 
\end{example}

\subsection*{Hitchin subgroups} Let $\iota : \PSL_2(\br)\to \PSL_d( \br)$ be the irreducible representation, which is unique up to conjugations. A Hitchin subgroup is the image of a representation $\pi: \Sigma \to \PSL_d(\br)$ of a non-elementary torsion-free geometrically finite subgroup $\Sigma<\PSL_2(\br)$, which is a type-preserving deformation of $\iota|_{\Sigma}$.
Hitchin subgroups are $\Pi$-transverse, as defined in the introduction, by
\cite{Canary2022cusped} and hence $\alpha_i$-discrete for each $i=1, \cdots, d-1$.
For a Zariski dense Hitchin subgroup $\Ga$, it follows from  Lemma \ref{lem.tent} that if $\delta_{\alpha_i}$ denotes the abscissa of convergence of $s \mapsto \sum_{\ga \in \Ga} e^{-s \alpha_i(\mu(\ga))}$,
then
$$\psi_\Ga^{\alpha_i}(w_i) \le \delta_{\alpha_i} \cdot \alpha_i(w_i)=\delta_{\alpha_i} .$$
For Hitchin subgroups, it was proved by Potrie and Sambarino \cite{Potrie2017eigenvalues} for $\Delta$ cocompact and Canary, Zhang and Zimmer \cite{canary2022entropy} for $\Delta$ geometrically finite
 that $$\delta_{\alpha_i} \le 1$$ for all $i$ (see also \cite{pozzetti2019anosov}). Hence $\max_{1\le i\le d-1} \psi_\Ga^{\alpha_i}(w_i) \le 1$.
We get a sharper bound in the following:

\begin{cor}\label{hit1} Let $\Ga<\PSL_d(\br)$ be a Zariski dense Hitchin subgroup. For each $i=1, \cdots, d-1$,
    $$\psi_\Ga^{\alpha_i}  < \frac{\max (i, d - i)}{d-1}\alpha_i \quad\text{on $ \fa_{\alpha_i}-\{0\}$} .$$
   
\end{cor}
\begin{proof}
   For a Zariski dense
    Hitchin subgroup $\Ga<G$, it is shown in \cite[Corollary 1.10]{kim2021tent} that
\be\label{kmo} \psi_\Ga (t_1, \cdots, t_d)< \frac{1}{d-1} (t_1 - t_d )\quad\text{for $(t_1, \cdots, t_d)\in \fa^+-\{0\}$} .\ee 
Indeed, \cite[Corollary 1.10]{kim2021tent} is stated only for $\Sigma$ cocompact. However in view of \cite{canary2022entropy} mentioned above, this bound holds for a general Hitchin subgroup.
 Hence by Lemma \ref{qu2}, we get 
\be \label{eqn.hitchinupper}
\psi_\Ga^{\alpha_i}(w_i)< \frac{1}{d-1}
\max { (t_1 - t_d) }
\ee
where the maximum is taken over all
  $t_1\ge \cdots \ge  t_d$ such that  $d\sum_{j=1}^i t_j = i(d-i)$ and $\sum_{j=1}^d t_j=0$.  Suppose that this maximum is realized at $(t_1, \cdots, t_d)$. Since $t_1-t_d$ does not involve any $t_j$, $2\le j\le d-1$, we may assume that $ t_2 = \cdots = t_i$ and $t_{i+1} = \cdots = t_{d-1}$, which we denote by $x$ and $y$ respectively.
  Since $\sum_{j = 1}^i t_j = \tfrac{i(d-i)}d$ and $\sum_{j = i+1}^d t_j = -\tfrac{i(d-i)}d$, we then have  $$ t_1 = \tfrac{i (d-i)}{d} - (i-1)x \quad \mbox{and} \quad t_d = - \tfrac{i(d-i)}{d} - (d-1-i)y.$$
  Therefore \be\label{m8}
  t_1 - t_d = \tfrac{2i (d-i)}{d} - ((i-1)x - (d-1-i)y).
  \ee
  It follows from  $t_j \ge t_{j+1} $ for all $j$ that ${d-i \over d} \ge x \ge y \ge -{i \over d}$.  Therefore, for each fixed $x$, the maximum in \eqref{m8} is obtained when $y = x$. Hence we have $$\begin{aligned}
    \psi_{\Ga}^{\alpha_i}(w_i) & < {1 \over d - 1} \max_{x \in [-i/d, (d-i)/d]} {2i (d-i) \over d} - (2i - d)x \\
    & = {1 \over d-1} \max ( i, d- i ).
\end{aligned}$$
\end{proof}

\section{On the proper and critical linear forms} \label{sec.linforms}
Let $\Ga$ be a $\theta$-discrete infinite subgroup of $G$.
\begin{Def}\label{properl}
A linear form $\psi\in \fa_\theta^*$ is called $(\Ga, \theta)$-proper if  $\op{Im} (\psi\circ \mu_\theta)  \subset [-\e, \infty)$ and  $\psi\circ \mu_\theta:\Ga\to [-\e , \infty)$
is  proper for some $\e>0$.
\end{Def}

Consider the series $\cal P_\psi=\cal P_{\G, \psi}$ given by
 \be\label{ppsi} {\cal P}_\psi(s)=\sum_{\ga \in \Ga} e^{-s \psi(\mu_{\theta}(\ga))}.\ee 

The abscissa of convergence of $\cal P_\psi$ is well-defined for a $(\Ga, \theta)$-proper linear form:

 \begin{lem}\label{wd} If $\psi$ is $(\Ga,\theta)$-proper, the following $\delta_\psi=\delta_\psi(\Ga)$ is well-defined (possibly $+\infty$):
 \be\label{ac}  \delta_\psi:=\sup \{ s \in \R : \cal P_{\psi}(s) = \infty \} = \inf \{ s \in \R : \cal P_{\psi}(s) < \infty \}\in [0, \infty].  \ee 
 Moreover, if $\Ga$ is Zariski dense, then $$0 <  \delta_\psi=\limsup_{t\to \infty}\frac{\log \#\{\ga\in \Ga:\psi(\mu_\theta(\ga))\le t\}}{t}.$$
\end{lem}
 \begin{proof}  Since $\psi$ is $(\Ga, \theta)$-proper, $\psi(\mu_{\theta}(\ga)) > 0$ for all but finitely many $\ga \in \Ga$.  Hence
 we may replace $\cal P_\psi(s)$ by the series
 ${\cal P}^+_\psi(s)=\sum_{\ga \in \Ga, \psi(\mu_\theta(\ga))>0} e^{-s \psi(\mu_{\theta}(\ga))}$
 in proving this claim.  Since ${\cal P}^+_\psi(s)$ is a decreasing function of $s\in \br$, 
 $I_1:=\{ \cal P_{\psi}(s) = \infty \} $
 and $ I_2:=\{\cal P_{\psi}(s) < \infty \}$ are disjoint intervals. 
 Since $\Ga$ is infinite, $0\in I_1$, and hence $\delta_\psi=\overline I_1\cap \overline I_2 \in [0, \infty]$ is well-defined. 
 
 Now suppose that $\Ga$ is Zariski dense. By Lemma \ref{tcone}, $\inte \L_{\theta} \neq \emptyset$. To show $\delta_\psi>0$,
 fix $u \in \inte \L_{\theta}$.
     Then $\psi(u) > 0$ by Lemma \ref{lem.propform}. Since $\psi_{\Ga}^{\theta}(u) > 0$ as well by Theorem \ref{three}(3), we have $s_0 \psi(u) < \psi_{\Ga}^{\theta}(u)$ for some $0< s_0 <\infty$. By \cite[Lemma 3.1.3]{Quint2002divergence}, we have $\cal P_{\psi}(s_0) = \infty$, and therefore $\delta_\psi\ge s_0>0$. The last claim follows by \cite[Lemma 3.1.1]{Quint2002divergence} since the counting measure on $\psi(\mu_\theta(\Ga))$ is locally finite and
     $\delta_\psi>0$.
     \end{proof}
Hence for a $(\Ga,\theta)$-proper form $\psi \in \fa_{\theta}^*$, 
 $0<\delta_\psi\le \infty $ 
is the abscissa of convergence of $\cal P_\psi(s)$.

\begin{lem} \label{lem.propform}
We have:
\begin{enumerate}
      \item 
 If $\psi>0$ on $\L_\theta- \{0\}$, then $\psi$ is $(\Ga, \theta)$-proper and $\delta_\psi<\infty$. 

\item If $\psi$ is $(\Ga, \theta)$-proper, then 
$\psi\ge 0$ on $\L_\theta$ and $\psi>0$ on $\inte\L_\theta$.
 \end{enumerate} 
\end{lem}
\begin{proof}
 If $\psi$ is positive on $\L_\theta - \{0\}$, then
    $\psi>0$ on some open cone $\cal C$ containing $\L_\theta - \{0\}$. Then
      for some $c>1$, $c^{-1} \|u\|\le \psi(u)\le c \|u\|$ for all $u\in \cal C$.
     Since there can be only finitely many points of $\mu_{\theta}(\Ga)$ outside $\cal C$ by Lemma \ref{tcone}, this implies that $\psi$ is $(\Ga, \theta)$-proper.
Since  $\delta_\Ga^{\theta}<\infty$ by Proposition \ref{prop.propertaufinite},
  we also have $\delta_\psi<\infty$.
  
    To prove (2), suppose to the contrary that $\psi(u) < 0$ for some $u \in \L_{\theta}$. Then there exists an open cone $\C \subset \fa_{\theta}$ containing $u$ so that $\psi < 0$ on $\C-\{0\}$. In particular, there are infinitely many $\ga_i \in \Ga$ such that $\psi(\mu_{\theta}(\ga_i)) < 0$, which contradicts $(\Ga, \theta)$-properness of $\psi$. Therefore, $\psi \ge 0$ on $\L_{\theta}$. Since $\ker \psi$ is a hyperplane in $\fa_{\theta}$, 
     it follows  $\psi > 0$ on $\inte \L_{\theta}$.
 \end{proof}

 \subsection*{Critical forms} Analogous to the critical exponent
 of a discrete subgroup of a rank one Lie group, we define:
\begin{Def} \label{defcrit} A linear form $\psi \in \fa_{\theta}^*$ is $(\Ga, \theta)$-critical if it is tangent to $\psi_{\Ga}^{\theta}$.
\end{Def}

The following lemma  can be proved by adapting the proof of \cite[Theorem 2.5]{kim2021tent} replacing $\psi_\Ga$ by $\psi_\Ga^\theta$.
\begin{lemma}  \label{lem.tent} 
Suppose that $\Ga$ is Zariski dense. If  a $(\Ga, \theta)$-proper $\psi \in \fa_{\theta}^*$ satisfies $\delta_{\psi} < \infty$, then
     $\delta_{\psi} \psi$ is $(\Ga, \theta)$-critical; in particular, $$\psi_\Ga^\theta\le \delta_{\psi}\psi.$$
\end{lemma}

\begin{proof} Suppose that $\delta_{\psi} < \infty$. By Lemma \ref{wd}, $\delta_\psi>0$.
We first claim 
\be\label{first} \psi^\theta_{\Ga}(v) \le \delta_{\psi} \psi(v) \quad \mbox{for all } v \in  \inte \L_\theta.\ee  Fix 
$v \in \inte \L_\theta$
and $\varepsilon > 0$. Since $\psi$ is $(\Ga, \theta)$-proper, $\psi(v)>0$ by Lemma \ref{lem.propform}.

We then consider 
$$\cal C_{\varepsilon}(v) = \left\{w \in \fa_{\theta} :  \psi(w) > 0 \mbox{ and } \left\| \frac{w}{\psi(w)} - \frac{v}{\psi(v)} \right\| < \varepsilon  \right\};$$ 
since $\psi(v)>0$, this is a well-defined 
open cone containing $v$.
Therefore by the definition of $\psi^\theta_\Ga$,
we have 
\be\label{tau} \psi^{\theta}_{\Ga}(v) \le \|v\| \tau_{\cal C_{\varepsilon}(v)}^{\theta}.\ee

Observe that for any $s\ge 0$,
$$\begin{aligned}
\sum_{\ga \in \Ga, \mu_\theta(\ga) \in \cal C_{\varepsilon}(v)} e^{-s \| \mu_\theta(\ga)\|} & \le \sum_{\ga \in \Ga, \mu_\theta(\ga) \in \cal C_{\varepsilon}(v)} e^{-s \psi(\mu_\theta(\ga)) \left( \frac{\|v\|}{\psi(v)} - \varepsilon \right)} \\
& \le \sum_{\ga \in \Ga}  e^{-s \psi(\mu_\theta(\ga)) \left( \frac{\|v\|}{\psi(v)} - \varepsilon \right)}.
\end{aligned}$$
It follows from the definitions of $\tau^\theta_{\cal C_{ \varepsilon}(v)}$ and $\delta_\psi$ that 
$$\tau^\theta_{\cal C_{ \varepsilon}(v)} \le \frac{\delta_{\psi}}{\|v\|{\psi(v)}^{-1} - \varepsilon} = \frac{\delta_{\psi}\psi(v)}{\|v\| - \varepsilon \psi(v)},$$
and hence $$\psi^\theta_{\Ga}(v) \le \|v\| \frac{\delta_{\psi}\psi(v)}{\|v\| - \varepsilon \psi(v)}.$$
Since $\varepsilon > 0$ is arbitrary, we get $\psi^\theta_{\Ga}(v) \le \delta_{\psi}\psi(v)$, proving the claim \eqref{first}.

We now claim that the inequality \eqref{first} also holds
 for any $v$ in the boundary $\partial \L_\theta$.
 Choose any $v_0 \in \inte \L_\theta$. From the concavity of $\psi^\theta_{\Ga}$ (Theorem \ref{three}), we have $$t\psi_\Ga^\theta(v_0) + (1-t)\psi_\Ga^\theta(v) \le \psi_\Ga^\theta(tv_0 + (1-t)v) \quad \mbox{for all } 0 < t < 1.$$ Since $\L_\theta$ is convex, $t v_0 + (1-t) v \in \inte \L_\theta$ for all $0 < t < 1$. As we have already shown $\psi^\theta_\Ga \le \delta_\psi \psi$ on $\inte\L_\theta$, we get $$t \psi_\Ga^\theta(v_0) + (1 - t) \psi_\Ga^\theta(v) \le \delta_\psi \psi(tv_0 + (1-t)v) \quad \mbox{for all } 0 < t < 1.$$ By sending $t \to 0^+$, we get $$\psi_\Ga^\theta(v) \le \delta_\psi \cdot \psi(v). $$ 

Since $\psi^\theta_\Ga=-\infty$ outside $\L_\theta$, we have established
$\psi^\theta_\Ga\le \delta_\psi  \psi$.  Suppose that
$\psi^\theta_\Ga <\delta_\psi  \psi $ on $\fa - \{0\}$.
Then the abscissa of convergence of the series $s\mapsto \sum_{\ga \in \Ga} e^{-s \delta_{\psi} \psi(\mu_\theta(\ga))}$ is strictly less than $1$ by \cite[Lemma 3.1.3]{Quint2002divergence}. However the abscissa of convergence of this series is equal to $1$ by the definition of $\delta_\psi$. Therefore 
$\delta_\psi \psi$ is tangent to $\psi^\theta_\Ga$, finishing the proof. \end{proof}

\begin{cor}\label{oneone}
Suppose that $\Ga$ is Zariski dense. A $(\Ga, \theta)$-proper linear form
$\psi \in \fa_{\theta}^*$ with $\delta_{\psi} = 1$ is $(\Ga, \theta)$-critical. Moreover, if $\psi > 0$ on $\L_{\theta}$, then $\psi$ is $(\Ga, \theta)$-critical if and only if $\delta_{\psi} = 1$. 
\end{cor}

Via the identification $\fa_\theta^*=\{\psi\in \fa^*: \psi=\psi \circ p_\theta\}$, we can consider $\fa_{\theta}^*$ as a subspace of $\fa^*$.
Lemma \ref{qu2} implies the following identity: 
 \begin{cor}
Suppose that $\Ga$ is Zariski dense.
     If $p_\theta|_{\L}$ is a proper map, then
$$
   \{\psi \in \fa_{\theta}^* : \psi \mbox{ is } (\Ga, \theta)\mbox{-critical}\} = \{\psi\in \fa^*: \psi=\psi\circ p_\theta,
\psi\text{ is $(\Ga,\Pi)$-critical}\}.
$$
\end{cor}
\begin{proof} To show the inclusion $\supset$, suppose $\psi=\psi\circ p_\theta$ and
$\psi$ is $(\Ga,\Pi)$-critical. Then
for any $u\in \fa_\theta$ and any 
$v'\in p_\theta^{-1}(u)$,
$\psi(u)=\psi(v') \ge  \psi_\Ga(v')$ and hence
 $\psi(u)\ge \psi_\Ga^\theta(u)$ by Lemma \ref{qu2}. 
Moreover, if $\psi(v)=\psi_\Ga(v)$, then for $u=p_\theta(v)$, 
$\psi (u)\ge \psi_\Ga^{\theta}(u)\ge \psi_\Ga(v)=\psi(v)=\psi(u)$ and hence $\psi(u)=\psi_\Ga^{\theta}(u)$, proving $\psi$ is
$(\Ga, \theta)$-critical. For the other inclusion $\subset$,
suppose that $\psi\ge \psi_\Ga^{\theta}$ on $\fa_\theta^+$ and
$\psi(u)=\psi_\Ga^{\theta}(u)$ for some $u\in \fa_\theta^+$.
Then for any $v\in \fa^+$, $\psi(v)=\psi(p_\theta(v))\ge
\psi_\Ga^{\theta}(p_\theta(v))\ge \psi_\Ga(v)$ by Lemma  \ref{qu}. Let $v\in p_\theta^{-1}(u)$ be such that
$\psi_\Ga^{\theta}(u)=\psi_\Ga(v)$ given by Lemma \ref{qu2}.
Then $\psi(v)=\psi(u)=\psi_\Ga^{\theta}(u)=\psi_\Ga(v) $; so $\psi$
is $(\Ga,\Pi)$-critical.
\end{proof}

\section{Limit set, $\theta$-conical set, and conformal measures} \label{sec.conical}
Let $\Ga<G$ be a closed subgroup.
\begin{Def}[$\theta$-limit set]\label{tli}
    We define the $\theta$-limit set of $\Ga$ as follows:
$$\lat=\lat(\Ga):=\{\lim {\ga}_i\in \F_\theta: {\ga}_i\in \Ga\}$$ where $\lim \ga_i$ is defined as in Definition \ref{fc}.
 \end{Def}
 This is a $\Ga$-invariant closed subset of $\F_\theta$, which may be empty in general. Set $\La=\La_{\Pi}$.
 Denote by $\op{Leb}_{\theta}$ the $K$-invariant probability measure on $\ft$. This definition of $\lat$ coincides with that of Benoist:
\begin{lem}[{\cite{Benoist1997proprietes}, \cite[Corollary 5.2, Lemma 6.3, Theorem 7.2]{Quint2002Mesures}, \cite[Lemma 2.13]{lee2020invariant}}] \label{lq} 
If $\Ga$ is Zariski dense in $G$, we have the following:
\begin{enumerate}
\item $\lat=\{\xi\in \ft:
({\ga}_i)_*\op{Leb}_\theta\to D_{\xi} \;\; \text{for some infinite sequence ${\ga}_i\in \Ga$}\} $ where $D_{\xi}$ is the Dirac measure at $\xi$;
\item $\lat =\pi_\theta(\La)$;
     \item  $\lat$ is the unique $\Ga$-minimal subset of $\F_\theta$.
\end{enumerate}
\end{lem}

In the rest of this section, suppose that $\Ga$ is discrete.

\begin{Def}[$\theta$-conical set]    
We define the $\theta$-conical set of $\Ga$ as 
\be\label{cd} \La_\theta^{\mathsf{con}}=\left\{gP_\theta \in \F_\theta : \limsup_{\ga \in \Ga} \ga g M_{\theta} A^+ \ne \emptyset \right\}.\ee
\end{Def} 

For $\theta=\Pi$,
 $\La_\Pi^{\mathsf{con}}=\{gP\in \F:  \limsup_{\ga \in \Ga} \ga g A^+ \ne \emptyset\}$ because $M_\Pi=M$ commutes with $A$.
Note that  the conical set
is not contained in the limit set $\La$ in general even for $\theta=\Pi$.
For example, if $G = \PSL_2(\R) \times \PSL_2(\R)$ and $\Ga = \Ga_1 \times \Ga_2 $ is a product of two convex cocompact subgroups, then
$\La =\La(\Ga_1)\times \La(\Ga_2)$ while $ \La^{\mathsf{con}}=(\La(\Ga_1) \times \S^1) \cup (\S^1\times \La (\Ga_2)) $.

\subsection*{$\theta$-shadows} For $ q \in X$ and $R>0$, let
$B(q,R)=\{x\in X: d(x, q) \le  R\}$.
For $p\in X$, the $\theta$-shadow
$O^\theta_R(p, q) \subset\F_\theta$ of $B(q,R)$ viewed from $p$ is defined as
\begin{align}\label{sh} O^\theta_R(p,q) &= 
\{gP_{\theta} \in \F_\theta : g\in G,\; go=p,\; g A^+o \cap B(q, R)\ne \emptyset \} \\  &= 
\{gP_{\theta} \in \F_\theta : g\in G,\; go=p,\; gM_\theta A^+o \cap B(q, R)\ne \emptyset \} \notag
. \end{align}

Clearly, for $O_R(p, q)=O_R^{\Pi}(p, q)$, we have
 $$O_R^{\theta}(p, q) :=\pi_\theta (O_R(p, q)) .$$

\begin{lem}\label{cons}
     We have
     $\xi \in \La_\theta^{\mathsf{con}}$  if and only if
there exist an infinite sequence $ \ga_i\in \Ga $ and $N>0$ such that
$ \xi\in \bigcap_i O^\theta_N(o,\ga_i o)$.
\end{lem}
\begin{proof}  The direction $\Rightarrow$ is clear.
  To see the other direction, suppose that $ \xi\in \bigcap_i O^\theta_N(o,\ga_i o)$ for some $N>0$ and an infinite sequence $\ga_i\in \Ga$, that is,
there exist sequences $k_i\in K$ and $ a_i\to \infty $ in $ A^+ $ such that $\xi=k_iP_\theta$ and the sequence $\gamma_i^{-1} k_i a_i $ is  bounded. By passing to a subsequence, we may assume $k_i$ converges to some $k\in K$. Since $\xi=k_iP_\theta$ for all $i$, we have $\xi=kP_\theta$. Since $k_iP_\theta=kP_\theta$
and $M_\theta=P_\theta\cap K$, we have $k_i=k m_i$
for some $m_i\in M_\theta$. Since $\gamma_i^{-1} k m_i a_i=\ga_i^{-1} k_i a_i$ is bounded, we have $\xi=kP_\theta\in \La_\theta^{\mathsf{con}}$.
\end{proof}
We remark that  we may replace $o$ by any $p\in X$ in the above lemma.

For each $N>0$, we set $$
\La_\theta^N:=
\left\{ \xi\in\F_\theta :\text{ there exists }\ga_i\to\infty\text{ in }\Ga\text{ such that }\xi\in \bigcap_i O^\theta_N(o,\ga_i o)\right\}.$$
By Lemma \ref{cons}, we have 
\be\label{lan} \La_\theta^{\mathsf{con}}=\bigcup_{N=1}^\infty \La_\theta^N .\ee

  \begin{Def} \label{reg1} For a $\theta$-discrete subgroup $\Ga$, we say that $\Ga$ is {\it $\theta$-regular}
if for any sequence $\ga_i \to \infty$ in $\Ga$, we have 
$$\min_{\alpha\in \theta} \alpha(\mu({\ga}_i))\to \infty .$$
\end{Def}
Observe that $\theta$-regularity is same as $\theta\cup \i(\theta)$-regularity by \eqref{mu} and that
not every $\theta$-discrete subgroup is $\theta$-regular.
\begin{prop} \label{compact}\label{reg}
If $\Ga$ is $\theta$-regular, then
 \begin{enumerate}
\item $\La^{\mathsf{con}}_{\theta}\subset \lat$;
\item  for any compact subset $Q\subset G$,
the union $\G Q\cup  \lat$ is compact where the topology is given by the convergence in Definition \ref{fc}; that is, any infinite sequence has a limit.
\end{enumerate}
\end{prop}

\begin{proof}To show (1), 
    let $\xi\in \lat^{\mathsf{con}}$. Then there exist $g\in G$, a sequence $\ga_i\in \Ga$, $m_i\in M_\theta$ and $a_i\in A^+$
    such that $\xi=g\xi_\theta$ and $d(gm_ia_i o, \ga_i o)$
    is uniformly bounded. Since $\mu (\ga_i)-\log a_i$ is uniformly bounded by Lemma \ref{lem.cptcartan}, and $\min_{\alpha\in \theta} \alpha(\mu(\ga_i))\to \infty $ by the $\theta$-regularity,
    we have  $\min_{\alpha\in \theta} \alpha(\log a_i)\to \infty $ as $i\to \infty$. 
     We may assume that $m_i$ converges to some $ m\in M_\theta$ by passing to a subsequence.
     Therefore as $i\to \infty$, $g m_i a_i o \to gm \xi_\theta =g\xi_\theta$ by Lemma \ref{lem.211inv}.
    This implies that $\ga_i o\to g\xi_\theta$ by Lemma \ref{bdd}.  Hence
    $\xi\in \La_\theta$. 
    For (2), if $\ga_i\in \Ga$ is an infinite sequence and $q_i\in Q$, then $\min_{\alpha\in \theta}\alpha(\mu(\ga_i q_i))\to \infty$ by the $\theta$-regularity of $\Ga$ and Lemma \ref{lem.cptcartan}.
    Hence the claim 
    is now immediate from Definition \ref{fc} and Lemma \ref{bdd}.
\end{proof}

\subsection*{Conical convergence}

From the viewpoint of Lemma \ref{cons}, we define the conical convergence as follows.

\begin{Def} We say that a sequence $g_i\in G$ converges to $\xi\in \F_\theta$
conically if $g_i \to \xi$ in the sense of Definition \ref{fc}
and there exists $R>0$ such that $\xi\in O_R^{\theta}(o, g_io)$ for all $i\ge 1$. Note that if 
$\ga_i\in \Ga$ converges to $\xi\in \F_\theta$ conically, then $\xi\in \La_{\theta}^{\mathsf{con}}$.
\end{Def} 

The following lemma is stated in \cite[Lemma 5.29]{Kapovich2017anosov} in a different language. We give a more direct proof.
\begin{lemma}\label{lem.klpconical}
    Let $g_i \in G$ be a sequence which converges to $\xi \in \F_{\theta}$. Then the following are equivalent:
    \begin{enumerate}
        \item The convergence $g_i \to \xi$ is conical.
        \item For any $\eta \in \F_{\i(\theta)}$ such that $(\xi, \eta) \in \F_{\theta}^{(2)}$, the sequence $g_i^{-1}(\xi, \eta)$ is precompact in $\F_{\theta}^{(2)}$.
         \item For some $\eta \in \F_{\i(\theta)}$ such that $(\xi, \eta) \in \F_{\theta}^{(2)}$, the sequence $g_i^{-1}(\xi, \eta)$ is precompact in $\F_{\theta}^{(2)}$.
    \end{enumerate}
\end{lemma}
\begin{proof}
The map $gL_\theta \to (gP_\theta, gw_0 P_{\i(\theta)})$ is a $G$-equivariant homeomorphism from $G/L_\theta$ to $\F_\theta^{(2)}$.
We first prove  $(1) \Rightarrow (2)$. Suppose (1). 
Then there exist sequences $k_i \in K$ and $a_i \to \infty$ in $A^+$ such that $\xi = k_i P_{\theta}$ for all $i$ and the sequence $g_i^{-1}k_i a_i$ is bounded. If $(\xi, \eta)\in \F_\theta^{(2)}$, then $\xi=hP_\theta$
and $\eta= hw_0P_{\i(\theta)}$ for some $h\in G$. Since $hP_{\theta} = k_i P_{\theta}$, $h = k_i p_i m_i$ for some $p_i \in P$ and $m_i \in M_{\theta}$, by using $P_{\theta} = P M_{\theta}$. In other words, we have $k_i^{-1} h m_i^{-1} = p_i$ and hence $p_i$ is a bounded sequence in $P$ since $k_i$ and $m_i$ are bounded sequences. In particular, it follows from $a_i \in A^+$ that the sequence $a_i^{-1} p_i a_i$ is bounded by Lemma \ref{lem.boundedofparabolic}. Therefore the sequence $g_i^{-1}hL_{\theta} = g_i^{-1}k_ip_iL_{\theta} = (g_i^{-1}k_ia_i)(a_i^{-1}p_ia_i)L_{\theta}$ is precompact in $G/L_{\theta}$, which is
equivalent to saying that $g_i^{-1}(\xi, \eta)$ is precompact, proving (2).
The implication  $(2) \Rightarrow (3)$ is clear. 

Now $(3) \Rightarrow (1)$ follows from Lemma \ref{lem.eventualshadow} below applied to the constant sequence $(\xi_i, \eta_i) = (\xi, \eta)$.
\end{proof}

\begin{lemma} \label{lem.eventualshadow}
    Let $g_i \in G$ and $\xi_i\in \F_\theta$  be sequences both  converging to some $\xi \in \F_{\theta}$. Suppose that 
    there exists a sequence $\eta_i\in \F_{\i(\theta)}$ converging to some $\eta\in \F_{\i(\theta)}$  such that $(\xi, \eta) \in \F_{\theta}^{(2)}$ and the sequence $g_i^{-1}(\xi_i, \eta_i)$ is precompact in $\F_{\theta}^{(2)}$. Then there exists $R > 0$ such that $$\xi_i \in O_R^{\theta}(o, g_i o) \quad \text{for all } i \ge 1.$$
\end{lemma}

\begin{proof}
    Under the identification $G/L_{\theta} = \F_{\theta}^{(2)}$ given by $g L_{\theta} = (g P_{\theta}, g w_0 P_{\i(\theta)})$, the hypothesis implies that there exists a sequence $h_i \in G$ with the limit $h \in G$ so that $(\xi_i, \eta_i) = h_i L_{\theta}$ for all $i \ge 1$ and $(\xi, \eta) = h L_{\theta}$.
    It follows from the precompactness of $g_i^{-1}(\xi_i, \eta_i)$ that there exists a sequence $\ell_i \in L_{\theta}$ such that $g_i^{-1}h_i \ell_i$ is a bounded sequence.
    
Since $L_{\theta} = M_{\theta}A M_{\theta}$, we can write $\ell_i = m_i a_i' m_i' \in M_{\theta} A M_{\theta}$, and hence we have $g_i^{-1} h_i m_i a_i'$ is bounded. For each $i$, let $w_i \in K$ be a representative of a Weyl element such that $w_i^{-1} a_i' w_i\in A^+$. After passing to a subsequence, we may assume that the sequence $m_i$ converges to some $m \in M_{\theta}$ and $w_i$ is a constant sequence, say $w$. We claim that $w\in M_\theta$.
Denoting by $a_i = w^{-1} a_i' w \in A^+$, \be \label{eqn.Levibdd}
\text{the sequence } g_i^{-1} h_i m_i w a_i \text{ is bounded.}
\ee Moreover, since $\min_{\alpha \in \theta} \alpha(\mu (g_i)) \to \infty$, we have $\min_{\alpha \in \theta} \alpha(\log a_i) \to \infty$ as $i \to \infty$ by Lemma \ref{lem.cptcartan}. Since $h_i m_i w a_i= g_i (g_i^{-1}h_i m_i w a_i)$,
$g_i \to \xi$,
and $g_i^{-1}h_i m_i w a_i$ is a bounded sequence by \eqref{eqn.Levibdd},
we have as $i\to \infty$, $$h_i m_i w a_i \to \xi$$ by  Lemma \ref{lem.210inv}. On the other hand, by  Lemma \ref{lem.211inv}, we have that as $i\to \infty$,
$$h_im_i w a_i \to hm w P_{\theta}.$$ Hence we have $hm wP_{\theta} = \xi = h P_{\theta}$. Since $m \in M_{\theta}$, it follows that
$$w \in K \cap P_{\theta} = M_{\theta}.$$ In particular, $\xi_i = h_i m_i w P_{\theta} $ for all $i$.

For each $i$, write $h_im_i w = k_i b_i n_i \in KAN$ in the Iwasawa decomposition. We then have $\xi_i = h_i m_i w P_{\theta} = k_i P_{\theta}$. Since the sequence $h_i m_i w$ is convergent and the product map $K \times A \times N \to G$ is a diffeomorphism, the sequences $b_i$ and $n_i$ are bounded.
Since $a_i \in A^+$, the sequence $a_i^{-1} n_ia_i$ is bounded  by Lemma \ref{lem.boundedofparabolic}, and so is the sequence $b_i a_i^{-1} n_i a_i$.
On the other hand, \eqref{eqn.Levibdd} implies that
\be \label{eqn.iwasawabdd}
\text{the sequence } g_i^{-1}k_ib_in_ia_i = (g_i^{-1} k_i a_i)(b_i a_i^{-1} n_i a_i) \text{ is bounded.}
\ee
Therefore it follows that $g_i^{-1} k_i a_i$ is bounded.  This mean that for some $R > 0$, $\xi_i = k_i P_{\theta} \in O_R^{\theta}(o, g_i o)$ for all $i$, as desired.
\end{proof}

\subsection*{Conformal measures}
The $\frak a$-valued Busemann map $\beta: \cal F\times G \times G \to\frak a $ is defined as follows: for $\xi\in \cal F$ and $g, h\in G$,
$$  \beta_\xi ( g, h):=\sigma (g^{-1}, \xi)-\sigma(h^{-1}, \xi)$$
where  $\sigma(g^{-1},\xi)\in \fa$ 
is the unique element such that we have the Iwasawa decomposition $g^{-1}k \in K \exp (\sigma(g^{-1}, \xi)) N$ for any $k\in K$ with $\xi=kP$.
We define the $\fa_{\theta}$-valued Busemann map $\beta^{\theta} : \F_{\theta} \times G \times G \to \fa_{\theta}$ as follows: for $(\xi,g,h)\in \cal F_\theta\times G\times G$, we set
 \be\label{Bu} \beta_{\xi}^\theta (g, h): = 
p_\theta ( \beta_{\xi_0} (g, h)) \quad\text{for $\xi_0\in \pi_\theta^{-1}(\xi)$};\ee 
this is well-defined independent of the choice of $\xi_0$ 
\cite[Lemma 6.1]{Quint2002Mesures}.

The following was shown for $\theta = \Pi$ in \cite[Lemma 5.7]{lee2020invariant} which directly implies the statement for general $\theta$ since $p_{\theta}$ is norm-decreasing.

\begin{lemma} \label{lem.buseandcartan} \label{shsh}
  There exists $\kappa > 0$ such that for any $g, h \in G$ and $R>0$, we have $$\sup_{\xi \in O^\theta_R(go, ho)}  \| \beta_{\xi}^{\theta}(g, h) - \mu_{\theta}(g^{-1}h) \| \le \kappa R.$$
\end{lemma}

Following the work of
Patterson-Sullivan (\cite{Patterson1976limit}, \cite{Sullivan1979density}) in rank one, Quint \cite{Quint2002Mesures} has introduced the notion of conformal measures in general. 
\begin{definition}[Conformal measures]
For a linear form $\psi \in \fa_{\theta}^*$ and a closed subgroup $\Ga<G$, a Borel probability measure $\nu$ on $\F_{\theta}$ is called a {\it $(\Ga, \psi)$-conformal measure} if $$\frac{d \gamma_*\nu}{d\nu}(\xi)=e^{\psi(\beta_\xi^\theta(e,\gamma))} \quad \text{for all $\gamma \in \Gamma$ and $ \xi \in \mathcal{F}_\theta$}. $$
 \end{definition}

\begin{prop}\label{qc}
 Suppose that $\Ga$ is Zariski dense and $\theta$-discrete. For any linear form $\psi \in \fa_{\theta}^*$ which is tangent to $\psi_{\Ga}^{\theta}$ at an interior direction of  $\fa_{\theta}^+$, there exists a $(\Ga, \psi)$-conformal measure supported on $\La_\theta$. 
\end{prop}
\begin{proof}
For $\theta = \Pi$,
this was shown by Quint using the concavity of $\psi_\Ga$ \cite[Theorem 8.4]{Quint2002Mesures}. Now that we established the concavity of the $\theta$-growth indicator $\psi_\Ga^\theta$ (Proposition \ref{st}), the same proof works for general $\theta$.
\end{proof}
As in the Patterson-Sullivan construction, the conformal measure in the above proposition can be obtained as a limit of a sequence of
certain weighted counting measures on $\Ga o$.
The assumption that $\psi$ is tangent to $\psi_\Ga^\theta$ at an {\it interior} direction of $\fa_\theta^+$ is needed to guarantee that the limiting measure
is supported on the limit set $\La_\theta$.
For a $\theta$-regular subgroup $\Gamma$,  the union $ \Ga o \cup \La_\theta$ is a compact space, and hence
the assumption that the tangent direction belongs to $\inte\fa_\theta^+$ is unnecessary.
The proof below is an easy adaptation of the standard construction of Patterson-Sullivan (see also  \cite[Section 2]{kmoHD}, \cite[Section 5]{sambarino2022report}, \cite{Canary2023}).

\begin{prop} \label{measure}
    Suppose that $\Ga$ is $\theta$-regular.
    For any  $(\Ga, \theta)$-proper $\psi\in \fa_\theta^*$ such that $\delta_\psi=1$ and
   $ \sum_{\ga\in \Ga} e^{-\psi (\mu_\theta(\ga))} =\infty $,
    there exists a $(\Ga, \psi)$-conformal measure supported on $\La_\theta$. 
\end{prop}
\begin{proof} By Proposition \ref{reg},
$\G o\cup \Lambda_\theta$ is a compact space.
Recall that $\cal P_\psi(s)=\sum_{\ga\in \Ga} e^{-s\psi (\mu_\theta(\ga))}$.
As $\delta_\psi=1$,  $\cal P_\psi(s)<\infty$  for $s>1$.
and hence we may consider the probability measure on $\G o \cup  \Lambda_\theta$ given by 
\be\label{exp00} \nu_{\psi, s}:= \frac{1}{\cal P_\psi(s)} \sum_{\ga\in \Ga} e^{-s\psi(\mu_\theta(\ga))} D_{\ga o}
\ee 
where $D_{\ga o}$ is the point mass at $\ga o$.

Since the space of probability measures on a compact metric space
a weak$^*$ compact space,
 by passing to a subsequence, as $s\to 1$, $\nu_{\psi, s}$ weakly converges to a probability measure, say $\tilde \nu_\psi$, on $\G o\cup \La_\theta$. Since $\cal P_\psi(1)=\infty$, $ \nu_\psi$ is
supported on $\La_\theta$. It is standard to check that $\nu_\psi$ is a $(\Ga, \psi)$-conformal measure. 
\end{proof}

Although we will not be using this generality,  Proposition \ref{measure} holds without the hypothesis $ \sum_{\ga\in \Ga} e^{-\psi (\mu_\theta(\ga))} =\infty $ (see
\cite[Proposition 3.2]{Canary2023}).

\section{Transverse subgroups and multiplicity of $\theta$-shadows} \label{sec.shadow}
 We say that a discrete subgroup $\Ga<G$ is {\it $\theta$-antipodal}
if any two distinct points $\xi\ne \eta$ in
$\La_{\theta\cup \i(\theta)}$ 
are in general position, i.e., 
$$\xi= gP_{\theta\cup \i(\theta)}\;\; \text{ and } \;\; \eta=gw_0P_{{\theta\cup \i(\theta)}}$$
for some $g\in G$. Recall that a  discrete subgroup $\Ga<G$ is called $\ts$-transverse if $\Ga$ is both
$\ts$-regular and $\ts$-antipodal. A $\theta$-transverse subgroup $\Ga < G$ is called \emph{non-elementary} if $\# \La_{\theta} \ge 3$. Note that for $\theta_1\subset \theta_2$, $\theta_2$-transverse implies $\theta_1$-transverse.

\begin{Rmk} \label{strong}   We may try to define $\Ga$ to be {\it $\theta$-Antipodal}
if for any $(\xi,\eta) \in \La_\theta\times \La_{\i(\theta)}$ such that
$\pi_\theta^{-1}(\xi)\cap \pi_{\i(\theta)}^{-1}(\eta)=\emptyset$, 
$(\xi, \eta)$ is in general position, i.e., $\xi= gP_\theta$ and $\eta=gw_0P_{\i(\theta)}$ for some $g\in G$.
While the $\t$-antipodality implies the $\theta$-Antipodality,  the converse direction is not true in general; for instance, any lattice of $\op{PSL}_3(\br)$ is $\{\alpha_1\}$-Antipodal but not 
$\{\alpha_1, \alpha_2\}$-Antipodal, i.e., not $\{\alpha_1\}$-antipodal, where $\alpha_i(\op{diag}(u_1, u_2, u_3))=u_i-u_{i+1}$ for $i=1,2$. \end{Rmk}

 The main aim of this section is to prove the following proposition, which is the essential reason why the main results of this paper are proved for $\ts$-transverse subgroups.

\begin{prop}[Bounded multiplicity of $\theta$-shadows] \label{prop.mult}  Assume that $\Ga$ is $\ts$-transverse.
Let $\phi\in \fa_\theta^*$ be a $(\Ga, \theta)$-proper linear form.
Then for any $R, D > 0$, there exists $q = q(\phi, R, D) > 0$ such that for any $T > 0$, the collection of shadows $$\left\{O^{\theta}_R(o, \ga o) \subset \F_\theta : T \le \phi (\mu_{\theta}(\ga)) \le T + D \right\}$$ have multiplicity at most $q$, i.e.,
for and $\xi \in \F_{\theta}$ and $T > 0$, there are at most $q$ number of shadows from the above collection that contain $\xi$.

\end{prop}

 The following lemma is a key ingredient of the proof of Proposition \ref{prop.mult}.

\begin{lem} \label{cor.approxcartan}\label{lem.comparecartan}
  Assume that $\Ga$ is $\ts$-transverse.
  For any compact subset $Q$  of $G$, there exists $C_0=C_0(Q)>0$ such that if ${\ga}_1,{\ga}_2\in\Ga$ are such that
   $Q\cap {\ga}_1 Qa^{-1}\cap {\ga}_2 Qb^{-1} m^{-1}\neq\emptyset$ for some  $a,b\in A^+$ and $m \in M_{\theta}$, then 
     \be\label{con} \min \{ \| \mu_\theta({\ga}_2) - \mu_\theta({\ga}_1)-\mu_\theta({\ga}_1^{-1}{\ga}_2)\|,
    \| \mu_\theta({\ga}_1)- \mu_\theta({\ga}_2)-\mu_\theta({\ga}_2^{-1}{\ga}_1)\|\} \le C_0.\ee 
\end{lem}

\begin{proof} Since $\|p_{\theta}(u)\|\le\| p_{\t\cup \i(\t)}(u)\|$ for all $u\in \fa$, it suffices to prove the lemma for $\t\cup \i(\t)$ in place of $\theta$. Therefore we may assume without loss of generality that $\i(\theta)=\theta$ by replacing $\theta$ with $\t\cup \i(\t)$.

We prove by contradiction. Suppose to the contrary that
 there exist sequences
$q_{0,i},q_{1,i},q_{2,i}\in Q$, $a_i, b_i\in A^+$, $m_i \in M_{\theta}$ and ${\ga}_{1,i},{\ga}_{2,i}\in\Ga$ such that
 \begin{align} & q_{0,i}={\ga}_{1,i}\,q_{1,i}\, a_i^{-1}={\ga}_{2,i}\,q_{2,i} b_i^{-1}m_i^{-1} ;\label{eq.w2}\\
 &\| \mu_\theta({\ga}_{2,i}) - \mu_\theta({\ga}_{1,i})-\mu_\theta({\ga}_{1,i}^{-1}{\ga}_{2,i})\| \to \infty ;\label{e1}\\
 &\|\mu_\theta({\ga}_{1,i}) - \mu_\theta({\ga}_{2,i})-\mu_\theta({\ga}_{2,i}^{-1}{\ga}_{1,i})\|\to \infty \label{ee2}.\end{align}

By Lemma \ref{lem.cptcartan}, it follows that all sequences
$\ga_{1,i}$, ${\ga_{2,i}}$, $\ga_{1,i}^{-1}\ga_{2,i}$
and $\ga_{2,i}^{-1}\ga_{1,i}$ are unbounded. 
Without loss of generality, we assume that each of these sequences tends to infinity. By \eqref{eq.w2} and Lemma \ref{lem.cptcartan}, there exists $C'=C'(Q) >1$ such that

\be \label{cc}
\sup_i \{\| \mu_{\theta}({\ga}_{1,i})-  \mu_{\theta}(a_i) \|, \| \mu_{\theta}({\ga}_{2,i}) - \mu_{\theta}(b_i) \|\} \le C'
\ee

As $\Ga$ is $\theta$-regular, as $i\to \infty$,
$$\min_{\alpha\in \theta} \alpha(\log a_{i}), \; \min_{\alpha\in \theta} \alpha(\log b_{i}) \to \infty.$$

Note that $\alpha (\log w_0^{-1} a^{-1} w_0)= \alpha (\i(\log a))= \i (\alpha)(\log a)$ for all $a\in A$ and all $\alpha\in \Phi$.
Since $\theta$ is symmetric,  it follows that
\be\label{symm} \min_{\alpha\in \theta} \alpha(\log (w_0^{-1} a_{i}^{-1} w_0)), \; \min_{\alpha\in \theta} \alpha(\log (w_0^{-1} b_{i}^{-1} w_0))
\to \infty.\ee
Passing to a subsequence, we may assume that $q_{1, i}$ converges to some $q_1 \in Q$. 
We claim that
\be \label{both} q_1 w_0 \xi_\theta\in \La_{\theta}\;\; \text{ and}\;\;  q_1 m_1 w \xi_\theta   \in \La_{\theta}
\ee 
for some $m_1\in M_\theta$ and $w\in N_K(A)$.
By Lemma \ref{compact},
we may also assume that ${\ga}_{1, i}^{-1}q_{0, i}o$ converges to some $\xi \in \La_{\theta}$ as $i \to \infty$. Since ${\ga}_{1, i}^{-1}q_{0, i} o = q_{1, i}a_i^{-1}o = q_{1, i}w_0 (w_0^{-1} a_i^{-1}w_0) o  $, it follows from Lemma \ref{lem.211inv} and
\eqref{symm} that $\xi = q_1 w_0 \xi_\theta$.
Therefore $$q_1 w_0 \xi_\theta  \in \La_{\theta}.$$

Since $A=A_\theta B_\theta$, we may
write $a_i = a_{1,i}a_{2,i} \in A_{\theta}^+ B_\theta^+$ and $b_i = b_{1,i} b_{2,i} \in A_{\theta}^+ B_\theta^+$.
Using $S_\theta=M_\theta B_\theta^+M_\theta$,
write $$a_{2,i}^{-1} m_i b_{2,i} = m_{1, i} c_i m_{2, i} \in M_{\theta} B_\theta^+ M_{\theta}.$$ 

Then
\begin{multline*}
{\ga}_{1,i}^{-1}{\ga}_{2,i}\,q_{2,i}=q_{1,i}\, a_i^{-1} m_i b_i
\\  = q_{1, i} (a_{1,i}^{-1} b_{1,i}) (a_{2,i}^{-1} m_i b_{2,i})
= q_{1, i} m_{1,i}  (a_{1,i}^{-1} b_{1,i}  c_i) m_{2,i}.
\end{multline*}
By passing to a subsequence, we have  $w\in N_K(A)$ such that for all $i\ge 1$, 
\be \label{eqn.decompwow}
d_i:= w^{-1}a_{1,i}^{-1}b_{1,i}c_i w\in A^+.
\ee
Then we have the following:
\be \label{eq.w3}
{\ga}_{1, i}^{-1} {\ga}_{2, i} q_{2, i} = q_{1, i}(m_{1, i} w ) d_i (w^{-1}m_{2, i}) \in q_{1, i}
KA^+K.
\ee

Since $\ga_{1,i}^{-1}\ga_{2,i}\to \infty$, by the $\theta$-regularity of $\Ga$, we have $\min_{\alpha\in \theta} \alpha(\log d_i)\to \infty.$
We may assume that $m_{1,i}\to m_1\in M_\theta$.
By Lemma \ref{compact} and Lemma \ref{lem.211inv}, we get
$$\lim_{i\to \infty}
{\ga}_{1, i}^{-1} {\ga}_{2, i} q_{2, i} =  q_1 m_{1} w\xi_\theta \in \La_{\theta}$$
by passing to a subsequence. Hence the claim \eqref{both} is proved. 

By the $\theta$-antipodal property of $\Ga$,
two distinct points of $\lat$ must be in general position; hence
\eqref{both} implies that we must have either
$$w_0 \xi_\theta=
 m_1 w \xi_\theta   \quad\text{or} \quad m_1 w \xi_\theta  \in N_\theta^+\xi_\theta. $$

First suppose that $(m_1w)\xi_\theta \in N_\theta^+\xi_\theta$. By Corollary \ref{cor.genweyl}, this implies that $w \in M_\theta$.
As  $a_{1, i}^{-1}b_{1,i}\in A_\theta$, using the commutativity of $M_\theta$ and $A_\theta$, we get from (\ref{eqn.decompwow})
that $d_i=(a_{1, i}^{-1}b_{1, i}) (w^{-1}c_i w)$.
Since $d_i\in A^+$, $a_{1, i}^{-1}b_{1, i}\in A_\theta$, and $w^{-1}c_i w\in B_\theta$, it follows  that
$a_{1, i}^{-1}b_{1, i}\in A_\theta^+$. 
 Hence \be \label{eqn.firstcontra}
\mu_{\theta}(d_i) = \log a_{1, i}^{-1}b_{1, i} = - \log a_{1, i} + \log b_{1, i} = -\mu_{\theta}(a_{i}) + \mu_{\theta}(b_{i}).
\ee
Since $\| \mu_{\theta}(\ga_{1, i}^{-1} \ga_{2, i}) - \mu_{\theta}(d_i) \|$ is uniformly bounded by \eqref{eq.w3} and Lemma \ref{lem.cptcartan}, \eqref{eqn.firstcontra} and \eqref{cc} imply that the sequence $\|\mu_{\theta}(\ga_{1, i}^{-1} \ga_{2, i}) + \mu_{\theta}(\ga_{1, i}) - \mu_{\theta}(\ga_{2, i}) \|$ is uniformly bounded. This contradicts \eqref{e1}.

Now suppose the other case that  $w_0 \xi_{\theta} = m_1 w \xi_{\theta}$. In this case, we have $$w \xi_{\theta} = m_1^{-1} w_0 \xi_{\theta} = w_0 (w_0^{-1} m_1^{-1} w_0) \xi_{\theta} = w_0 \xi_{\theta}$$ since $m_1 \in M_{\theta}$ and $w_0^{-1} M_{\theta} w_0 = M_{\theta}$ by the symmetricity of $\theta$. Hence we have $w \in w_0 (P_{\theta} \cap K) = w_0 M_{\theta} = M_{\theta} w_0$, and thus $ w w_0^{-1} \in  M_{\theta}$.
Since $w w_0^{-1} \in M_{\theta}$, we may write using \eqref{eqn.decompwow} that
$$\begin{aligned} 
    w_0 d_i^{-1} w_0^{-1} & = (ww_0^{-1})^{-1} a_{1, i} b_{1, i}^{-1} c_i^{-1} (w w_0^{-1}) \\
    & = (a_{1, i} b_{1, i}^{-1}) ((ww_0^{-1})^{-1}  c_i^{-1} (w w_0^{-1}))\in A_\theta B_\theta
    \end{aligned}$$
     Since $d_i \in A^+$, we have $w_0 d_i^{-1} w_0^{-1} \in A^+$.
    It follows that  $a_{1, i} b_{1, i}^{-1} \in A_{\theta}^+ $.
     Hence we have \be \label{eqn.secondcontra}
    \mu_{\theta}(d_i^{-1}) = p_\theta( \log (w_0 d_i^{-1} w_0^{-1})) = \log a_{1, i} - \log b_{1, i} = \mu_{\theta}(a_i) - \mu_{\theta}(b_i).
    \ee
    Similarly, \eqref{eq.w3} implies that the sequence $\| \mu_{\theta}(\ga_{2, i}^{-1} \ga_{1, i}) - \mu_{\theta}(d_i^{-1})\|$ is uniformly bounded. Hence it follows from \eqref{eqn.secondcontra} and \eqref{cc} that the sequence $\| \mu_{\theta}(\ga_{2, i}^{-1} \ga_{1, i}) - \mu_{\theta}(\ga_{1, i}) + \mu_{\theta}(\ga_{2, i}) \|$ is uniformly bounded, contradicting \eqref{ee2}.
 This completes the proof.
\end{proof}

\subsection*{Proof of Proposition \ref{prop.mult}}
Suppose that there exists $\xi \in \bigcap_{i = 1}^n O_R^{\theta}(o, {\ga}_i o)$ and $T \le \phi(\mu_{\theta}({\ga}_i)) \le T + D$ for some distinct ${\ga}_i \in \Ga$, $i = 1, \cdots, n$.  Setting $Q = KA_RK$ where $A_R := \{a \in A : d(o, a o) < R \}$,
let $C_0=C_0(Q)$ be as in Lemma \ref{cor.approxcartan}. Note also that $Q = \{ g \in G : d(o, go) < R \}$.  Set
$$D'  =D'(\phi, Q, D):= \|\phi\|C_0 + D$$
where $\|\phi\|$ is the operator norm
of $\phi:\fa_{\theta}\to \br$.  Then the following number 
$$q:=\# \{ \ga \in \Ga : \phi ( \mu_{\theta}(\ga) ) \le D' \}$$ 
 is finite by the $(\Ga,\theta)$-properness of $\phi$.
We claim that $$n\le 2q;$$ this proves the proposition.
It suffices to show that
\be \label{eqn.claimmaxmin}
\max_{i} \min \{ \phi( \mu_{\theta}({\ga}_1^{-1} {\ga}_i) ), \phi( \mu_{\theta}({\ga}_i^{-1} {\ga}_1) ) \} \le D',
\ee as this  implies that 
$$n=\# \{{\ga}_1, \cdots, {\ga}_n\} \le \# \{ {\ga}_1 \ga, {\ga}_1 \ga^{-1} : \ga \in \Ga, \phi(\mu_{\theta}(\ga)) \le D' \}\le 2q.$$

To prove \eqref{eqn.claimmaxmin}, for each $i = 1, \cdots, n$, there exist $k_i \in K$ and $a_i \in A^+$ such that $\xi = k_i \xi_\theta$ and $d(k_i a_i o, {\ga}_i o) < R$. Then $k_i = k_1 m_i$ for some $m_i \in K \cap P_{\theta} = M_{\theta}$. Hence we have $d({\ga}_1^{-1} k_1 a_1 o, o) < R$ and $d({\ga}_i^{-1} k_1m_{i} a_i o, o) < R$, which implies $$k_1 \in Q \cap {\ga}_1 Q a_1^{-1} \cap {\ga}_i Q a_i^{-1}m_{i}^{-1}.$$

By Lemma \ref{cor.approxcartan}, we have $$
\| \mu_{\theta}({\ga}_i) - \mu_{\theta}({\ga}_1) -\mu_{\theta}({\ga}_1^{-1}{\ga}_i) \|\le C_0\quad \mbox{or} \quad \|\mu_{\theta}({\ga}_1) - \mu_{\theta}({\ga}_i) - \mu_{\theta}({\ga}_i^{-1}{\ga}_1)\|\le C_0.$$

Suppose first that $\| \mu_{\theta}({\ga}_i) - \mu_{\theta}({\ga}_1) -\mu_{\theta}({\ga}_1^{-1}{\ga}_i) \|\le C_0$. Now we have
$$
\begin{aligned}
    \phi( \mu_{\theta}(\ga_1^{-1} \ga_i) ) & = \phi( \mu_{\theta}(\ga_1^{-1}\ga_i) - (\mu_{\theta}(\ga_i) - \mu_{\theta}(\ga_1)) ) + \phi (\mu_{\theta}(\ga_i) - \mu_{\theta}(\ga_1))\\
    & \le \|\phi\| C_0 + | \phi(\mu_{\theta}(\ga_i)) - \phi(\mu_{\theta}(\ga_1)) | \\
    & \le \| \phi \| C_0 + D=D'
\end{aligned}$$
where the last inequality follows from $\phi(\mu_{\theta}(\ga_1)), \phi(\mu_{\theta}(\ga_i)) \in [T, T + D]$.
When $\|\mu_{\theta}({\ga}_1) - \mu_{\theta}({\ga}_i) - \mu_{\theta}({\ga}_i^{-1}{\ga}_1)\|\le C_0$, similarly, we have $$\phi(\mu_{\theta}(\ga_i^{-1}\ga_1)) \le \|\phi \| C_0 + D=D'.$$ Therefore \eqref{eqn.claimmaxmin} follows.
\qed

 \section{Dimensions of conformal measures and growth indicators}\label{tangent} \label{sec.dimensions}

For a general Zariski dense discrete subgroup $\Ga<G$,
Quint \cite[Theorem 8.1]{Quint2002Mesures} showed that
if there exists a $(\Ga, \psi)$-conformal measure on $\F_\Pi$ for $\psi\in \fa^*$, then  $$\psi\ge \psi_\Ga.$$ The main aim of this section is to prove the following analogous inequality  for $\ts$-transverse subgroups, using Theorem \ref{tgg0} whose key ingredient is the control on multiplicity of shadows obtained in Proposition \ref{prop.mult}.

\begin{theorem}\label{tgg}
    Let $\Ga$ be a Zariski dense $\ts$-transverse subgroup of $G$. If there exists a $(\Ga, \psi)$-conformal measure $\nu$ on $\F_{\theta}$ for a $(\Ga, \theta)$-proper $\psi \in \fa_{\theta}^*$, then 
    \be\label{oppp}\psi\ge \psi_{\Ga}^{\theta}.\ee 
    Moreover  if $\sum_{\ga\in \Ga} e^{-\psi (\mu_\theta(\ga))}=\infty$ in addition, then $\delta_\psi=1$ and $\psi$ is $(\Ga, \theta)$-critical.
\end{theorem}

\begin{lemma}[$\theta$-shadow lemma] \label{lem.shadow} Let $\Ga < G$ be a discrete subgroup. Let $\nu$
be a $(\Ga, \psi)$-conformal measure on $\F_{\theta}$ for $\psi \in \fa_{\theta}^*$.  Suppose that $\supp \nu$ is not contained in $\F_\theta-\ell N_\theta^+P_\theta$ for any $\ell\in K$.
Then we have the following:
\begin{enumerate}
    \item for some $R = R(\nu) > 0$, we have $c := \inf_{\ga \in \Ga} \nu(O_R^{\theta}(\ga o, o)) > 0$; and
    \item for all $r \ge R$ and for all $\ga \in \Ga$, 
    \be \label{ineq} c e^{-\|\psi\|\kappa r} e^{-\psi(\mu_{\theta}(\ga))} \le \nu(O_r^{\theta}(o, \ga o)) \le e^{\|\psi\|\kappa r} e^{-\psi(\mu_{\theta}(\ga))}\ee  where $\kappa > 0$ is a constant given in Lemma \ref{lem.buseandcartan}.
\end{enumerate}
In particular, if $\Ga$ is Zariski dense, \eqref{ineq} holds for any $(\Ga, \psi)$-conformal measure $\nu$.

Moreover, if $\Ga$ is a  $\theta$-transverse subgroup, then \eqref{ineq} holds for any $(\Ga, \psi)$-conformal measure $\nu$ on $\F_\theta$ such that  $$(\op{supp}\nu, \eta) \cap  \F_{\theta}^{(2)} \ne \emptyset\quad\text{
for any $\eta \in \La_{\i(\theta)}$.}$$
\end{lemma}

\begin{proof}    This lemma was 
proved in \cite[Lemma 7.8]{lee2020invariant} for $\theta=\Pi$, and a general case can be
proved verbatim,  by replacing $P$ and $N$ by $P_\theta$ and $N_\theta$ respectively and noting that the projection $p_\theta:\fa \to \fa_\theta$ is a  Lipschitz map.
We provide a proof for completeness.
To prove (1), 
suppose not. Then there exist $R_i \to \infty$ and $\ga_i \in \Ga$ such that $\nu(O_{R_i}^{\theta}(\ga_i^{-1}o, o)) < 1/i$ for all $i\ge 1$. We write the Cartan decomposition $\ga_i = k_i' a_i \ell_i^{-1} \in KA^+K$ and after passing to a subsequence, we may assume that $k_i' \to k'$ and $\ell_i \to \ell$ as $i \to \infty$. We claim that $N_{\theta}^+ P_{\theta} \subset \limsup O_{R_i}^{\theta}(a_i^{-1} o, o).$ Let $h \in N_{\theta}^+$ and write $a_i h = k_i b_i n_i \in KAN$. Since $a_i ha_i^{-1}$ is bounded by Lemma \ref{lem.boundedofparabolic} and
$a_i h a_i^{-1} = k_i  (b_i a_i^{-1})(a_i n_i a_i^{-1}) \in KAN,$ it follows that both sequences $b_i a_i^{-1}$ and $n_i$ are bounded. Hence for all large $i\ge 1$, $h n_i^{-1} b_i^{-1} a_i o \in B(o, R_i)$ and hence $h P_{\theta} = h n_i^{-1}b_i^{-1} P_{\theta} \in O_{R_i}^{\theta}(hn_i^{-1}b_i^{-1}o, o)$. Since $h n_i^{-1} b_i^{-1} = a_i^{-1}k_i$, we have $h P_{\theta} \in O_{R_i}^{\theta}(a_i^{-1} o, o)$, proving the claim.

Since $O_{R_i}^{\theta}(\ga_i^{-1} o, o)  = \ell_i O_{R_i}^{\theta}(a_i^{-1}o, o)$ and $\ell_i\to \ell$, it follows that $\nu(\ell N_{\theta}^+P_{\theta}) = 0$.
Since $\ell N_{\theta}^+P_{\theta}$
is Zariski open in $\F_\theta$, it follows that $\text{supp }\nu \cap\ell  N_{\theta}^+P_{\theta}=\emptyset$.
This is a contradiction to the hypothesis. Hence this proves (1).
 To see (2), let $\ga \in \Ga$ and $r \ge  R$. By Lemma \ref{lem.buseandcartan}, for all $\xi \in O_r^{\theta}(\ga^{-1} o, o)$, we have $\| \beta_{\xi}^{\theta}(\ga^{-1} o, o) - \mu_{\theta}(\ga) \| \le \kappa r$. Since $\nu(O_r^{\theta}(o, \ga o)) = \int_{O_r^{\theta}(\ga^{-1} o, o)} e^{- \psi (\beta_{\xi}^{\theta}(\ga^{-1} o, o))} d\nu (\xi),$ (2) follows from (1).

If $\Ga$ is Zariski dense, then $\La_{\theta}$ is Zariski dense in $\F_\theta$ and is contained in
  $\supp\,\nu$. Hence any $\Ga$-conformal measure $\nu$ satisfies the hypothesis.

For the last claim in the statement, letting $\Ga$ be a  $\theta$-transverse subgroup and $\nu$ a $(\Ga, \psi)$-conformal measure such that for any $\eta \in \La_{\i(\theta)}$, $(\xi, \eta) \in \F_{\theta}^{(2)}$ for some $\xi \in \supp \nu$, it suffices to show that  $\inf_{\ga \in \Ga} \nu(O_R^{\theta}(\ga o, o)) > 0$.
If not, there exist $R_i \to \infty$ and $\ga_i \in \Ga$ such that $\nu(O_{R_i}^{\theta}(\ga_i^{-1} o, o)) < 1/i$ for all $i \ge 1$.
Write the Cartan decomposition $\ga_i = k_i' a_i \ell_i^{-1} \in KA^+K$ and assume  $\ell_i \to \ell \in K$ as $i \to \infty$. By the same argument as above, we have $\supp \nu \cap \ell N_{\theta}^+ P_{\theta} = \emptyset$. By \eqref{op}, this implies that every element of  $\supp \nu$ is not in general position with $\ell w_0 P_{\i(\theta)}$. On the other hand, it follows from $\ga_i^{-1} = \ell_i w_0 (w_0^{-1} a_i^{-1} w_0) w_0^{-1} k_i'^{-1}$ for all $i \ge 1$ that $\ell w_0 P_{\i(\theta)} = \lim_i \ga_i^{-1} \in \La_{\i(\theta)}$.  By the hypothesis on $\supp \nu$, there exists an element of $\supp \nu$ in general position with $\ell w_0 P_{\i(\theta)}$. This contradicts  $\supp \nu \cap \ell N_{\theta}^+ P_{\theta} = \emptyset$. This finishes the proof.
\end{proof}

\begin{theorem}\label{tgg0}     Let $\Ga$ be a Zariski dense $\ts$-transverse subgroup of $G$. If there exists a $(\Ga, \psi)$-conformal measure $\nu$ on $\F_{\theta}$ for a $(\Ga, \theta)$-proper $\psi \in \fa_{\theta}^*$,
  then $$\delta_\psi \le 1.$$
\end{theorem}
\begin{proof}
For each $n \in \Z$, we set $$\Ga_{\psi, n}:= \{ \ga \in \Ga : n \le \psi(\mu_{\theta}(\ga)) < n + 1\}.$$
    Since $\psi$ is $(\Ga, \theta)$-proper,
    $\bigcup_{n<0} \Ga_{\psi, n}$ is a finite subset, and hence
    can be ignored in the arguments below. Let $\nu$ be a $(\Ga, \psi)$-conformal measure.
    We fix a sufficiently large $R > 0$ satisfying the conclusion of Lemma \ref{lem.shadow} for $\nu$. Since $\psi$ is a $(\Ga, \theta)$-proper linear form, by Proposition \ref{prop.mult}, we have that for all $n\in \mathbb N$,
    $$1 \gg \sum_{\ga \in \Ga_{\psi, n}} \nu(O_R^{\theta}(o, \ga o)) \gg \sum_{\ga \in \Ga_{\psi, n}} e^{-\psi(\mu_{\theta}(\ga))} \ge e^{-(n+1)} \# \Ga_{\psi, n}$$ where the implied constants do not depend on $n$. It implies $$\# \Ga_{\psi, n} \ll e^{n+1}\quad\text{for each $n \ge 0$. }$$ Therefore, we have (cf. \cite[Lemma 3.1.1]{Quint2002divergence})
    \be \label{eqn.countinginnote}
    \begin{aligned}     \delta_{\psi} & \le \limsup_{N \to \infty}{\log \# \{ \ga \in \Ga : \psi(\mu_{\theta}(\ga)) < N\} \over N} \\
      & \le \limsup_{N \to \infty} {1 \over N} \log \sum_{ 0 \le n < N} e^{n+1} =1. 
    \end{aligned}
    \ee 
    Hence the claim follows.
\end{proof}

\subsection*{Proof of Theorem \ref{tgg}}
  By Lemma \ref{lem.tent} and Theorem \ref{tgg0}, we have that $\delta_{\psi} \le 1$ and $\delta_{\psi}\psi$ is tangent to $\psi_{\Ga}^{\theta}$, and therefore we have $$\delta_{\psi} \psi\ge  \psi_{\Ga}^{\theta} .$$
    Since $\psi$ is $(\Ga, \theta)$-proper, $\psi\ge 0$ on $\L_\theta$ by Lemma \ref{lem.propform} and hence $\psi\ge \delta_\psi \psi $ on $\L_\theta$. Therefore $\psi\ge \psi_\Ga^{\theta}$ on $\L_\theta$.
    Since $\psi_{\Ga}^{\theta} = -\infty$ outside of $\L_{\theta}$, 
    $\psi\ge \psi_\Ga^{\theta}$ on $\fa_\theta$. 
If $\sum_{\ga\in \Ga} e^{-\psi (\mu_\theta(\ga))}=\infty$ in addition, then $\delta_\psi\ge 1$ and hence $\delta_\psi=1$. In particular, $\psi = \delta_{\psi} \psi$ is tangent to $\psi_{\Ga}^{\theta}$.
Therefore this finishes the proof.
\qed

\section{Divergence of Poincar\'e series and conical sets} \label{sec.supp}
Let $\psi\in \fa_{\theta}^*$ and $\Ga<G$ be discrete subgroup.
Denote by ${\mathsf M}^\theta_\psi={\mathsf M}^\theta_{\Ga, \psi}$
the collection of all  $(\Ga,\psi)$-conformal (probability) measures on $\ft$. 
Define the following subset of $\mathsf{M}^{\theta}_{\psi}$:
$$
\mathsf{N}_{\Ga, \psi}^{\theta} = \mathsf{N}_{\psi}^{\theta} := \left\{ \nu \in {\mathsf M}^\theta_\psi : \begin{matrix}
\text{either }
\nu(\La_{\theta}) = 1 \text{ and } \# \supp \nu \ge 2, \quad  \text{or} \\
 (\supp \nu|_{\F_{\theta} - \La_{\theta}}, \eta) \cap  \F_{\theta}^{(2)} \neq \emptyset \text{ for all }  \eta \in \La_{\i(\theta)} 
 \end{matrix}
 \right\}.
$$
The reason for this definition is to guarantee that  for $\Ga$  $\theta$-transverse,  the shadow lemma (Lemma \ref{lem.shadow}) holds for any $\nu \in \mathsf N_{\psi}^\theta$ as well as its restriction $\nu|_{\F_{\theta} - \La_{\theta}}$ (if non-zero).

\begin{lem} \label{lem.npsiismpsi}
  \begin{enumerate}
      \item 
  If $\Ga$ is Zariski dense, then $$\mathsf{N}_{\psi}^{\theta} = \mathsf{M}_{\psi}^{\theta}.$$ 
  \item  If $\Ga$ is non-elementary $\theta$-transverse, then
  $$\mathsf{N}_{\psi}^{\theta} \supset \{\nu \in \mathsf{M}_{\psi}^{\theta}:\nu(\La_\theta)=1\}. $$
  \end{enumerate}
\end{lem}

\begin{proof}
    Since $\# \La_{\theta} \ge 3$ for a non-elementary $\theta$-transverse subgroup $\Ga$, (2) is straightforward.  For (1), suppose that $\Ga$ is a Zariski dense discrete subgroup. Then for any $\nu \in \mathsf{M}_{\psi}^{\theta}$, $\supp \nu$ is a closed $\Ga$-invariant set, and hence is Zariski dense in $\F_\theta$ (Lemma \ref{lq}). Therefore
if $\nu(\La_{\theta}) = 1$, then $\nu \in \mathsf{N}_{\psi}^{\theta}$. Otherwise, we have $\nu(\F_{\theta} - \La_{\theta}) > 0$, and $\supp \nu|_{\F_{\theta} - \La_{\theta}}$ is a non-empty closed $\Ga$-invariant set, and thus Zariski dense in $\F_\theta$. Given $\eta \in \La_{\i(\theta)}$, the set $\{ \xi \in \F_{\theta} : (\xi, \eta) \in \F_{\theta}^{(2)} \}$ is a  Zariski open subset of $\F_{\theta}$ and hence $(\supp \nu|_{\F_{\theta} - \La_{\theta}} , \eta) \cap \F_{\theta}^{(2)} \neq \emptyset$, finishing the proof.
\end{proof}

The main goal of this section is to prove the following theorem and discuss its applications. Note that we do not assume that $\psi$ is $(\Ga, \theta)$-proper in the following theorem.

\begin{theorem} \label{thm.divthensupp} \label{thm.convthen0} \label{ggg}

Let $\Ga$ be any $\theta$-transverse subgroup (which may be elementary).
Then the following are equivalent:
\begin{enumerate}
    \item  $\sum_{\ga \in \Ga} e^{-\psi(\mu_{\theta}(\ga))} = \infty$ \hspace{4em} (resp. $\sum_{\ga \in \Ga} e^{-\psi(\mu_{\theta}(\ga))} <\infty$)
    \item $\nu(\La^{\mathsf{con}}_{\theta}) = 1 $ for all $\nu \in \mathsf{N}^{\theta}_{\psi}$ \hspace{1.8em} (resp. $\nu(\La^{\mathsf{con}}_{\theta}) = 0 $ for all $\nu \in \mathsf{N}^{\theta}_{\psi}$).
\end{enumerate}

\end{theorem}

In the rest of this section, suppose that $\Ga$ is $\theta$-transverse.
We make the following simple observation:
\begin{lemma} \label{lem.posthenfull}\label{dicc}
  Suppose that $\nu(\La^{\mathsf{con}}_{\theta}) > 0$ for all $\nu\in \mathsf N_\psi^{\theta}$. 
 Then $$\nu(\La^{\mathsf{con}}_{\theta}) = 1\quad\text{for all $\nu\in \mathsf N_\psi^{\theta}$. }$$
\end{lemma}
\begin{proof}
    Suppose that for some $\nu \in \mathsf{N}_{\psi}^{\theta}$, we have
    $0 < \nu(\La^{\mathsf{con}}_{\theta}) < 1$. Then $\nu' := {1 \over \nu(\F_{\theta} - \La^{\mathsf{con}}_{\theta})}\nu|_{\F_{\theta} - \La^{\mathsf{con}}_{\theta}}$ belongs to $\mathsf{M}_{\psi}^{\theta}$. We now show that $\nu' \in \mathsf{N}_{\psi}^{\theta}$. There are two cases:
    \begin{enumerate}
        \item If $\nu$ satisfies that $(\supp \nu|_{\F_{\theta} - \La_{\theta}}, \eta) \cap \F_{\theta}^{(2)} \neq \emptyset$ for all $\eta \in \La_{\i(\theta)}$, then the same holds for $\nu'$, so $\nu' \in \mathsf{N}_{\psi}^{\theta}$.

        \item Otherwise, $\nu(\La_{\theta}) = 1$ and $\#\supp \nu \ge 2$. We consider the following two subcases:
        \begin{itemize}
            \item If $\Ga$ is non-elementary, then $\supp \nu' = \La_{\theta}$. Since $\# \La_{\theta} = \infty$ in this case, we have $\nu \in \mathsf{N}_{\psi}^{\theta}$.

            \item 
        If $\Ga$ is elementary, then $\# \La_{\theta} \le 2$. Since $\nu(\La_{\theta}) = 1$ and $\# \supp \nu \ge 2$, we have $\supp \nu = \La_{\theta}$ and $\# \La_{\theta} = 2$. Since $0 < \nu(\La_{\theta}^{\sf con}) < 1$, we must have that $\# \La_{\theta}^{\sf con} = 1$. Since both $\La_{\theta}$ and $\La_{\theta}^{\sf con}$ are $\Ga$-invariant, this implies that each point of $\La_{\theta}$ is $\Ga$-invariant. Now for $\xi \in \La_{\theta} - \La_{\theta}^{\sf con}$, let $\ga_i \in \Ga$ be a sequence that converges to $\xi$. For $\eta \in \La_{\i(\theta)}$ such that $(\xi, \eta) \in \F_{\theta}^{(2)}$, we have $\ga_i^{-1}(\xi, \eta) = (\xi, \eta) \in \F_{\theta}^{(2)}$, and therefore the convergence $\ga_i \to \xi$ is conical by Lemma \ref{lem.klpconical}. This contradicts $\xi \notin \La_{\theta}^{\sf con}$.

        \end{itemize}
    \end{enumerate}
    Therefore, in any case, we have $\nu' \in \mathsf{N}_{\psi}^{\theta}$. On the other hand, $\nu'(\La_{\theta}^{\sf con}) = 0$, contradicting the hypothesis. This finishes the proof.
\end{proof}

We will use the following: \begin{lemma}[Kochen-Stone Lemma \cite{Kochen1964note}] \label{lem.KS}
Let $(Z, \nu)$ be a finite measure space. If $\{\mathsf A_n\}$ is a sequence of measurable subsets of $Z$ such that
\be \label{eqn.KS}
\sum_{n = 1}^{\infty} \nu(\mathsf A_n) = \infty \quad \mbox{and} \quad \liminf_{N \to \infty} {\sum_{m = 1}^N \sum_{n = 1}^N \nu(\mathsf A_n \cap \mathsf A_m) \over \left( \sum_{n = 1}^N \nu(\mathsf A_n) \right)^2} < \infty,
\ee then $\nu (\limsup_n \mathsf A_n) >0$.
\end{lemma}

\subsection*{Proof of Theorem \ref{ggg}}
 Suppose that $\sum_{\ga \in \Ga} e^{-\psi(\mu_{\theta}(\ga))} = \infty$.
By Lemma \ref{lem.posthenfull}, it suffices to show that
$\nu(\La^{\mathsf{con}}_{\theta}) >0 $ for all $\nu \in \mathsf{N}_{\psi}^{\theta}$.
 Let $\nu \in \mathsf{N}_{\psi}^{\theta}$. Since $\Ga$ is $\theta$-transverse, it follows from the definition of $\mathsf{N}_{\psi}^{\theta}$ that $\nu$ satisfies Lemma \ref{lem.shadow}.

 We fix $\alpha \in \theta$. 
 Since $\Ga$ is $\theta$-regular, $\alpha \in \theta$ is $(\Ga, \theta)$-proper; in particular, $\alpha(\mu_\theta(\Ga))$ is a discrete closed subset of $[0, \infty)$. Therefore we may enumerate $\Ga = \{{\ga}_1, {\ga}_2, \cdots \}$ so that $\alpha(\mu_{\theta}({\ga}_n)) \le \alpha(\mu_{\theta}({\ga}_{n+1}))$ for all $n \in \N$. Fix a sufficiently large $R$ which satisfies the conclusion of Lemma \ref{lem.shadow}. Setting $\mathsf A_n:=O_R^{\theta}(o, {\ga}_n o)$, we then have
 $$\sum_{n = 1}^{\infty} \nu(\mathsf A_n) \gg  \sum_{\ga \in \Ga} e^{-\psi(\mu_{\theta}(\ga))} =  \infty$$
where the implied constant depends only on $R$.
Since $\limsup_n \mathsf A_n \subset \La^{\mathsf{con}}_{\theta}$,
by Lemma \ref{lem.KS}, it suffices to show that
\be\label{ann} \liminf_{N \to \infty} {\sum_{m = 1}^N \sum_{n = 1}^N \nu(\mathsf A_n \cap \mathsf A_m) \over \left( \sum_{n = 1}^N \nu(\mathsf A_n) \right)^2} < \infty .\ee 

Set  $Q := K A_R^+K$ where $A_R^+=\{a\in A^+:\|\log a\|\le R\}$ and $C_0=C_0(Q)$ be as in
 Lemma \ref{cor.approxcartan}. Note that $Q = \{g \in G : d(o, g o) \le R \}$.
Define $$T_N := \max \{ n \in \N : \alpha( \mu_{\theta}({\ga}_n) ) \le \alpha(\mu_{\theta}({\ga}_N) ) + \|\alpha\|C_0\}$$
for each $N\ge 1$.  Clearly, $N \le T_N$. Unless mentioned otherwise, all implied constants in this proof are independent of $N$. Since $\Ga$ is $\theta$-regular, $\alpha|_{\fa_\theta}$ is $(\Ga,\theta)$-proper. 
Proposition \ref{prop.mult} implies that the collection $\sf A_n$, $N\le n\le T_N$, has multiplicity at most $q=q(\alpha, R, \|\alpha\| C_0)$, and hence
$$\sum_{ N\le n \le  T_N} \nu({\sf A_n})  \le  q\cdot \nu(\F_\theta).$$
 Therefore by Lemma \ref{lem.shadow}, we have that for all $N\ge 1$,
 \begin{align*}
   & \left| \sum_{n = 1}^{T_N} e^{-\psi(\mu_{\theta}({\ga}_n))} - \sum_{n = 1}^N e^{-\psi(\mu_{\theta}({\ga}_n))} \right|\ll   \sum_{n = N + 1}^{T_N} \nu(\sf A_n)
   \\ & \ll 
    \nu(\F_\theta)= e^{\psi(\mu_{\theta}({\ga}_1))}e^{-\psi(\mu_{\theta}({\ga}_1))}\le  e^{\psi(\mu_{\theta}({\ga}_1))} \sum_{n = 1}^N e^{-\psi(\mu_{\theta}({\ga}_n))}
    \end{align*}
    with all implied constants independent of $N$.
Therefore we have:
\be \label{lem.TNtoN} \sum_{n = 1}^{T_N} e^{-\psi(\mu_{\theta}({\ga}_n))} \ll  \sum_{n = 1}^N e^{-\psi(\mu_{\theta}({\ga}_n))}.\ee

 Fix $N \in \N$. If $\mathsf A_n\cap \mathsf A_m \neq \emptyset$ for some $n, m \le N$, 
then there exist $k \in K$ and $m_{\theta} \in M_{\theta}$ such that $d(kA^+o, {\ga}_n o) < R$ and $d(km_{\theta}A^+o, {\ga}_m o) < R$.
Since $K\subset Q$, it follows that 
$$Q \cap {\ga}_n Q a_n^{-1} \cap {\ga}_m Q a_m^{-1} m_{\theta}^{-1} \neq \emptyset$$ for some $a_n, a_m \in A^+$. Hence, setting $$
\begin{aligned}
    E_{1} & = \{(n, m) : n, m \le N \mbox{ and }\|\mu_{\theta}({\ga}_n) - (\mu_{\theta}({\ga}_m) + \mu_{\theta}({\ga}_m^{-1}{\ga}_n))\| \le C_0\}, \\
    E_{2} & = \{(n, m) : n, m \le N \mbox{ and } \|\mu_{\theta}({\ga}_m) - (\mu_{\theta}({\ga}_n) + \mu_{\theta}({\ga}_n^{-1}{\ga}_m))\| \le C_0\},
\end{aligned}$$
we get from Lemma \ref{cor.approxcartan} 
 that \be \label{eqn.part1and2}
\sum_{n, m \le N}  \nu(\sa_n\cap \sa_m) 
 \le  \sum_{(n, m) \in E_{1}} \nu(\sa_n) + \sum_{(n, m) \in E_{2}} \nu(\sa_m).\ee

For all $(n,m)\in E_1$,
we have \be\label{bes} 
\begin{aligned}
\alpha(\mu_{\theta}({\ga}_m^{-1}{\ga}_n)) & \le \alpha(\mu_\theta ({\ga}_m) + \mu_{\theta}({\ga}_m^{-1}{\ga}_n)) \\
& = \alpha(\mu_\theta ({\ga}_m) + \mu_{\theta}({\ga}_m^{-1}{\ga}_n) - \mu_{\theta}(\ga_n)) + \alpha(\mu_{\theta}(\ga_n)) \\
& \le \| \alpha \| C_0 + \alpha(\mu_{\theta}(\ga_n)).
\end{aligned}
\ee  
Therefore, by Lemma \ref{lem.shadow},
\be \label{eqn.part1}
\begin{aligned}
    \sum_{(n, m) \in E_{1}} \nu(\sa_n) & \ll \sum_{(n, m) \in E_{1}} e^{-\psi(\mu_{\theta}({\ga}_n))} \\
    & \ll  \sum_{(n, m) \in E_{1}} e^{-\psi(\mu_{\theta}({\ga}_m))} e^{-\psi(\mu_{\theta}({\ga}_m^{-1}{\ga}_n))}
    \\ &\le  \sum_{m = 1}^N \sum_{j = 1}^{T_N} e^{-\psi(\mu_{\theta}({\ga}_m))} e^{-\psi (\mu_{\theta}({\ga}_j))};
\end{aligned}
\ee
the last inequality follows because, for each fixed $1\le m\le N$, 
the correspondence $n \leftrightarrow \ga_m^{-1}\ga_n$ is one-to-one and when $(n, m) \in E_1$, ${\ga}_j={\ga}_m^{-1}{\ga}_n $ for some $j\le T_n\le T_N$ by \eqref{bes}. 
Similarly, we have $$\sum_{(n, m) \in E_{2}} \nu(\sa_m) \ll   \sum_{n = 1}^N \sum_{j = 1}^{T_N} e^{-\psi(\mu_{\theta}({\ga}_n))} e^{-\psi (\mu_{\theta}({\ga}_j))}.$$

By \eqref{eqn.part1and2}, we have 
\begin{align*}
 \sum_{n, m \le N}  \nu(\sa_n\cap \sa_m) &\ll \sum_{n = 1}^N \sum_{j = 1}^{T_N} e^{-\psi(\mu_{\theta}({\ga}_n))} e^{-\psi (\mu_{\theta}({\ga}_j))} \\
    & =  \left(\sum_{n = 1}^N e^{-\psi(\mu_{\theta}({\ga}_n))} \right) \left(\sum_{n = 1}^{T_N}  e^{-\psi (\mu_{\theta}({\ga}_n))} \right)\\ &\ll 
    \left(\sum_{n = 1}^N e^{-\psi(\mu_{\theta}({\ga}_n))} \right)^2 \ll 
\left( \sum_{n = 1}^N \nu(\mathsf A_n) \right)^2
\end{align*}
where we have applied \eqref{lem.TNtoN} for the second last inequality and  Lemma \ref{lem.shadow} for the last inequality. Hence \eqref{ann} is verified, completing the proof of the first statement.

We now suppose that $\sum_{\ga \in \Ga} e^{-\psi(\mu_{\theta}(\ga))} < \infty$. 
Consider the following increasing sequence
$$\La_{\theta}^N =  \left\{ \xi \in \F_{\theta} : \begin{matrix}
    \exists \text{ infinite sequence } \ga_n \in \Ga \text{ s.t.} \\
    \xi \in \bigcap_{n \ge 1} O_N^{\theta}(o, \ga_n o)
    \end{matrix} \right\}, \quad N\ge 1.$$
Since $\La^{\mathsf{con}}_{\theta} = \bigcup_N \La_{\theta}^N$,
 it suffices to show $\nu(\La_{\theta}^N) = 0$ for all sufficiently large $N\ge1$.
Since 
    $$\La_{\theta}^N \subset \bigcup_{\ga \in \Ga, \|\mu_{\theta}(\ga)\| > t} O^\theta_N(o, \ga o)$$
    for any $t > 0$,
we get from Lemma \ref{lem.shadow} that 
for all $t>0$, $$\nu(\La_{\theta}^N) \ll \sum_{\ga \in \Ga, \|\mu_{\theta}(\ga)\| > t} e^{-\psi(\mu_{\theta}(\ga))}$$
where the implied constant depends only on $N$.
Since $\sum_{\ga \in \Ga} e^{-\psi(\mu_{\theta}(\ga))} <  \infty$ implies that 
$\lim_{t\to \infty} \sum_{\ga \in \Ga, \|\mu_{\theta}(\ga)\| > t} e^{-\psi(\mu_{\theta}(\ga))}=0$, we have
$\nu(\La_{\theta}^N) =0$, finishing the proof.
\qed

 \subsection*{Comparing with $\psi_\Ga$}
 Quint showed that for a Zariski dense discrete subgroup $\Ga<G$, the existence of a $(\Ga, \psi)$-conformal measure on $\F_\theta$ for $\psi\in \fa_\theta^*$ implies the inequality \be\label{quint} \psi \circ p_\theta+ 2\rho_{\Pi - \theta} \ge \psi_\Ga \quad \text{on $\fa$,}\ee
where $2\rho_{\Pi - \theta}$ is the sum of all positive roots which can be written as $\mathbb Z$-linear combinations of elements of $\Pi - \theta$ (counted with multiplicity) \cite[Theorem 8.1]{Quint2002Mesures}. For $\theta$-transverse subgroups, Theorem \ref{main}  and \eqref{qq} imply that the term $2\rho_{\Pi - \theta}$ turns out to be redundant:

 \begin{cor}\label{coo}
  Let $\G<G$ be a  Zariski dense $\t$-transverse subgroup and $\psi\in \fa_\theta^*$ be $(\Ga, \theta)$-proper.
If there exists a $(\Ga, \psi)$-conformal measure $\nu$ on $\F_\theta$,
then \be\label{pos} \psi \circ p_\theta \ge \psi_\Ga \quad \mbox{on } \fa.\ee
Moreover, if $\nu(\La_{\theta}^{\mathsf{con}}) > 0$, then $\psi \circ p_\theta$ is tangent to $\psi_\Ga$.
 \end{cor}
 
 \begin{proof}
The first statement follows from Theorem \ref{tgg} and Lemma \ref{qu}. For the second claim, if $\nu(\La_{\theta}^{\mathsf{con}}) > 0$, then we have $\sum_{\ga \in \Ga} e^{-(\psi \circ p_{\theta})(\mu(\ga))} = \infty$ by Theorem \ref{ggg}. If $\psi \circ p_{\theta} $ were strictly bigger than $\psi_{\Ga}$, then  by \cite[Lemma 3.1.3]{Quint2002divergence}
we would have $\sum_{\ga \in \Ga} e^{-(\psi \circ p_{\theta})(\mu(\ga))} < \infty$. Therefore $\psi \circ p_\theta$ must be tangent to $\psi_\Ga$.
\end{proof}

\section{Properly discontinuous actions of $\Ga$}
\label{pdt}
Recall $ \F_{\theta}^{(2)}=\{(\xi, \eta)\in \F_\theta\times \F_{\i(\theta)}: \text{$\xi$, $\eta$ are in general position}\}$ and
consider the action of $G$  on the space $\F_{\theta}^{(2)} \times \fa_{\theta}$ defined as 
\be\label{hopf} g . (\xi, \eta, u) = (g
\xi, g \eta, u + \beta_{\xi}^{\theta}(g^{-1}, e)) \ee 
for all $g\in G$ and $(\xi, \eta, u)\in \F_{\theta}^{(2)} \times \fa_{\theta}$. 
A discrete subgroup $\Ga<G$ preserves the subspace 
$\La_{\theta}^{(2)} \times \fa_{\theta}$
where $$\La_{\theta}^{(2)} = 
 (\La_{\theta} \times \La_{\i(\theta)})\cap \F_\theta^{(2)} .$$

 When $\theta=\Pi$,
the Hopf parametrization of $G/M$ gives a $G$-equivariant homeomorphism between $\F^{(2)}\times \fa $ and  $G/M$, and hence any discrete subgroup $\Ga<G$ acts properly discontinuously on $\F^{(2)} \times \fa$ and hence the quotient space 
$\Ga\ba \La_{\Pi}^{(2)} \times \fa$ is a locally compact Hausdorff space. For a general $\theta$, this is not the case.
The aim of this section is to establish the following two theorems on
properly discontinuous actions of $\theta$-transverse subgroups.

\begin{theorem}\label{tproper}
    If $\Ga$ is a non-elementary $\theta$-transverse subgroup, the $\Ga$-action on $\La_{\theta}^{(2)} \times \fa_{\theta}$ is properly discontinuous and hence $$\Omega_\theta:=\Ga \ba \La_{\theta}^{(2)} \times \fa_{\theta}$$ is a locally compact Hausdorff space.
\end{theorem}

For a $(\Ga,\theta)$-proper form $\varphi\in \fa_\theta^*$,
consider the $\Ga$-action 
\be\label{hopf3} \ga . (\xi, \eta, s) = (\ga
\xi, \ga \eta, s + \varphi(\beta_{\xi}^{\theta}(\ga^{-1}, e))) \ee 
for all $\ga\in \Ga$ and $(\xi, \eta, s)\in \La_{\theta}^{(2)} \times \br$.

\begin{theorem}\label{tproper2}
    Let $\Ga$ be a non-elementary $\theta$-transverse subgroup of $G$ and $\varphi\in \fa_\theta^*$ a $(\Ga, \theta)$-proper form.  Then the action $\Ga $ on $\La_{\theta}^{(2)} \times \br$ given by \eqref{hopf3} is properly discontinuous and hence $$\Omega_{\varphi}:=\Ga \ba \La_{\theta}^{(2)} \times \br$$ is a locally compact Hausdorff space. Moreover, 
    $\Omega_\varphi$ is compact if and only if $\Ga$ is $\theta$-Anosov.
\end{theorem}

\begin{Def}\label{conver}
Let $Z$ be a compact metrizable space with at least $3$
points.
 An action of a countable group $\Ga$ on $Z$ by homeomorphisms
is called a {\it convergence group action} if for any sequence of distinct elements $\ga_n \in \Ga$, there exist  a subsequence $\ga_{n_k}$ and $a, b \in Z$ such that as $k \to \infty$, 
$\ga_{n_k}(z) $ converges to $ a $ for all $z\in Z-\{b\}$, uniformly on compact subsets.
\end{Def}

We will use the following property of a $\theta$-transverse subgroup:
\begin{prop} \cite[Theorem 4.16]{Kapovich2017anosov} \label{prop.transisconv}
For a $\theta$-transverse subgroup $\Ga$,
the action of  $\Ga$ on $\La_{\theta}$ is a convergence group action.
\end{prop}
It is also proved in \cite{Kapovich2017anosov} that
$\La_{\theta}$ is same as the limit set as the convergence group action; this also follows from Lemma \ref{lem.29inv}.
In particular, if $\Ga$ is non-elementary, then the $\Ga$-action on $\La_{\theta}$ is minimal.

The following observation is useful to transfer statements from $\theta$ symmetric to general $\theta$.

\begin{lem}\label{bij}
 Suppose that $\Ga$ is $\theta$-antipodal.
For any $\theta_1\subset \theta_2\subset \theta\cup \i(\theta)$,
the projection map
$p: \La_{\t_2}\to \La_{\theta_1}$ given by $gP_{\t_2}\to gP_{\t_1}$
is a $\Ga$-equivariant homeomorphism. In particular, 
for any $(\Ga, \psi)$-conformal measure $\nu$ supported on $\La_{\theta_1}$ for $\psi\in \fa_{\theta_1}^*\subset \fa_{\theta_2}^*$, the pull back $p^*\nu$ is a $(\Ga, \psi)$-conformal measure on $\La_{\theta_2}$.
\end{lem}

\begin{proof}
It suffices to show that $p$ is injective when $\theta_2=\theta\cup \i(\theta)$. Suppose that $\xi\ne  \eta\in \La_{\t \cup \i(\t)}$. By the $\theta$-antipodality of $\Ga$,
 $\xi=g P_{\t\cup \i(\t)}$ and $\eta= g w_0 P_{\t\cup \i(\t)}$ for some $g\in G$. Then $p(\xi)=g P_{\t_1}$ and  $p(\eta)=
 g w_0 P_{\t_1}$, and hence $p(\xi)\ne p(\eta)$, showing that $p$ is injective.
\end{proof}

The following observation will be useful:

\begin{lemma} \label{lem.convseqconvseq} \label{cor.convinpair}
    Let $\Ga$ be a non-elementary $\theta$-transverse subgroup and $\ga_i \in \Ga$ an infinite sequence. Let $(\xi_i, \eta_i) \in \La_{\theta}^{(2)}$ be a convergent sequence in $\La_{\theta}^{(2)}$. If the sequence $\ga_i (\xi_i, \eta_i)$ converges in $\La_{\theta}^{(2)}$, then there exists $R > 0$ so that either $$\begin{aligned} 
    \xi_i &\in O_R^{\theta}(o, \ga_i^{-1} o) & \text{for all }i \ge 1; &\text{ or} \\
    \eta_i & \in O_R^{\i(\theta)}(o, \ga_i^{-1} o) &\text{for all } i \ge 1.&
    \end{aligned}$$

      In particular,
      if the sequence $\gamma_i (\xi, \eta) \in \La_{\theta}^{(2)} $ converges in
        $\La_{\theta}^{(2)}$ for some  $(\xi, \eta) \in \La_{\theta}^{(2)}$,
       then  $\ga_i^{-1}$ converges conically either to $\xi$ or $\eta$.
\end{lemma}

\begin{proof}
    Set  $(\xi, \eta) = \lim_i (\xi_i, \eta_i) \in \La_{\theta}^{(2)}$ and $(\xi_0, \eta_0) = \lim_i \ga_i (\xi_i, \eta_i) \in \La_{\theta}^{(2)}$.
Since the projections $\La_{\theta \cup \i(\theta)} \to \La_{\theta}$ and $\La_{\theta \cup \i(\theta)} \to \La_{\i(\theta)}$ are $\Ga$-equivariant homeomorphisms by Lemma \ref{bij}, we also let $\xi', \xi_0', \xi_i' \in \La_{\theta \cup \i(\theta)}$ be the preimages of $\xi$, $\xi_0$, and $\xi_i$ for all $i \ge 1$ under the projection $\La_{\theta \cup \i(\theta)} \to \La_{\theta}$ respectively, and similarly $\eta', \eta_0', \eta_i' \in \La_{\theta \cup \i(\theta)}$ the preimages of $\eta$, $\eta_0$, and $\eta_i$. Note that $\xi' \neq \eta'$, $\xi_0' \neq \eta_0'$, and $\xi_i' \neq \eta_i'$ for all $i \ge 1$ and $\xi_i' \to \xi'$, $\eta_i' \to \eta'$, $\ga_i \xi_i' \to \xi_0'$, and $\ga_i\eta_i'\to \eta_0'$ as $i\to \infty$. 
    
    Since the action of $\Ga$ on $\La_{\theta \cup \i(\theta)}$ is a convergence group action by Proposition \ref{prop.transisconv}, there exist $a, b \in \La_{\theta \cup \i(\theta)}$ such that
    \be\label{exp0}\ga_i|_{\La_{\theta \cup \i(\theta)} - \{b\}} \to a\ee  uniformly on compact subsets, after passing to a subsequence. That is, for any compact subsets $C_a \subset \La_{\theta \cup \i(\theta)} - \{a\}$ and $C_b \subset \La_{\theta \cup \i(\theta)} - \{b\}$, 
    $$\# \{ \ga_i : \ga_i C_b \cap C_a \neq \emptyset\} < \infty,$$ 
   or equivalently $\# \{ \ga_i^{-1} : \ga_i^{-1} C_a \cap C_b \neq \emptyset\} < \infty$. Therefore we have, as $i\to \infty$,  \be\label{exp}\ga_i^{-1} |_{\La_{\theta \cup \i(\theta)} - \{a\}} \to b\ee  uniformly on compact subsets.

 We claim that \be\label{ab} (a, b) = (\eta_0', \xi')\quad\text{or}\quad (a, b) = (\xi_0', \eta').\ee 
 Suppose $\xi'\ne b$. Excluding finitely many elements from $\{\xi_i' : i \ge 1 \}$, we may assume that $\{\xi_i' : i \ge 1\} \cup \{\xi'\}$ is a compact subset of $\La_{\theta \cup \i(\theta)} - \{b\}$. Hence \eqref{exp0} implies that $\xi_0' = \lim_i \ga_i \xi_i' = a$.
 If $\eta'$ were not equal to $ b$, then we may also assume that $\{\eta_i' : i \ge 1\} \cup \{\eta'\}$ is a compact subset of $\La_{\theta \cup \i(\theta)} - \{b\}$ and hence \eqref{exp0} implies $\eta_0' = \lim_i \ga_i \eta_i' =a$. Since $\xi_0'\ne \eta_0'$, this is a contradiction. This implies $\eta'=b$.
 Now suppose that $\xi'=b$. 
 Since $\eta' \neq \xi' = b$, we have $\eta_0' = \lim_i \ga_i \eta_i' = a$ by the above argument.
 This proves the claim.

 Now \eqref{exp} and \eqref{ab} imply that
    \be \label{eqn.propdiscconvergence}
    \ga_i^{-1}|_{\La_{\theta \cup \i(\theta)} - \{\eta_0'\}} \to \xi' \quad \text{or} \quad \ga_i^{-1}|_{\La_{\theta \cup \i(\theta)} - \{\xi_0'\}} \to \eta'
    \ee uniformly on compact subsets.

    Since $\Ga$ is $\theta \cup \i(\theta)$-regular, we may assume that by passing to a subsequence, the sequence $\ga_i^{-1}$ converges to some point, say, $z=\lim_i \ga_i^{-1}$, in $ \La_{\theta \cup \i(\theta)}$ in the sense of Definition \ref{fc}. We  claim that $z$ is either $\xi'$ or $\eta'$.
    Write $\ga_i^{-1} = k_i b_i \ell_i^{-1} \in KA^+K$ using the Cartan decomposition. By passing to a subsequence, we may assume that 
    $k_i\to k_0\in K$ and $\ell_i\to \ell_0\in K$. Choose $x\in \La_{\theta\cup \i(\theta)} - \{\eta_0', \xi_0', \}$ in general position with $\ell_0 w_0 P_{\theta\cup \i(\theta)} = \lim_i \ga_i$, which is possible by  the $\theta$-antipodality and non-elementary hypothesis of $\Ga$. Since $\Ga$ is $\theta \cup \i(\theta)$-regular, by Lemma \ref{lem.cptcartan},
    we have $\min_{\alpha\in \theta \cup \i(\theta)}\alpha(\log b_i)\to \infty$. Hence, by Lemma \ref{lem.29inv}, we have $$\ga_i^{-1} x \to z=k_0P_{\theta\cup \i(\theta)} .$$ Therefore, it follows from \eqref{eqn.propdiscconvergence} that $z = \xi' \text{ or } \eta'$.
    
    If $\lim_i \ga_i^{-1} = \xi'$, then by Lemma \ref{lem.eventualshadow}, there exists $R_1 > 0$ such that $\xi_i' \in O_{R_1}^{\theta \cup \i(\theta)} (o, \ga_i^{-1} o)$ for all $i \ge 1$. Otherwise, if $\lim_i \ga_i^{-1} = \eta'$, then we apply Lemma \ref{lem.eventualshadow} to the sequence $(\eta_i', \xi_i')$ to obtain $R_2 > 0$ such that $\eta_i' \in O_{R_2}^{\theta \cup \i(\theta)} (o, \ga_i^{-1} o)$ for all $i \ge 1$. Setting $R := \max(R_1, R_2)$ and taking the projections $\La_{\theta \cup \i(\theta)} \to \La_{\theta}$ and $\La_{\theta \cup \i(\theta)} \to \La_{\i(\theta)}$, we have either $$\begin{aligned} 
    \xi_i &\in O_R^{\theta}(o, \ga_i^{-1} o) & \text{for all }i \ge 1; &\text{ or} \\
    \eta_i & \in O_R^{\i(\theta)}(o, \ga_i^{-1} o) &\text{for all } i \ge 1,&
    \end{aligned}$$ completing the proof.
\end{proof}

\begin{prop}\label{mainte}

        Let $\Ga$ be a non-elementary $\theta$-transverse subgroup and $\varphi \in \fa_{\theta}^*$ a $(\Ga, \theta)$-proper form.
        Let $\ga_i \in \Ga$ be an infinite sequence and $(\xi_i, \eta_i) \in \La_{\theta}^{(2)}$ a convergent sequence in $\La_{\theta}^{(2)}$.
      If the sequence $\gamma_i (\xi_i, \eta_i) \in \La_{\theta}^{(2)} $ converges in
        $\La_{\theta}^{(2)}$,
       then the sequence
   $\varphi (\beta_{\xi_i}^{\theta}(\ga_i^{-1}, e))$  is  unbounded. In particular, $\beta_{\xi_i}^{\theta}(\ga_i^{-1}, e)$ is unbounded.

\end{prop}
    
\begin{proof}

By Lemma \ref{lem.convseqconvseq}, there exists $R > 0$ so that either $$\begin{aligned} 
    \xi_i &\in O_R^{\theta}(o, \ga_i^{-1} o) & \text{for all }i \ge 1; &\text{ or} \\
    \eta_i & \in O_R^{\i(\theta)}(o, \ga_i^{-1} o) &\text{for all } i \ge 1.&
    \end{aligned}$$

We consider these two cases separately.

\medskip
\noindent{\bf Case A.} 
Suppose that $\xi_i \in O_R^{\theta}(o, \ga_i^{-1} o)$ for all $i \ge 1$. By Lemma \ref{shsh}, we have
$$\sup_i \| \beta_{\xi_i}^{\theta}(e, \ga_i^{-1}) - \mu_{\theta}(\ga_i^{-1})\| < \infty$$ 
and hence
$$\sup_i |\varphi (\beta_{\xi_i}^{\theta}(e, \ga_i^{-1}) - \mu_{\theta}(\ga_i^{-1}))| < \infty.$$ 

The $\theta$-regularity of $\Ga$ implies $\mu_{\theta}(\ga_i^{-1}) \to \infty$ as $i \to \infty$.
Since $\varphi$ is $(\Ga, \theta)$-proper, we have $\varphi(\mu_{\theta}(\ga_i^{-1})) \to \infty$. Therefore $$\varphi (\beta_{\xi_i}^{\theta}(\ga_i^{-1}, e) )= - \varphi(\beta_{\xi_i}^{\theta}(e, \ga_i^{-1})) \to - \infty,$$ as desired. 

\medskip 
 \noindent{\bf Case B.}
Now suppose that $\eta_i \in O_R^{\i(\theta)}(o, \ga_i^{-1} o)$ for all $i \ge 1$. Then there exist a sequence $k_i \in K$ and a sequence $a_i \to \infty$ in $A^+$ such that $\eta_i = k_i P_{\i(\theta)}$ for all $i \ge 1$ and the sequence $\ga_i k_i a_i$ is bounded. By the hypothesis that the sequence $(\xi_i, \eta_i)$ converges in $\La_{\theta}^{(2)}$, there exists a bounded sequence $h_i \in G$ such that $(\xi_i, \eta_i) = h_i L_{\theta}$, which means that $\xi_i = h_i P_{\theta}$ and $\eta_i = h_i w_0 P_{\i(\theta)}$. Since $\eta_i = h_i w_0 P_{\i(\theta)} = k_i P_{\i(\theta)}$ for each $i$, we have $h_i w_0 m_i' p_i = k_i$ for some $m_i' \in M_{\i(\theta)}$ and $p_i \in P$, using $P_{\i(\theta)} = M_{\i(\theta)} P$. Since the sequences $h_i$, $k_i$, and $m_i'$ are bounded, the sequence $p_i \in P$ is bounded as well. This implies that the sequence $a_i^{-1} p_i a_i$ is bounded since $a_i \in A^+$ by Lemma \ref{lem.boundedofparabolic}. Hence it follows from the boundedness of the sequence  $\ga_i k_i a_i = \ga_i h_i w_0 m_i' p_i a_i = \ga_i h_i w_0 m_i' a_i (a_i^{-1} p_i a_i)$ that  $$\text{the sequence } g_i := \ga_i h_i w_0 m_i' a_i \text{ is bounded.}$$

For each $i$, set $m_i=w_0m_i'w_0^{-1}\in M_\theta $.
Then $$\eta_i = h_iw_0 P_{\i(\theta)} = h_i w_0 m_i' P_{\i(\theta)} = h_i m_i w_0 P_{\i(\theta)}, \quad \xi_i = h_iP_{\theta} = h_im_iP_{\theta}$$ and
$$g_i = \ga_i h_i w_0 m_i' a_i = \ga_i h_i m_i w_0 a_i.$$
Using $\xi_i=h_im_i P_\theta$, we have 
$$\begin{aligned}
    \beta_{\xi_i}^{\theta}(\ga_i^{-1}, e) & = \beta_{\ga_i \xi_i}^{\theta}(e, \ga_i) = \beta_{\ga_i \xi_i}^{\theta}(e, g_i) + \beta_{\ga_i \xi_i}^{\theta}(g_i, \ga_i) \\
    & = \beta_{\ga_i \xi_i}^{\theta}(e, g_i) + \beta_{\xi_i}^{\theta}(h_i m_i w_0 a_i, e)
    \\ &= \beta_{\ga_i \xi_i}^{\theta}(e, g_i) + \beta_{P_{\theta}}^{\theta}( w_0 a_i, e) 
   + \beta_{P_\theta}^{\theta}(e, m_i^{-1} h_i^{-1}).
\end{aligned} $$
Since $g_i$ and $m_i^{-1}h_i^{-1}$ are bounded sequences,
the sequences $\beta_{\ga_i \xi_i}^{\theta}(e, g_i)$ and $\beta_{P_\theta}^{\theta}(e, m_i^{-1} h_i^{-1})$ are bounded by \cite[Lemma 5.1]{lee2020invariant}. 

Hence it suffices to show that  as $i\to \infty$, \be\label{betaf} \varphi(\beta_{P_\theta}^{\theta}( w_0 a_i, e)) \to \infty.\ee 
Note that $\beta_{P_\theta}^{\theta}(w_0 a_i, e) = p_{\theta}(\beta_{P}(w_0a_i, e))$ and $$\beta_{P}( w_0 a_i, e)=
\beta_{P}( w_0 a_iw_0^{-1}, e) =\i (\log a_i).$$
Since the sequences $g_i = \ga_i h_i m_i w_0 a_i$ and $h_im_i$ are bounded and $\ga_i^{-1} g_i= h_i m_i w_0 a_i $, we have 
$\|\mu(\ga_i^{-1}) - \log a_i\| = \|\mu(\ga_i) - \i(\log a_i)\|$ is uniformly bounded by Lemma \ref{lem.cptcartan} and the identity \eqref{mu}.
Therefore
$$\sup_i |\varphi(\mu_{\theta}(\ga_i) - (p_{\theta} \circ \i)(\log a_i))| <\infty .$$
It follows from the $\theta$-regularity of $\Gamma$ and the $(\Ga, \theta)$-properness of $\varphi$ that $\varphi(\mu_\theta(\ga_i)) \to \infty$ as $i \to \infty$, and hence $\varphi((p_{\theta} \circ \i)(\log a_i)) \to \infty$, implying \eqref{betaf}. Therefore, we have $\varphi(\beta_{\xi_i}^{\theta}(\ga_i^{-1}, e))\to \infty$. This finishes the proof.
\end{proof}

Recall the definition of a $\theta$-Anosov subgroup given in
the introduction. Anosov subgroups are word hyperbolic.
The notion of a $\theta$-conical set in \cite{Kapovich2017anosov} is equal to the one we use here for $\theta$-Anosov subgroups, by the Morse property of $\theta$-Anosov subgroups obtained in   \cite{Kapovich2017anosov}.

\begin{theorem} \label{anosov4} \cite[Theorem 1.1]{Kapovich2017anosov}
For a $\theta$-transverse subgroup $\Ga$, $\Ga$ is $\theta$-Anosov if and only if     
$\lat=\La^{\mathsf{con}}_{\theta}.$
\end{theorem}

\subsection*{Proof of Theorems \ref{tproper} and \ref{tproper2}}

 Suppose to the contrary that the $\Ga$-action on $\La_{\theta}^{(2)} \times \fa_{\theta}$ is not properly discontinuous. Then there exists a compact subset $Q \subset \La_{\theta}^{(2)} \times \fa_{\theta}$ such that $\ga_i Q \cap Q \neq \emptyset$ for an infinite sequence $\ga_i \in \Ga$. In particular, there exists a sequence $(\xi_i, \eta_i, u_i) \in Q$ such that $\ga_i(\xi_i, \eta_i, u_i) \in Q$ for all $i \ge 1$. By passing to a subsequence, we may assume that the sequences $(\xi_i, \eta_i, u_i)$ and $\ga_i(\xi_i, \eta_i, u_i)$ converge in $Q \subset \La_{\theta}^{(2)} \times \fa_{\theta}$. On the other hand, $$\ga_i(\xi_i, \eta_i, u_i) = (\ga_i \xi_i, \ga_i \eta_i, u_i + \beta_{\xi_i}^{\theta}(\ga_i^{-1}, e)) \quad \text{for all } i \ge 1$$ which cannot converge by Proposition \ref{mainte}, yielding a contradiction.
 Hence Theorem \ref{tproper} follows.
 
The first part of Theorem \ref{tproper2} follows from Proposition \ref{mainte} as well.
Now suppose that  $\Omega_{\varphi}$ is compact. Fix a sequence $s_i\to +\infty$ and let $\xi\in \La_\theta$. Choose any $\eta \in \La_{\i(\theta)}$ so that
 $(\xi, \eta) \in \La_{\theta}^{(2)}$. Then there exists a sequence $\ga_i \in \Ga$ such that the sequence $\ga_i(\xi, \eta, s_i) = (\ga_i \xi, \ga_i \eta, s_i+ \varphi(\beta_{\xi}^{\theta}(\ga_i^{-1}, e)))$
converges  by passing to a subsequence. Hence  the sequence $\ga_i(\xi, \eta)$ is convergent in $\La_{\theta}^{(2)}$ and $\varphi(\beta_{\xi}^{\theta}(\ga_i^{-1}, e)) \to -\infty$ as $i \to \infty$. By Lemma \ref{cor.convinpair}, the sequence $\ga_i^{-1}$ converges to $\xi$ or $\eta$ conically as $i \to \infty$. If $\ga_i^{-1} \to \eta$ conically, then as in the Case B of the proof of Proposition \ref{mainte}, we must have $\varphi(\beta_{\xi}^{\theta}(\ga_i^{-1}, e)) \to +\infty$, which is impossible. Therefore, $\ga_i^{-1} \to \xi$ conically as $i \to \infty$, and hence $\xi \in \La_{\theta}^{\sf con}$. Since $\xi$ is arbitrary, we have $\La_{\theta} = \La_{\theta}^{\sf con}$. 
By Theorem \ref{anosov4}, $\Ga$ is $\theta$-Anosov.
 
 Suppose that $\Ga$ is Anosov. 
 By \cite[Theorem 10.1]{CZZ_rel}, we have $\varphi > 0$ on $\L_{\theta} - \{0\}$.
 Hence it is a consequence of the H\"older reparametrization theorem for Anosov subgroups (\cite[Proposition 4.1]{BCLS_gafa}, \cite[Theorem 4.15]{CS_local}) that $\Omega_{\varphi}$ is compact (see also \cite[Theorem 3.5]{CS_local}). This finishes the proof.
 \qed

\section{Bowen-Margulis-Sullivan measures on $\Omega_\theta$ and $\Omega_\varphi$} \label{sec.BMS measures}

Let $\Ga < G$ be a non-elementary  $\theta$-transverse subgroup and  $\psi \in \fa_{\theta}^*$. As $\psi$ can be considered as a linear form on $\fa$ which is $p_\theta$-invariant and hence
$\psi\circ \i$ is a linear form on $\fa$ which is $p_{\i(\theta)}$-invariant, we have $ \psi\circ \i\in \fa_{\i(\theta)}^*$.
For a pair of a $(\Ga, \psi)$-conformal measure $\nu$ on $\La_{\theta}$ and  a $(\Ga, \psi \circ \i)$-conformal measure 
$\nu_{\i}$ on $\La_{\i(\theta)}$, we  define a Radon measure $d \tilde{\mathsf m}_{\nu, \nu_{\i}}$ on $\La_{\theta}^{(2)} \times \fa_{\theta}$ as follows: \be\label{bms} d\tilde{\mathsf m}_{\nu, \nu_{\i}} (\xi, \eta, u) = e^{\psi \left( \beta_{\xi}^{\theta}(e, g) + \i(\beta_{\eta}^{\i(\theta)}(e, g))\right)} d\nu(\xi) d \nu_{ \i}(\eta) du\ee  where $g \in G$ is chosen so that $(\xi, \eta) = (gP_{\theta}, g w_0 P_{\i(\theta)})$ and $du$ is the Lebesgue measure on $\fa_{\theta}$. This definition is independent of the choice of $g$ by Lemma \ref{welld} below.
The measure $d\tilde{\mathsf m}_{\nu, \nu_{\i}} $
is left $\Ga$-invariant and right $A_\theta$-invariant.
We denote by 
\be\label{bdef} {\mathsf m}_{\nu, \nu_{\i}} \ee  the $A_\theta$-invariant Borel measure on $\Omega_\theta$ induced by $\tilde{\mathsf m}_{\nu, \nu_{\i}} $, which we call the Bowen-Margulis-Sullivan measure associated to the pair $(\nu, \nu_{\i})$.

\begin{lem}\label{welld}
 If  $g, g'\in G$ satisfy  $(\xi, \eta)=(gP_{\theta}, g w_0 P_{\i(\theta)}) =(g'P_{\theta}, g' w_0 P_{\i(\theta)})$,
 then
  $$\beta_{\xi}^{\theta}(e, g) + \i(\beta_{\eta}^{\i(\theta)}(e, g))=\beta_{\xi}^{\theta}(e, g') + \i(\beta_{\eta}^{\i(\theta)}(e, g')).$$
\end{lem}
\begin{proof} The hypothesis on $g$ and $g'$ means that
$g'=gh$ for some $h \in L_\theta$. 

Since \begin{align*}
  &  \beta_{\xi}^{\theta}(e, g') + \i(\beta_{\eta}^{\i(\theta)}(e, g'))
\\&= (\beta_{\xi}^{\theta}(e, g) + \i(\beta_{\eta}^{\i(\theta)}(e, g)))
+ (\beta_{P_\theta}^{\theta}(e, h) + \i(\beta_{w_0P_{\i(\theta)}}^{\i(\theta)}(e, h)))
\end{align*}
it suffices to prove that 
$$\beta_{P_\theta}^{\theta}(e, h) + \i(\beta_{w_0P_{\i(\theta)}}^{\i(\theta)}(e, h))=0 .$$

Write $h=as$ where $a\in A_\theta$
and $s\in S_\theta$.
Since $p_\theta (\log (A\cap S_\theta) )=0$ and
$$\beta_{P_\theta}(e, s) +
\i(\beta_{w_0P_{\i(\theta)}}(e, s))\in  \log (A\cap S_\theta),$$
we have
$$ \beta_{P_\theta}^{\theta}(e, h) + \i(\beta_{w_0P_{\i(\theta)}}^{\i(\theta)}(e, h))=
 \beta_{P_\theta}^{\theta}(e, a) + \i(\beta_{w_0P_{\i(\theta)}}^{\i(\theta)}(e, a)).$$
On the other hand, by the definition of
the Busemann map, $\beta_{P}(e,a)=\log a$
and $\beta_{w_0P }(e, a)= \beta_{P}(e, w_0 aw_0^{-1})=  \op{Ad}_{w_0}(\log a) = -\i (\log a)$.
Hence
$$\beta_{P}( e,a) + \i(\beta_{w_0P }( e, a))= \log a -\i^2 (\log a) =0,$$
finishing the proof.
\end{proof}

For a $(\Ga, \theta)$-proper form $\varphi$,
consider the $\Ga$-equivariant projection
$\Lambda_\theta^{(2)}\times \fa_\theta \to \Lambda_\theta^{(2)}\times \br$
given by $(\xi, \eta, u)\to (\xi, \eta, \varphi(u))$.
By Theorem \ref{tproper2}, this induces an affine bundle with fiber $\ker \varphi$:
 \be\label{bun} \Omega_\theta\to \Omega_\varphi;\ee 
it is a standard fact that such a bundle is indeed a trivial vector bundle
and hence we have a homeomorphism
\be\label{bdef1} \Omega_\theta\simeq \Omega_\varphi \times \ker \varphi\simeq \Omega_\varphi \times \br^{\#\theta -1}.\ee 
We denote by the push-forward of the measure 
${\mathsf m}_{\nu, \nu_{\i}}$ on $\Omega_\varphi$ by ${\mathsf m}^{\varphi}_{\nu, \nu_{\i}}$ which is an $\br$-invariant Radon measure on $\Omega_\varphi$. Then  
\be\label{bdef2}{\mathsf m}_{\nu, \nu_{\i}} = {\mathsf m}^\varphi_{\nu, \nu_{\i}} \otimes \text{Leb}_{\ker \varphi} .\ee

By the following proposition, the measures $\mathsf{m}_{\nu, \nu_{\i}}$ and $\mathsf{m}_{\nu, \nu_{\i}}^{\varphi}$ are non-zero.

\begin{proposition} \label{prop.nontrivialmeasure}
    Let $\Ga < G$ be a discrete subgroup and let $\la$ and $\la_{\i}$ be probability measures on $\F_{\theta}$ and $\F_{\i(\theta)}$ respectively. Suppose one of the following:
    \begin{enumerate}
        \item $\Ga$ is Zariski dense and $\la$ is $\Ga$-quasi-invariant.
        \item $\Ga$ is non-elementary $\theta$-transverse, $\la$ is $\Ga$-quasi-invariant, and $\la$ and $\la_{\i}$ are supported on $\La_{\theta}$ and $\La_{\i(\theta)}$ respectively.
    \end{enumerate}
    Then \be \label{eqn.suppongenpos}
    (\la \times \la_{\i})(\F_{\theta}^{(2)}) > 0.
    \ee

\end{proposition}

\begin{proof} 
Suppose $(\la \times \la_{\i})(\F_{\theta}^{(2)}) = 0$. Then by Fubini's theorem,
\be\label{eta}
\la \left( \{ \xi \in \F_{\theta} : (\xi, \eta) \in \F_{\theta}^{(2)} \} \right) = 0 \quad \text{for } \la_{\i}\text{-a.e. } \eta \in \F_{\i(\theta)}.
\ee 
We now deduce a contradiction in each case. 
In the case of (1),  let $\eta\in \F_{\i(\theta)}$. Since $\supp \nu \subset \F_{\theta}$ must be Zariski dense in this case by Lemma \ref{lq} and $\{ \xi \in \F_{\theta} : (\xi, \eta) \in \F_{\theta}^{(2)} \}$ is a non-empty Zariski open subset, we have $\la \left( \{ \xi \in \F_{\theta} : (\xi, \eta) \in \F_{\theta}^{(2)} \} \right) > 0$, contradicting  \eqref{eta}.

In the case (2),  let $\eta \in \La_{\i(\theta)}$. Since $\Ga$ is $\theta$-transverse,  there exists $\xi_0 \in \La_{\theta}$  such that
$\La_\theta-\{\xi_0\}\subset \{ \xi \in \F_{\theta} : (\xi, \eta) \in \F_{\theta}^{(2)} \} $.
Hence it suffices to note that $\lambda(\La_\theta-\{\xi_0\})>0$. If not,
$\lambda$ is supported on $\{\xi_0\}$, which must be fixed by $\Ga$ due to the quasi-invariance of $\lambda$.

Since the $\Ga$-action on $\La_{\theta}$ is minimal (Proposition \ref{prop.transisconv}), $\La_{\theta}=\{\xi_0\}$,
contradicting the non-elementary hypothesis on $\Ga$. 
\end{proof}

 \section{Conservativity and ergodicity of the $\fa_\theta$-action}\label{ec}
 In this section, we expand the dichotomies in Theorem \ref{ggg}
to a criterion on conservativity and ergodicity of $\fa_\theta$-action on the 
quotient space $\Omega_\theta=\Gamma\ba \La_\theta^{(2)}\times \fa_\theta$,
or equivalently a criterion on 
conservativity and ergodicity of $\br$-action on the 
quotient space $\Omega_\varphi=\Gamma\ba \La_\theta^{(2)}\times \br$,
when $\Ga$ is a non-elementary $\theta$-transverse subgroup and $\varphi$ is a $(\Ga, \theta)$-proper linear form. First of all, this makes sense thanks to Theorems \ref{tproper} and \ref{tproper2}.

We recall the notion of complete conservativity and ergodicity. Let $H$ be 
a locally compact unimodular group. We denote by $dh$ the Haar measure on $H$.
Consider the dynamical system $(H, \Omega, \lambda)$
where $\Omega$ is a separable, locally compact and $\sigma$-compact topological space on which $H$ acts continuously and $\lambda$ is a Radon measure which is quasi-invariant by $H$.  A Borel subset $B\subset \Omega$ is called wandering
if $\int_H {\mathbbm 1}_B(h. w)dh <\infty$ for $\mu$-almost all $w\in B$. The Hopf decomposition theorem says that $\Omega$ can be written as the disjoint union $\Omega_C\cup \Omega_D$ of $H$-invariant subsets where $\Omega_D$ is a countable union of wandering subsets which is maximal in the sense that $\Omega_C$
does not contain any wandering subset of positive measure. If $\lambda(\Omega_D)=0$, the system is called completely conservative. 
If $\lambda(\Omega_C)=0$, the system is called completely dissipative. 
The dynamical system
$(H, \Omega, \lambda)$ is ergodic if any $H$-invariant $\lambda$-measurable subset  is either null or co-null.
An ergodic system $(H, \Omega, \lambda)$ is either completely conservative or completely dissipative. 
If $(H, \Omega, \lambda)$ is ergodic, 
$H$ is countable and $\lambda$ is atomless, then it is completely conservative \cite[Theorem 14]{Kaimanovich2010hopf}.
The following is standard \cite[Lemma 6.1]{lee2022dichotomy}:

\begin{lem} \label{ccc} Suppose that $\lambda$ is $H$-invariant.
Then  $(H, \Omega, \lambda)$ is completely conservative if and only if for $\lambda$-a.e. $x\in \Omega$,
there exists a compact subset $B_x\subset \Omega$ such that $\int_{h\in H} {\mathbbm 1}_{B_x}(h.x)\; dh =\infty$. 
\end{lem}

The following theorem implies Theorem \ref{c} in the introduction. For a non-elementary $\theta$-transverse subgroup $\Ga < G$ and $\psi \in \fa_{\theta}^*$, we denote by $$\mathcal{M}_{\psi}^{\theta}\subset \mathsf{M}_\psi^{\theta}$$
the space of all $(\Ga,\psi)$-conformal measures {\it  supported on} $\La_\theta$. 

\begin{theorem}\label{ceq}
   Let $\Ga<G$ be a non-elementary
$\ts$-transverse subgroup. Let $\psi \in\fa_\theta^*$ be $(\Ga, \theta)$-proper such that $\mathcal{M}_{\psi}^{\theta} \neq \emptyset$. Then the following are equivalent to each other.
\begin{enumerate}

\item $\sum_{\ga\in \Ga}e^{-\psi(\mu_\theta(\ga))}=\infty$
(resp. $\sum_{\ga\in \Ga}e^{-\psi(\mu_\theta(\ga))}<\infty$);

  \item For any $\nu \in \mathcal{M}_{\psi}^{\theta}$, $\nu(\La_{\theta}^{\mathsf{con}}) > 0$
 (resp. $\nu(\La_{\theta}^{\mathsf{con}}) = 0$);
    \item For any $\nu \in \mathcal{M}_{\psi}^{\theta}$, $\nu(\La_{\theta}^{\mathsf{con}}) = 1$
     (resp. $\nu(\La_{\theta}^{\mathsf{con}}) = 0$);

\item For any $(\nu, \nu_{\i}) \in \mathcal{M}_{\psi}^{\theta} \times \mathcal{M}_{\psi \circ \i}^{\i(\theta)}$, the $\Ga$-action on $(\La_\theta^{(2)}, \nu\times \nu_{\i})$ is completely conservative and ergodic (resp. completely dissipative and non-ergodic);

\item  For any $(\nu, \nu_{\i}) \in \mathcal{M}_{\psi}^{\theta} \times \mathcal{M}_{\psi \circ \i}^{\i(\theta)}$, the $\fa_\theta$-action on $(\Omega_\theta, {\mathsf m}_{\nu,\nu_{\i}})$  is completely conservative and ergodic (resp. completely dissipative and non-ergodic);

\item For any $(\nu, \nu_{\i}) \in \mathcal{M}_{\psi}^{\theta} \times \mathcal{M}_{\psi \circ \i}^{\i(\theta)}$ and any $(\Ga, \theta)$-proper $\varphi \in \fa_{\theta}^*$, the $\br$-action on $(\Omega_\varphi, {\mathsf m}^\varphi_{\nu,\nu_{\i}})$  is completely conservative and ergodic (resp. completely dissipative and non-ergodic).

 \end{enumerate}
 \end{theorem}

 In the proof of Theorem \ref{ceq}, we will use the following observation.

 \begin{lemma} \label{lem.ergodicatomless}
     Let $\Ga < G$ be a non-elementary $\theta$-transverse subgroup. Let $\la$ and $\la_{\i}$ be $\Ga$-quasi-invariant probability measures on $\La_{\theta}$ and $\La_{\i(\theta)}$ respectively. If the $\Ga$-action on $(\La_{\theta}^{(2)}, \la \times \la_{\i})$ is ergodic, then $\la \times \la_{\i}$ has no atom in $\La_{\theta}^{(2)}$.
 \end{lemma}

 \begin{proof} By Proposition \ref{prop.nontrivialmeasure},
    we have $(\la \times \la_{\i}) (\La_{\theta}^{(2)}) > 0$. 
    Suppose that $\la \times \la_{\i}$ has an atom, say $(\xi_0, \eta_0) \in \La_{\theta}^{(2)}$.
     By the ergodicity hypothesis, $\la \times \la_{\i}$ is supported on a single $\Ga$-orbit $\Ga(\xi_0, \eta_0) \subset \La_{\theta}^{(2)}$. Since $\la(\xi_0)>0$ and $\la_{\i}(\eta_0)>0$, we have $$(\Ga \xi_0 \times \Ga \eta_0) \cap \La_{\theta}^{(2)} \subset \Ga(\xi_0, \eta_0).$$ 
Since $\Ga$ is $\theta$-antipodal, $$\Ga \xi_0 \subset \Ga_{\eta_0} \xi_0 \cup \{\eta_0'\}$$ where $\Ga_{\eta_0}$ is the stabilizer of $\eta_0$ in $\Ga$ and $\eta_0'$ is the image of $\eta_0$ under the $\Ga$-equivariant homeomorphism $\La_{\i(\theta)} \to \La_{\theta}$ obtained in Lemma \ref{bij}.
In addition, the $\Ga$-equivariance of $\La_{\i(\theta)} \to \La_{\theta}$ implies that $\Ga_{\eta_0} = \Ga_{\eta_0'}$ and hence 
\be \label{eqn.noneltcontra}
\Ga \xi_0 \subset \Ga_{\eta_0'} \xi_0 \cup \{\eta_0'\}.
\ee
Since the $\Ga$-action on $\La_{\theta}$ is a convergence group action (Proposition \ref{prop.transisconv}), $\La_{\theta}$ is perfect and equal to the set of all accumulation points of $\Ga\xi_0$.
On the other hand,
$\Ga_{\eta_0'}$ is an elementary subgroup and hence  $\Ga_{\eta_0'} \xi_0$ has at most two accumulation points in $\La_{\theta}$ (\cite{Tukia_convergence}, \cite{Bowditch1999convergence}). Therefore, we obtain a contradiction.
 \end{proof}

\subsection*{Proof of Theorem \ref{ceq}}
Note that $\fa_{\theta}^*$ can be regarded as a subspace of $\fa_{\theta \cup \i(\theta)}^*$ and  that $\psi \in \fa_{\theta}^*$ is $(\Ga, \theta)$-proper if and only if $\psi \circ \i$ is $(\Ga, \i(\theta))$-proper. By Lemma \ref{bij}, we have $\Ga$-equivariant homeomorphisms $\La_{\theta}\to \La_{\theta\cup \i(\theta)} \to \La_{\i(\theta)}$ and hence we can push-forward measures in $\mathcal{M}_{\psi}^{\theta}$ to $\mathcal{M}_{\psi \circ \i}^{\i(\theta)}$. In particular, $\mathcal{M}_{\psi \circ \i}^{\i(\theta)} \neq \emptyset$. Note that since $\Ga$ is non-elementary $\theta$-transverse, the equivalence $(1)\Leftrightarrow (2) \Leftrightarrow (3)$ follows from Lemma \ref{lem.npsiismpsi} and Theorem \ref{ggg}.

\medskip
{\noindent \bf The divergent case.}
We will show  $(3) \Rightarrow (5) \Rightarrow (4)$, $(3) \Rightarrow (6) \Rightarrow (4)$, and $(4) \Rightarrow (1)$, which will then finish the proof of this case.

In order to show $(3) \Rightarrow (5)$, assume (3). Consider a pair $(\nu, \nu_{\i}) \in \mathcal{M}_{\psi}^{\theta} \times \mathcal{M}_{\psi \circ \i}^{\i(\theta)}$. Then for $\nu$-a.e. $\xi \in \La_{\theta}$, $\xi$ belongs to $ \La_{\theta}^{\mathsf{con}}$, that is, there exist $g \in G$ and sequences $\ga_i \in \Ga$, $m_i\in M_\theta$ and $a_i \in A^+$ such that $\xi = gP_{\theta}$, the sequence $\ga_i gm_i a_i$ is bounded, and the sequence $\ga_i$ is infinite. By the $\theta$-regularity of $\Ga$ and Lemma \ref{lem.cptcartan}, we have $\min_{\alpha\in \theta} \alpha(\log a_i) \to \infty$ as $i \to \infty$. For any $\eta \in \La_{\i(\theta)}$ such that $(\xi, \eta) \in \La_{\theta}^{(2)}$ and any $u \in \fa_{\theta}$, there exists $n \in N_{\theta}$ and $a \in A_{\theta}$ such that $gna S_{\theta} \in G/S_{\theta}$ represents $(\xi, \eta, u) \in \La_{\theta}^{(2)} \times \fa_{\theta} \subset G/S_{\theta}$. Since $a_i \in A^+$, the sequence
$$
\ga_i g na m_i a_i = (\ga_i g m_i a_i) ( a_i^{-1} m_i^{-1} n m_i a_i a) \quad \text{is bounded}.
$$
 This implies that writing $u_i = p_{\theta}(\log a_i) \in \fa_{\theta}^+$,
$$
\ga_i (\xi, \eta, u + u_i) \in \La_{\theta}^{(2)} \times \fa_{\theta} \quad \text{is precompact.}
$$
Moreover, since $\alpha(\log a_i) \to \infty$ for all $\alpha \in \theta$, we also have $u_i \to \infty$ in $\fa_{\theta}$. Projecting to $\Omega_{\theta}$, this implies that there exists a compact subset $Q \subset \Omega_{\theta}$ so that
$$
\int_{v \in \fa_{\theta}} {\mathbbm 1}_{Q}(\Ga (\xi, \eta, u + v)) dv = \infty.
$$
Since this holds for $\nu$-a.e. $\xi \in \La_{\theta}$, any $\eta \in \La_{\i(\theta)}$ with $(\xi, \eta) \in \La_{\theta}^{(2)}$, and any $u \in \fa_{\theta}$, the $\fa_{\theta}$-action on $(\Omega_\theta, {\mathsf m}_{\nu,\nu_{\i}})$  is completely conservative by Lemma \ref{ccc}. By \cite[Lemma 8.7]{KOW}, the complete conservativity implies the ergodicity, showing (5).

To see the implication $(5) \Rightarrow (4)$, note that for $(\nu, \nu_{\i}) \in  \mathcal{M}_{\psi}^{\theta} \times \mathcal{M}_{\psi \circ \i}^{\i(\theta)}$, the ergodicity of the $\fa_{\theta}$-action on $(\Omega_{\theta}, \mathsf{m}_{\nu, \nu_{\i}})$ is equivalent to the ergodicity of the $\Ga$-action on $(\La_{\theta}^{(2)}, \nu \times \nu_{\i})$ by the definition of $\mathsf{m}_{\nu, \nu_{\i}}$. Hence, if (5) holds, then $\nu \times \nu_{\i}$ has no atom in $\La_{\theta}^{(2)}$  by Lemma \ref{lem.ergodicatomless}. Consider the measure $\la$ on $\La_{\theta}^{(2)}$ defined by 
$$
d \la(\xi, \eta) := e^{\psi \left( \beta_{\xi}^{\theta}(e, g) + \i(\beta_{\eta}^{\i(\theta)}(e, g))\right)} d\nu(\xi) d \nu_{ \i}(\eta)
$$
where $g \in G$ is chosen so that $(\xi, \eta) = (gP_{\theta}, g w_0 P_{\i(\theta)})$ as in \eqref{bms}. Then $\la$ is $\Ga$-invaraint, $\Ga$-ergodic, and atomless.
Since $\Ga$ is countable, this implies that the $\Ga$-action on $(\La_{\theta}^{(2)}, \la)$ is completely conservative \cite[Theorem 14]{Kaimanovich2010hopf}. Therefore, (4) follows. This establishes $(3) \Rightarrow (5) \Rightarrow (4)$. The implications $(3) \Rightarrow (6) \Rightarrow (4)$ can be proved by a similar argument.

To show the implication $(4)
\Rightarrow (1)$,  fixing a pair $(\nu, \nu_{\i}) \in \mathcal{M}_{\psi}^{\theta} \times \mathcal{M}_{\psi \circ \i}^{\i(\theta)}$, we will show that the complete conservativity of the $\Ga$-action on $(\La_{\theta}^{(2)}, \nu \times \nu_{\i})$ implies (1).
Since $(\Ga, \La_{\theta}^{(2)}, \nu \times \nu_{\i})$  is completely conservative, it follows from Lemma \ref{ccc} that for $\nu \times \nu_{\i}$-a.e. $(\xi, \eta) \in \La_{\theta}^{(2)}$, there exists a compact subset $B_{(\xi, \eta)} \subset \La_{\theta}^{(2)}$ and a sequence $\ga_i \in \Ga$ such that $\ga_i (\xi, \eta) \in B_{(\xi, \eta)}$ for all $i$. In particular, after passing to a subsequence, we have that the sequence $\ga_i (\xi, \eta)$ is convergent in $\La_{\theta}^{(2)}$. By Lemma \ref{cor.convinpair}, we have $\ga_i^{-1} \to \xi$ or $\ga_i^{-1} \to \eta$ conically. In particular,
either $\xi\in \La_{\theta}^{\mathsf{con}}$ or $\eta \in \La_{\i(\theta)}^{\mathsf{con}}$, and therefore $$\max\{\nu(\La_{\theta}^{\mathsf{con}}), \nu_{\i}(\La_{\i(\theta)}^{\mathsf{con}})\} > 0.$$
In either case, it follows from Theorem \ref{ggg} that $$\sum_{\ga \in \Ga} e^{-\psi(\mu_{\theta}(\ga))} = \sum_{\ga \in \Ga} e^{-(\psi \circ \i)(\mu_{\i(\theta)}(\ga))} = \infty.$$ Now (1) follows.

\medskip
{\noindent \bf The convergent case.}
From the divergent case, we have the following equivalences for the convergent case:
$$
\begin{tikzcd}[column sep=small, row sep=small]
 & & & (4) \arrow[ld, Rightarrow] \\
(1) \arrow[r, Leftrightarrow] & (2) \arrow[r, Leftrightarrow] & (3)  & (5) \arrow[l, Rightarrow] \\
& & & (6) \arrow[ul, Rightarrow] 
\end{tikzcd}
$$

We first observe $(4) \Rightarrow (5)$ and $(4) \Rightarrow (6)$. As mentioned in the proof of the divergent case, the ergodicity in (4), (5), and (6) are all equivalent to each other. Moreover, if $B \subset \La_{\theta}^{(2)}$ is a wandering set for the $\Ga$-action on $(\La_{\theta}^{(2)}, \nu \times \nu_{\i})$ where $(\nu, \nu_{\i}) \in \mathcal{M}_{\psi}^{\theta} \times \mathcal{M}_{\psi \circ \i}^{\i(\theta)}$, then for any non-empty compact subset $V \subset \fa_{\theta}$, the set $\Ga (B \times V) \subset \Omega_{\theta}$ is a wandering set for the $\fa_{\theta}$-action on $(\Omega_{\theta}, \mathsf{m}_{\nu, \nu_{\i}})$. Since $\fa_{\theta}$ is $\sigma$-compact, this implies that if $(\La_{\theta}^{(2)}, \nu \times \nu_{\i})$ is a countable union of wandering subsets, then so is $(\Omega_{\theta}, \mathsf{m}_{\nu, \nu_{\i}})$, up to measure zero. Therefore, the complete dissipativity in (4) implies the one in (5), and hence $(4) \Rightarrow (5)$ follows. The implication $(4) \Rightarrow (6)$ can be shown similarly.

We finish the proof by showing $(1) \Rightarrow (4)$. Assume (1) and fix $(\nu, \nu_{\i}) \in \mathcal{M}_{\psi}^{\theta} \times \mathcal{M}_{\psi \circ \i}^{\i(\theta)}$. We first show that the $\Ga$-action on $(\La_{\theta}^{(2)}, \nu \times \nu_{\i})$ is completely dissipative. We write the Hopf decomposition $\La_{\theta}^{(2)} = \Omega_C \cup \Omega_D$ and suppose to the contrary that $(\nu \times \nu_{\i})(\Omega_C) > 0$. By applying Lemma \ref{ccc} to the restriction $(\nu \times \nu_{\i})|_{\Omega_C}$, we deduce that there exists a Borel subset $\Omega \subset \La_{\theta}^{(2)}$ with $(\nu \times \nu_{\i})(\Omega) > 0$ such that for any $(\xi, \eta) \in \Omega$, there exist a compact subset $B_{(\xi, \eta)} \subset \Omega$ and a sequence $\ga_i \in \Ga$ such that $\ga_i (\xi, \eta) \in B_{(\xi, \eta)}$ for all $i$. Hence after passing to a subsequence, the sequence $\ga_i (\xi, \eta)$ is convergent in $\Omega \subset \La_{\theta}^{(2)}$, and therefore it follows from Lemma \ref{cor.convinpair} that $\ga_i^{-1} \to \xi$ or $\ga_i^{-1} \to \eta$ conically. Since $(\nu \times \nu_{\i})(\Omega) > 0$, it implies $$\max\{\nu(\La_{\theta}^{\mathsf{con}}), \nu_{\i}(\La_{\i(\theta)}^{\mathsf{con}})\} > 0.$$
In either case, it follows from Theorem \ref{ggg} that $$\sum_{\ga \in \Ga} e^{-\psi(\mu_{\theta}(\ga))} = \sum_{\ga \in \Ga} e^{-(\psi \circ \i)(\mu_{\i(\theta)}(\ga))} = \infty,$$ which contradicts (1). Therefore, $(\nu \times \nu_{\i})(\Omega_C) = 0$ and hence the $\Ga$-action on $(\La_{\theta}^{(2)}, \nu \times \nu_{\i})$ is completely dissipative.

Now it remains to show that the $\Ga$-action on $(\La_{\theta}^{(2)}, \nu \times \nu_{\i})$ is non-ergodic. Suppose not. 
Then the $\Ga$-action on $(\La_{\theta}^{(2)}, \nu \times \nu_{\i})$ is ergodic, and hence $\nu \times \nu_{\i}$ has no atom in $\La_{\theta}^{(2)}$ by Lemma \ref{lem.ergodicatomless}. As before, since $\Ga$ is countable, this must imply that the $\Ga$-action on $(\La_{\theta}^{(2)}, \nu \times \nu_{\i})$ is completely conservative \cite[Theorem 14]{Kaimanovich2010hopf}. This is a contradiction, and (4) follows.
\qed

\subsection*{Proof of Theorem \ref{main}}

Theorem \ref{tgg} implies Theorem \ref{main}(1). Theorem \ref{main}(2) follows from Theorem \ref{ggg} and the following corollary.
\qed

\begin{cor} \label{cor.singleconfmeas} Let $\Ga$ be a Zariski dense $\theta$-transverse subgroup.
  If $\psi \in \fa_{\theta}^*$ is $(\Ga, \theta)$-proper with $\mathsf{M}_{\psi}^{\theta} \neq \emptyset$ and $\sum_{\ga \in \Ga} e^{-\psi(\mu_{\theta}(\ga))} = \infty$, then $\# {\mathsf M}^\theta_\psi=1$.
\end{cor}
\begin{proof}
 By Theorem \ref{tgg} and the hypothesis on $\psi$, we have $\delta_\psi=1$.
By Proposition \ref{measure},
there exists a $(\Ga,\psi)$-conformal measure on $\F_{\t \cup \i(\t)}$, and is supported on $\La_{\theta \cup \i(\theta)}$. Moreover it is unique by \cite[Theorem 1.4]{Canary2023}. 
It then follows from Lemma \ref{bij} that there exists a unique $(\Ga,\psi)$-conformal measure on $\F_{\theta}$ as well.
\end{proof}

\section{Lebesgue measures of conical sets and disjoint dimensions}\label{app}
In this section, we discuss some of consequences of Theorem \ref{ggg}.
\subsection*{Lebesgue measure of conical sets}
\begin{theorem} \label{Ahlfors}
    If $\Ga < G$ is a Zariski dense $\ts$-transverse subgroup, then $$\La_{\theta} = \F_{\theta} \quad \mbox{or} \quad
    \Leb_\theta(\La_{\theta}^{\mathsf{con}}) = 0.$$ Moreover, in the former case, $\theta$ is the simple root of a rank one factor of $G$.
\end{theorem}

We need the following proposition for the second claim of the above theorem.
\begin{prop}\label{one}
    Suppose that $\Ga$ is Zariski dense and $\ts$-antipodal and that $\La_\theta=\F_\theta$. Then $\theta$ consists of the simple root of a rank one factor of $G$.
\end{prop}
\begin{proof}
 We write $G$ as the almost direct product of simple real algebraic groups $G=G_1\cdots G_m$. Let $n$ be an index such that $\theta$ contains a simple root of $G_n$. Denoting by $\pi_n:G\to G_n$ the canonical projection,  $\pi_n(P_\theta)$ is a proper parabolic subgroup  of $G_n$ and the limit set of $\overline{\pi_n(\Ga)}$ in $G_n/\pi_n(P_\theta)$ is equal to all of $G_n/\pi_n(P_\theta)$, as the limit set of a Zariski dense subgroup is the unique minimal set (Lemma \ref{lq}). Suppose that the rank of $G_n$ is at least $2$. Fix $kP_{\theta\cup \i(\theta)}\in \La_{\theta\cup \i(\theta)}$ for some $k\in K$. Let $w$ be a Weyl element given by Lemma \ref{weyl} below  such that
$ w\notin w_0 N_{\theta}^+ P_{\theta}\cup P_{\theta}$. 
Noting that
$w_0 N_{\theta \cup \i(\theta)}^+ P_{\theta\cup \i(\theta)} M_\theta \subset w_0 P_{\theta}^+ P_{\theta}= w_0N_\theta^+P_\theta$, we have
\be\label{mtt} w\notin  w_0 N_{\theta \cup \i(\theta)}^+ P_{\theta\cup \i(\theta)}M_\theta\cup  P_{\theta\cup \i(\theta)}M_\theta.\ee 
Note again that both $\La_{\theta}$ and $\La_{\Pi}$ are unique $\Ga$-minimal subsets of $\F_{\theta}$ and $\F$, and hence the canonical projection $\F \to \F_{\theta}$ maps $\La_{\Pi}$ onto $\La_{\theta}$.
Since $\F=K/M$ and  $kwM_\theta\in \F_\theta=K/M_\theta=\La_\theta$, we may choose $m\in M_\theta$ such that $kwm P \in\La_{\Pi}$, and hence
$kwm P_{\theta\cup \i(\theta)}\in\La_{\theta\cup \i(\theta)}$.
Then by \eqref{mtt},  $$wm \notin w_0 N_{\theta \cup \i(\theta)}^+ P_{\theta\cup \i(\theta)}\cup P_{\theta\cup \i(\theta)}.$$
The condition that $wm\notin P_{\theta\cup \i(\theta)}$ implies that
$kw m P_{\theta\cup\i(\theta)}
 \cap kP_{\theta\cup \i (\theta)}=\emptyset$. 
 Also, by Corollary \ref{cor.genweyl},
 the condition that $wm \notin w_0 N_{\theta \cup \i(\theta)}^+ P_{\theta\cup \i(\theta)}$  implies 
that $(kw m P_{\theta\cup\i(\theta)}, kP_{\theta\cup \i(\theta)})
\notin G . (P_{\theta\cup\i(\theta)}, w_0P_{\theta\cup\i(\theta)}),$ that is,
$kwm P_{\theta\cup\i(\theta)}$ is not in general position with $P_{\theta\cup\i(\theta)}$.
This yields a contradiction to the $\theta\cup \i(\theta)$-antipodality of $\Ga$. Therefore 
for any $n$ such that $\theta$ contains a simple root of $G_n$,
the rank of $G_n$ must be one.
If there are $n\ne n'$ with this property, the map $\ga \to (\pi_n(\ga), \pi_{n'}(\ga))$
must be a Zariski dense subgroup of $G_nG_{n'}$ with full limit set $G_n/\pi_n(P_\theta) \times G_{n'}/\pi_{n'}(P_\theta)$.
However this yields a contradiction to the $\theta$-antipodal property, because the product of two rank one geometric boundaries does not have the antipodal property.
Therefore $\theta$ must be a singleton, proving the claim. 
\end{proof}

We now prove the following lemma which was used in the above proof.
\begin{lem}\label{weyl} If $G$ has a connected normal subgroup $G_n$ of rank at least $2$ and  $\theta\subset \Pi$ contains a simple root of $G_n$, then
we can find a representative of a Weyl element $w \in N_K(A)$ such that
 $w\notin  w_0 N_{\theta}^+ P_{\theta}\cup P_\theta$. 
\end{lem}

\begin{proof} By replacing $\theta$ with the intersection of $\theta$ and the set of simple roots of $G_n$, we may assume without loss of generality that $G=G_n$.
Since the rank of $G$ is at least $2$,
we can find a representative $w \in N_K(A)$ of a Weyl element such that $\op{Ad}_w(\fa_{\theta}^+)$ is  equal to neither $\fa_{\theta}^+$ nor $-\fa_{\i(\theta)}^+$. If $w$ were contained in $P_{\theta}\cap K=M_\theta$, $w$ would commute with $\fa_\theta$ and hence $\op{Ad}_w(\fa_{\theta}^+)=\fa_{\theta}^+$. Therefore $w\notin P_\theta$. On the other hand, if $w \in w_0 N_{\theta}^+P_{\theta}$, then  $w_0^{-1}w  \in M_{\theta}$ by  Corollary \ref{cor.genweyl}, and hence $\op{Ad}_w(\fa_{\theta}^+)  = \op{Ad}_{w_0}(\fa_{\theta}^+) = - \fa_{\i(\theta)}^+$, which contradicts our choice of $w$. Hence $w\notin w_0N_{\theta}^+P_{\theta}.$
\end{proof}

\subsection*{Proof of Theorem \ref{Ahlfors}}
Note that $\op{Leb}_\theta$ is a $(\Ga, 2\rho \circ p_{\theta})$-conformal measure  where $\rho$ is the half sum of all positive roots of $(\fg, \fa^+)$ \cite[Lemma 6.3]{Quint2002Mesures}.   If $\La_{\theta} \neq \F_{\theta}$,  $\Leb_\theta(\lat^{\mathsf{con}})\le \Leb_\theta(\La_{\theta}) < 1$ as $\F_{\theta} - \La_{\theta}$ is a non-empty open subset. Therefore $\Leb_\theta(\La_{\theta}^{\mathsf{con}}) = 0$ by Theorem \ref{ggg}. The second claim follows from Proposition \ref{one} above.
\qed

\subsection*{Disjoint dimensions and entropy drop}

Recall from the introduction that
$$\cal D_\Ga^\theta=\{\psi\in \fa_\theta^*: (\Ga, \theta)\text{-proper}, \delta_\psi=1, \cal P_\psi(1)=\infty\} .$$

\begin{lem}\label{dsame}
    For a Zariski dense $\theta$-transverse $\Ga$, we have
$$\cal D_\Ga^\theta=\left\{\psi\in \fa_\theta^*: \text{$(\Ga, \theta)$-proper},  \text{$\exists$ a $(\Ga,\psi)$-conformal measure},
\cal P_\psi(1)=\infty \right\}.$$
\end{lem}
\begin{proof} The inclusion $\subset$ follows from Proposition \ref{measure}.  If there exists a $(\Ga,\psi)$-conformal measure on $\F_{\theta}$ for $(\Ga, \theta)$-proper
$\psi$, then
$\delta_\psi \le 1$ by Theorem \ref{tgg0}. 
If $\delta_\psi<1$, $\cal P_\psi(1)<\infty$. Hence this implies the inclusion $\supset$.
\end{proof}

Note that any subgroup of a $\ts$-transverse subgroup of $G$ is again a $\ts$-transverse subgroup.
\begin{theorem} [Disjoint dimensions]\label{disss}
    Let $\Ga<G$ be a non-elementary $\t$-transverse subgroup. For any  subgroup $\Ga_0<\Ga$ with $\La_\theta(\Ga_0)\ne \La_\theta(\Ga)$, we have
     $$\mathcal D^\theta_{\Ga}\cap \mathcal D^\theta_{\Ga_0}=\emptyset.$$
\end{theorem}
\begin{proof}
Let $\psi\in \mathcal D_\Ga^{\theta}$.
  By  Proposition \ref{measure}, there exists a $(\Ga, \psi)$-conformal measure $\nu$ on $\La_\theta(\Ga)$. By  Theorem \ref{thm.divthensupp},
$\nu(\La_\theta^{\mathsf{con}}(\Ga))=1$. 

Since $\La_\theta(\Ga_0)\ne \La_\theta(\Ga)$,
$\La_{\theta}(\Ga) - \La_{\theta}(\Ga_0)$ is a non-empty open subset of $\La_{\theta}(\Ga)$. Hence, it follows from the $\Ga$-minimality on $\La_{\theta}(\Ga)$ and the compactness of $\La_{\theta}(\Ga)$ that $\La_{\theta}(\Ga)$ is covered by translates of $\La_{\theta}(\Ga) - \La_{\theta}(\Ga_0)$ under finitely many elements of $\Ga$. Since $\nu$ is $\Ga$-quasi-invariant, this implies $\nu(\La_{\theta}(\Ga) - \La_{\theta}(\Ga_0)) > 0$, and hence,  $\nu(\La_\theta^{\mathsf{con}}(\Ga_0))< 1$. Moreover, by the $\theta$-antipodality of $\Ga$, it also follows from $\nu(\La_{\theta}(\Ga) - \La_{\theta}(\Ga_0)) > 0$ that  $\nu \in \mathsf{N}_{\Ga_0, \psi}^{\theta}$ in Theorem \ref{thm.divthensupp}.
Again by Theorem \ref{thm.divthensupp}, $\sum_{\ga\in \Ga_0}e^{-\psi(\mu_\theta(\ga))}<\infty$. Hence $\psi\notin \mathcal D_{\Ga_0}^{\theta}$, finishing the proof.
\end{proof}

This turns out to be equivalent to the entropy drop phenomenon which is proved by Canary-Zhang-Zimmer \cite[Theorem 4.1]{Canary2023} for $\theta=\i(\theta)$: 

\begin{cor}[Entropy drop] \label{entropy} Let $\Ga<G$ be a non-elementary   $\ts$-transverse subgroup. Let $\Ga_0<\Ga$ be a  subgroup such that $\La_\theta(\Ga_0)\ne \La_\theta(\Ga)$.
    If $\psi\in \fa_\theta^*$ with $\delta_\psi(\Ga)<\infty$
    and  $\sum_{\ga\in \Ga_0}e^{-\delta_{\psi}(\Ga_0) \psi(\mu_\theta(\ga))}=\infty$, then 
    $$\delta_\psi (\Ga_0)<\delta_\psi(\Ga).$$
\end{cor}
\begin{proof} Suppose that $\delta_{\psi}(\Ga) < \infty$; this implies that $\psi$ is $(\Ga, \theta)$-proper. Let $\Ga_0<\Ga$ be a  subgroup
such that
$\sum_{\ga\in \Ga_0}e^{-\delta_{\psi}(\Ga_0) \psi(\mu_\theta(\ga))}=\infty$ and  $\delta_\psi (\Ga_0)=\delta_\psi(\Ga)$.
    If we set $\phi=\delta_{\psi}(\Ga) \cdot \psi=\delta_{\psi}(\Ga_0) \cdot \psi$,
    then $\delta_{\phi}(\Ga)=\delta_{\phi}(\Ga_0)=1$.
Since $\infty= \sum_{\ga\in \Ga_0}e^{-\phi(\mu_\theta(\ga))} \le  \sum_{\ga\in \Ga}e^{-\phi(\mu_\theta(\ga))} $, we have $\phi\in \cal D_\Ga^{\theta}\cap \cal D_{\Ga_0}^{\theta}$. By Theorem \ref{disss}, this implies that 
$\La_\theta(\Ga_0) = \La_\theta(\Ga)$, proving the corollary.
\end{proof}

\section{Conformal measures for $\theta$-Anosov subgroups}\label{tangentA} \label{sec.Anosov}
Note that $\Ga$ is $\theta$-Anosov if and only if $\Ga$ is $\t\cup \i(\t)$-Anosov by \eqref{mu}.

\begin{proposition}[{\cite{Guichard2012anosov}, \cite[Theorem 1.1]{Kapovich2017anosov}}]
\label{prop.interior} If $\Ga$ is $\theta$-Anosov, then
\begin{enumerate}
    \item $\Ga$ is $\theta$-regular;
    \item $\La_{\theta} = \La_{\theta}^{\sf con}$;
    \item $\L_\theta-\{0\}\subset \inte\fa_\theta^+$;
    \item  $\theta$-antipodal.
    
    \end{enumerate}
\end{proposition}
In particular, a $\theta$-Anosov subgroup is $\ts$-transverse.

 Sambarino \cite[Theorem A]{sambarino2022report} showed that if $\Ga$ is $\theta$-Anosov, then
the set $\{\psi\in \fa_\theta^*:\delta_\psi=1\}$ is analytic and  is equal to the boundary 
of a strictly convex subset $\{\psi \in \fa_{\theta}^* : 0<\delta_\psi < 1\}$. 
By the duality lemma (\cite[Section 4]{Quint_indicator}, \cite[Lemma 4.8]{samb_hyper}), we then deduce the following property of  the $\theta$-growth indicator: 

\begin{theorem}\label{ver}
 If $\Ga$ is Zariski dense $\theta$-Anosov, then  $\psi_\Ga^\theta$ is strictly concave, differentiable on $\inte \L_{\theta}$, and vertically tangent.
\end{theorem}
  The vertical tangency means that
if $\psi^\theta_\Ga(u)=\psi(u)$ for some $(\Ga, \theta)$-critical form $\psi$ and $u\ne 0$, then 
$u\in \inte \L_\theta$. 
Recall 
$$ \cal T_\Ga^\theta :=\{\psi\in \fa_\theta^*:
\text{$\psi$ is $(\Ga, \theta)$-critical}\}.$$

\begin{corollary} \label{cor.equiv in 1.15}
Let $\Ga < G$ be a Zariski dense $\theta$-Anosov subgroup. For any subgroup $\Ga_0 < \Ga$,
    $$
    \cal T_{\Ga}^{\theta} \cap \cal T_{\Ga_0}^{\theta} = \emptyset \quad \Longleftrightarrow \quad \psi_{\Ga_0}^{\theta} < \psi_{\Ga}^{\theta} \text{ on } \inte \L_{\theta}(\Ga).
    $$
\end{corollary}

\begin{proof}
    Suppose that $\psi \in \cal T_{\Ga}^{\theta} \cap \cal T_{\Ga_0}^{\theta}$. Then there exists $u \in \L_{\theta}(\Ga_0)$ such that 
    $$
    \psi_{\Ga_0}^{\theta}(u) = \psi(u).
    $$
    Since $\psi_{\Ga_0}^{\theta} \le \psi_{\Ga}^{\theta} \le \psi$, it follows that $\psi$ is tangent to $\psi_{\Ga}^{\theta}$ at $u$ as well. By the vertical tangency of $\psi_{\Ga}^{\theta}$ (Theorem \ref{ver}), $u \in \inte \L_{\theta}(\Ga)$.
    Therefore, the implication $(\Leftarrow)$ follows.

    Conversely, suppose that $\psi_{\Ga_0}^{\theta}(u) = \psi_{\Ga}^{\theta}(u)$ for some $u \in \inte \L_{\theta}(\Ga)$. Then by the concavity of $\psi_{\Ga}^{\theta}$ (Theorem \ref{ver}), there exists $\psi \in \cal T_{\Ga}^{\theta}$ such that $\psi(u) = \psi_{\Ga}^{\theta}(u)$. Since $\psi_{\Ga_0}^{\theta} \le \psi_{\Ga}^{\theta} \le \psi$ and $\psi_{\Ga_0}^{\theta}(u) = \psi_{\Ga}^{\theta}(u)$, we have $\psi \in \cal T_{\Ga}^{\theta} \cap \cal T_{\Ga_0}^{\theta}$. This shows the implication $(\Rightarrow)$.
\end{proof}

\begin{lem} \label{cri} 
If $\Ga$ is Zariski dense $\theta$-Anosov, then
  $$\cal T_\Ga^\theta=  
\{\psi\in \fa_\theta^*: (\Ga, \theta)\text{-proper},\delta_\psi=1\} = \cal D_\Ga^\theta.$$
\end{lem}
\begin{proof} The second identity is proved in \cite[Section 5.9]{sambarino2022report}.
It suffices to prove the inclusion $\subset$ in the first equality due to Corollary \ref{oneone}.
   Suppose that $\psi \in \fa_{\theta}^*$ is tangent to $ \psi_{\Ga}^{\theta}$.
  Since  $\psi_\Ga^{\theta}$ is vertically tangent (Theorem \ref{ver}), $\psi>\psi_\Ga^{\theta}$ on $\partial \L_\theta$. It follows that
   $\psi>0$ on $\L_\theta$. Hence by the second claim in Corollary \ref{oneone}, $\delta_\psi=1$.
\end{proof}

\begin{lem}\label{fp}
    If $\Ga$ is a non-elementary $\theta$-Anosov subgroup and there exists a $(\Ga, \psi)$-conformal measure on $\F_\theta$ for $\psi\in \fa_\theta^*$, then $\psi$ is $(\Ga, \theta)$-proper.
\end{lem}
\begin{proof}
     If $\sum_{\ga \in \Ga} e^{-\psi(\mu_{\theta}(\ga))} < \infty$, then it implies that 
$\#\{\ga\in \Ga:\psi(\mu_\theta(\ga))<T\}$ is finite for any $T>0$.
Therefore  $\psi$ is $(\Ga, \theta)$-proper. 
If $\sum_{\ga \in \Ga} e^{-\psi(\mu_{\theta}(\ga))} =\infty$, then $\nu(\La_{\theta})  =1$ by Theorem \ref{thm.divthensupp}.
This implies that
$\limsup \frac{1}{T} \log \#\{\ga\in \Ga:\psi(\mu_\theta(\ga))<T\}<\infty$ by \cite[Theorem A]{sambarino2022report}. Therefore, $\psi$ is $(\Ga, \theta)$-proper in either case.
\end{proof}
\subsection*{Proof of Theorem \ref{g2}}
Let $\Ga$ be  Zariski dense $\theta$-Anosov. Note that a $\theta$-Anosov group is $\ts$-transverse. Hence (1) follows from Theorem \ref{tgg} since $\psi$
is $(\Ga, \theta)$-proper by Lemma \ref{fp}.

Since $\La_{\theta} = \La_{\theta}^{\sf con}$ (Proposition \ref{prop.interior}), $(a) \Leftrightarrow (b)$ in (2) follows from Theorem \ref{ggg}. The equivalence $(b) \Leftrightarrow (c)$
follows from Lemma \ref{cri} and Sambarino's parametrization of
the space of all conformal measures on $\La_\theta$ as
$\{\delta_\psi=1\}$ \cite[Theorem A]{sambarino2022report}, together with (1) shown above. 
For (3), let $\psi$ be a $(\Ga, \theta)$-critical form. 
By Lemma \ref{cri} and Proposition \ref{measure}, there exists a $(\Ga, \psi)$-conformal measure $\nu_{\psi}$ on $\La_\theta$, which is the unique $(\Ga, \psi)$-conformal measure on $\La_\theta$ by \cite[Theorem A]{sambarino2022report} (see also Corollary \ref{cor.singleconfmeas}).  Since  $\sum_{\ga \in \Ga} e^{-\psi(\mu_{\theta}(\ga))} = \infty$,
by Theorem \ref{ggg}, any
 $(\Ga, \psi)$-conformal measure on $\F_\theta$ is supported on $\La_\theta$. Moreover, by Theorem \ref{ceq}, the $\fa_{\theta}$-action on $(\Omega_{\theta}, \mathsf{m}_{\nu_{\psi}, \nu_{\psi \circ \i}})$ is completely conservative and ergodic.
 This finishes the proof.
 \qed

\subsection*{Proof of Corollary \ref{critical}}
Since a $\theta$-Anosov subgroup is $\ts$-transverse and $\La_{\theta} = \La_{\theta}^{\sf con}$ (Theorem \ref{anosov4}),
we deduce from Theorem \ref{Ahlfors} that either $\La_{\theta} = \F_{\theta}$ or $\Leb_{\theta}(\La_{\theta}) = 0$. In the former case, $\theta$
is the simple root of a rank one factor $G_0$
of $G$ with $\F_\theta=\La_\theta$ by Proposition \ref{one}, the projection of $\Ga$ to $G_0$ is a convex cocompact subgroup with full limit set, and hence a cocompact lattice of $G_0$.
\qed

\subsection*{Proof of Corollary \ref{coeq}}

Consider the map $\cal T_{\Ga}^{\theta} \to \{ u \in \inte \L_{\theta} : \|u\| = 1\}$ given by $\psi \mapsto u_{\psi}$, where $u_{\psi}$ satisfies $\psi(u_{\psi}) = \psi_{\Ga}^{\theta}(u_{\psi})$. By Theorem \ref{ver}, $\psi_{\Ga}^{\theta}$ is strictly concave and vertically tangent, and hence such a map is well-defined and also surjective. Moreover, since $\psi_{\Ga}^{\theta}$ is differentiable on $\inte \L_{\theta}$ (Theorem \ref{ver}), this map is injective as well, and therefore bijective. This gives the one-to-one correspondence between (1) and (2).

By Lemma \ref{cri} and \cite[Theorem A]{sambarino2022report}, for each $\psi \in \cal T_{\Ga}^{\theta}$, there exists a unique $(\Ga, \psi)$-conformal measure $\nu_{\psi}$ supported on $\La_{\theta}$, and vice versa. Hence the map $\psi \mapsto \nu_{\psi}$ is the one-to-one correspondence between (1) and (3).

Finally, by Theorem \ref{g2}, the sets (3) and (4) are in fact identical, which finishes the proof.
\qed

\subsection*{Proof of Corollary \ref{dis00}}
By Theorem \ref{disss} and Lemma \ref{cri},
 it remains to prove the second part. Since $\Ga_0<\Ga$, we have $\psi^\theta_{\Ga_0}\le \psi^\theta_\Ga$.
Suppose that $\psi^\theta_{\Ga_0}(u)= \psi^\theta_\Ga(u)$ for some $u$
in the interior of $\L_\theta(\Ga)$. Then there exists a tangent form $\psi$ to $\psi_\Ga^{\theta}$ at $u$ by Corollary \ref{cor.tangentinterior}. Since $\psi_{\Ga_0}^{\theta} \le \psi_{\Ga}^{\theta}$ and $\psi_{\Ga_0}^{\theta}(u) = \psi_{\Ga}^{\theta}(u)$, $\psi$ is also tangent to $\psi_{\Ga_0}^\theta$ at $u$. Hence $\psi\in\cal T_\Ga^\theta\cap \cal T_{\Ga_0}^\theta$, contradicting the first part.
\qed


\begin{thebibliography}{10}

\bibitem{Aaronson1984rational}
J.~Aaronson and D.~Sullivan.
\newblock Rational ergodicity of geodesic flows.
\newblock {\em Ergodic Theory Dynam. Systems}, 4(2):165--178, 1984.

\bibitem{Ahlfors1964finitely}
L.~Ahlfors.
\newblock Finitely generated {K}leinian groups.
\newblock {\em Amer. J. Math.}, 86:413--429, 1964.

\bibitem{Benoist1997proprietes}
Y.~Benoist.
\newblock Propri\'{e}t\'{e}s asymptotiques des groupes lin\'{e}aires.
\newblock {\em Geom. Funct. Anal.}, 7(1):1--47, 1997.

\bibitem{Bowditch1999convergence}
B.~Bowditch.
\newblock Convergence groups and configuration spaces.
\newblock In {\em Geometric group theory down under ({C}anberra, 1996)}, pages
  23--54. de Gruyter, Berlin, 1999.

\bibitem{BCLS_gafa}
M.~Bridgeman, R.~Canary, F.~Labourie, and A.~Sambarino.
\newblock The pressure metric for {A}nosov representations.
\newblock {\em Geom. Funct. Anal.}, 25(4):1089--1179, 2015.

\bibitem{burger2021hopf}
M.~Burger, O.~Landesberg, M.~Lee, and H.~Oh.
\newblock The {H}opf--{T}suji--{S}ullivan dichotomy in higher rank and
  applications to {A}nosov subgroups.
\newblock {\em J. Mod. Dyn.}, 19:301--330, 2023.

\bibitem{Canary2022cusped}
R.~Canary, T.~Zhang, and A.~Zimmer.
\newblock Cusped {H}itchin representations and {A}nosov representations of
  geometrically finite {F}uchsian groups.
\newblock {\em Adv. Math.}, 404(part B):Paper No. 108439, 67, 2022.

\bibitem{canary2022entropy}
R.~Canary, T.~Zhang, and A.~Zimmer.
\newblock Entropy rigidity for cusped hitchin representations.
\newblock {\em Preprint, arXiv:2201.04859}, 2022.

\bibitem{Canary2023}
R.~Canary, T.~Zhang, and A.~Zimmer.
\newblock Patterson-{S}ullivan measures for transverse subgroups.
\newblock {\em J. Mod. Dyn.}, 20:319--377, 2024.

\bibitem{CZZ_rel}
R.~Canary, T.~Zhang, and A.~Zimmer.
\newblock Patterson-{S}ullivan measures for relatively {A}nosov groups.
\newblock {\em Math. Ann.}, 392(2):2309--2363, 2025.

\bibitem{CS_local}
M.~Chow and P.~Sarkar.
\newblock Local mixing of one-parameter diagonal flows on {A}nosov homogeneous
  spaces.
\newblock {\em Int. Math. Res. Not. IMRN}, (18):15834--15895, 2023.

\bibitem{Corlette1999limit}
K.~Corlette and A.~Iozzi.
\newblock Limit sets of discrete groups of isometries of exotic hyperbolic
  spaces.
\newblock {\em Trans. Amer. Math. Soc.}, 351(4):1507--1530, 1999.

\bibitem{Das2016tukia}
T.~Das, D.~Simmons, and M.~Urba\'{n}ski.
\newblock Tukia's isomorphism theorem in {$\rm{CAT}(-1)$} spaces.
\newblock {\em Ann. Acad. Sci. Fenn. Math.}, 41(2):659--680, 2016.

\bibitem{Dey}
S.~Dey and M.~Kapovich.
\newblock Patterson-{S}ullivan theory for {A}nosov subgroups.
\newblock {\em Trans. Amer. Math. Soc.}, 375(12):8687--8737, 2022.

\bibitem{edwards2021unique}
S.~Edwards, M.~Lee, and H.~Oh.
\newblock Uniqueness of conformal measures and local mixing for {A}nosov
  groups.
\newblock {\em Michigan Math. J.}, 72:243--259, 2022.

\bibitem{Quint_inprep}
R.~Grigorchuk and J.-F. Quint.
\newblock Directional counting for automatic languages.
\newblock {\em in preparation}.

\bibitem{Gueritaud2017anosov}
F.~Gu\'{e}ritaud, O.~Guichard, F.~Kassel, and A.~Wienhard.
\newblock Anosov representations and proper actions.
\newblock {\em Geom. Topol.}, 21(1):485--584, 2017.

\bibitem{Guichard2012anosov}
O.~Guichard and A.~Wienhard.
\newblock Anosov representations: domains of discontinuity and applications.
\newblock {\em Invent. Math.}, 190(2):357--438, 2012.

\bibitem{hopf1971}
E.~Hopf.
\newblock Ergodic theory and the geodesic flow on surfaces of constant negative
  curvature.
\newblock {\em Bull. Amer. Math. Soc.}, 77:863--877, 1971.

\bibitem{Kaimanovich2010hopf}
V.~Kaimanovich.
\newblock Hopf decomposition and horospheric limit sets.
\newblock {\em Ann. Acad. Sci. Fenn. Math.}, 35(2):335--350, 2010.

\bibitem{Kapovich2018discrete}
M.~Kapovich and B.~Leeb.
\newblock Discrete isometry groups of symmetric spaces.
\newblock In {\em Handbook of group actions. {V}ol. {IV}}, volume~41 of {\em
  Adv. Lect. Math. (ALM)}, pages 191--290. Int. Press, Somerville, MA, 2018.

\bibitem{Kapovich2017anosov}
M.~Kapovich, B.~Leeb, and J.~Porti.
\newblock Anosov subgroups: dynamical and geometric characterizations.
\newblock {\em Eur. J. Math.}, 3(4):808--898, 2017.

\bibitem{Kapovich2018morse}
M.~Kapovich, B.~Leeb, and J.~Porti.
\newblock A {M}orse lemma for quasigeodesics in symmetric spaces and
  {E}uclidean buildings.
\newblock {\em Geom. Topol.}, 22(7):3827--3923, 2018.

\bibitem{kmoHD}
D.~M. Kim, Y.~Minsky, and H.~Oh.
\newblock Hausdorff dimension of directional limit sets for self-joinings of
  hyperbolic manifolds.
\newblock {\em J. Mod. Dyn.}, 19:433--453, 2023.

\bibitem{kim2021tent}
D.~M. Kim, Y.~Minsky, and H.~Oh.
\newblock Tent property of the growth indicator functions and applications.
\newblock {\em Geom. Dedicata}, 218(1):Paper No.14, 18, 2024.

\bibitem{KOW}
D.~M. Kim, H.~Oh, and Y.~Wang.
\newblock Ergodic dichotomy for subspace flows in higher rank.
\newblock {\em Commun. Am. Math. Soc.}, 5:1--47, 2025.

\bibitem{Kochen1964note}
S.~Kochen and C.~Stone.
\newblock A note on the {B}orel-{C}antelli lemma.
\newblock {\em Illinois J. Math.}, 8:248--251, 1964.

\bibitem{Labourie2006anosov}
F.~Labourie.
\newblock Anosov flows, surface groups and curves in projective space.
\newblock {\em Invent. Math.}, 165(1):51--114, 2006.

\bibitem{lee2020invariant}
M.~Lee and H.~Oh.
\newblock Invariant {M}easures for {H}orospherical {A}ctions and {A}nosov
  {G}roups.
\newblock {\em Int. Math. Res. Not. IMRN}, (19):16226--16295, 2023.

\bibitem{lee2022dichotomy}
M.~Lee and H.~Oh.
\newblock Dichotomy and {M}easures on {L}imit {S}ets of {A}nosov {G}roups.
\newblock {\em Int. Math. Res. Not. IMRN}, (7):5658--5688, 2024.

\bibitem{oh}
H.~Oh.
\newblock Uniform pointwise bounds for matrix coefficients of unitary
  representations and applications to {K}azhdan constants.
\newblock {\em Duke Math. J. 113 (2002), no. 1, 133--192.}

\bibitem{Patterson1976limit}
S.~Patterson.
\newblock The limit set of a {F}uchsian group.
\newblock {\em Acta Math.}, 136(3-4):241--273, 1976.

\bibitem{Potrie2017eigenvalues}
R.~Potrie and A.~Sambarino.
\newblock Eigenvalues and entropy of a {H}itchin representation.
\newblock {\em Invent. Math.}, 209(3):885--925, 2017.

\bibitem{pozzetti2019anosov}
B.~Pozzetti, A.~Sambarino, and A.~Wienhard.
\newblock Anosov representations with {L}ipschitz limit set.
\newblock {\em Geom. Topol.}, 27(8):3303--3360, 2023.

\bibitem{Quint2002divergence}
J.-F. Quint.
\newblock Divergence exponentielle des sous-groupes discrets en rang
  sup\'{e}rieur.
\newblock {\em Comment. Math. Helv.}, 77(3):563--608, 2002.

\bibitem{Quint2002Mesures}
J.-F. Quint.
\newblock Mesures de {P}atterson-{S}ullivan en rang sup\'{e}rieur.
\newblock {\em Geom. Funct. Anal.}, 12(4):776--809, 2002.

\bibitem{Quint_indicator}
J.-F. Quint.
\newblock L'indicateur de croissance des groupes de {S}chottky.
\newblock {\em Ergodic Theory Dynam. Systems}, 23(1):249--272, 2003.

\bibitem{quint_ka}
J.-F. Quint.
\newblock Propri\'{e}t\'{e} de {K}azhdan et sous-groupes discrets de covolume
  infini.
\newblock In {\em Travaux math\'{e}matiques. {F}asc. {XIV}}, volume~14 of {\em
  Trav. Math.}, pages 143--151. Univ. Luxemb., Luxembourg, 2003.

\bibitem{Roblin2003ergodicite}
T.~Roblin.
\newblock Ergodicit\'{e} et \'{e}quidistribution en courbure n\'{e}gative.
\newblock {\em M\'{e}m. Soc. Math. Fr. (N.S.)}, (95):vi+96, 2003.

\bibitem{samb_hyper}
A.~Sambarino.
\newblock Hyperconvex representations and exponential growth.
\newblock {\em Ergodic Theory Dynam. Systems}, 34(3):986--1010, 2014.

\bibitem{sambarino2022report}
A.~Sambarino.
\newblock A report on an ergodic dichotomy.
\newblock {\em Ergodic Theory Dynam. Systems}, 44(1):236--289, 2024.

\bibitem{Sullivan1979density}
D.~Sullivan.
\newblock The density at infinity of a discrete group of hyperbolic motions.
\newblock {\em Inst. Hautes \'{E}tudes Sci. Publ. Math.}, (50):171--202, 1979.

\bibitem{Tsuji1959potential}
M.~Tsuji.
\newblock {\em Potential theory in modern function theory}.
\newblock Maruzen Co. Ltd., Tokyo, 1959.

\bibitem{Tukia1985isomorphisms}
P.~Tukia.
\newblock On isomorphisms of geometrically finite {M}\"{o}bius groups.
\newblock {\em Inst. Hautes \'{E}tudes Sci. Publ. Math.}, (61):171--214, 1985.

\bibitem{Tukia_convergence}
P.~Tukia.
\newblock Convergence groups and {G}romov's metric hyperbolic spaces.
\newblock {\em New Zealand J. Math.}, 23(2):157--187, 1994.

\bibitem{Yaman2004topological}
A.~Yaman.
\newblock A topological characterisation of relatively hyperbolic groups.
\newblock {\em J. Reine Angew. Math.}, 566:41--89, 2004.

\end{thebibliography}
\end{document}